\documentclass[twoside, 11pt]{amsart}

\usepackage[T1]{fontenc}
\usepackage[lf]{Baskervaldx} 
\usepackage[cal=boondoxo]{mathalfa} 

\usepackage[utf8]{inputenc}

\usepackage{latexsym}
\usepackage{amsthm,thmtools}
\usepackage{amscd}
\usepackage{amsmath}
\usepackage{amssymb}
\usepackage{graphicx}
\usepackage{mathrsfs}
\usepackage{mathtools}
\usepackage{pict2e}
\usepackage{stackrel}
\usepackage{faktor}
\usepackage{wasysym}
\usepackage{enumitem}
\usepackage{adjustbox}
\usepackage{needspace}

\usepackage[all]{xy}
\usepackage[dvipsnames]{xcolor}
\usepackage[colorlinks,final,hyperindex]{hyperref}
\usepackage[noabbrev,capitalize]{cleveref}
\usepackage{tikz}
\usepackage{tikz-cd}
\usepackage{tkz-graph}
\usetikzlibrary{decorations.pathmorphing,decorations.markings,arrows,calc,shapes.geometric,arrows.meta,positioning}
\tikzset{math3d/.style=
	{x= {(-0.353cm,-0.353cm)}, z={(0cm,1cm)},y={(1cm,0cm)}}}
\tikzset{JLL3d/.style=
	{x= {(0.4cm,-0.2cm)}, z={(0cm,1cm)},y={(-1cm,0cm)}}}

\definecolor{Chocolat}{rgb}{0.36, 0.2, 0.09}
\definecolor{BleuTresFonce}{rgb}{0.215, 0.215, 0.36}
\definecolor{BleuMinuit}{RGB}{0, 51, 102}

\hypersetup{citecolor=BleuMinuit, linkcolor=Chocolat, filecolor=black, urlcolor=BleuMinuit}

\usepackage{todonotes}

\definecolor{armygreen}{rgb}{0.29, 0.33, 0.13}

\allowdisplaybreaks

\usepackage{perpage} 
\MakePerPage{footnote}

\newcommand{\defi}[1]{\emph{#1}}

\newcommand{\Nn}{\mathbb{N}} 
\newcommand{\Zz}{\mathbb{Z}} 
\newcommand{\Qq}{\mathbb{Q}} 
\newcommand{\Rr}{\mathbb{R}} 
\newcommand{\Cc}{\mathbb{C}}

\newcommand{\Tt}{\mathbb{T}}
\newcommand{\Ll}{\mathbb{L}}

\newcommand{\id}{\mathrm{id}}

\newcommand{\tx}[1]{\mathrm{#1}}

\def\G_#1{\mathfrak{#1}} 
\def\t_#1{\widetilde{#1}}

\DeclareMathOperator{\Sym}{Sym}
\DeclareMathOperator{\Map}{Map}
\DeclareMathOperator{\Hom}{Hom}
\DeclareMathOperator{\End}{End}
\DeclareMathOperator{\Mod}{Mod}
\newcommand{\iHom}{\underline{\mathrm{Hom}}}

\newcommand{\Op}{\mathbf{Op}}
\newcommand{\coOp}{\mathbf{coOp}}
\newcommand{\chk}{\mathrm{Mod}_k}

\newcommand{\cdga}{\mathbf{Alg}_{\mathbf{Com}}}
\newcommand{\dga}{\mathbf{Alg}_{\mathbf{Ass}}}
\newcommand{\T}{\G_T}
\newcommand{\Tc}{\G_T^{\tx{co}}}
\newcommand{\symseq}{\mathbf{Seq}_{\Sigma}}

\newcommand{\antishriek}{\text{\raisebox{\depth}{\textexclamdown}}}
\newcommand{\art}{\tx{Art}}
\newcommand{\dccomaug}{\_A_{\mathbf{Com}}^{\mathrm{aug}}}

\newcommand{\dcP}{\_A_{\_P}}
\renewcommand{\small}{\mathbf{Art}_{\dccomaug}}
\newcommand{\smallp}{\mathbf{Art}_\_P}

\newtheorem{Lem}{Lemma}[section]
\newtheorem{Th}[Lem]{Theorem}
\newtheorem*{Th*}{Theorem}
\newtheorem{Cor}[Lem]{Corollary}
\newtheorem{Prop}[Lem]{Proposition}

\newtheorem{Conv}[Lem]{Convention}

\theoremstyle{definition}
\newtheorem{Def}[Lem]{Definition}
\newtheorem{RQ}[Lem]{\sc Remark}
\newtheorem{Ex}[Lem]{\sc Example}
\newtheorem{Que}[Lem]{\sc Question}
\newtheorem*{Que*}{\sc Question}
\newtheorem{Refo}[Lem]{\sc Reformulation}
\newtheorem{Cons}[Lem]{\sc Construction}
\newtheorem{War}[Lem]{\sc Warning}

\author{Ricardo Campos and Albin Grataloup}
\title{Operadic Deformation Theory}
\date{\today}
\address{Ricardo Campos, Institut de Mathématiques de Toulouse, UMR5219, Université de Toulouse, CNRS, UPS, F-31062 Toulouse Cedex 9, France}
\email{\href{mailto:ricardo.campos@math.univ-toulouse.fr}{ricardo.campos@math.univ-toulouse.fr}}

\address{Albin Grataloup, IMAG, Univ. Montpellier, CNRS, Montpellier, France}

\email{\href{mailto:albin.grataloup@umontpellier.fr}{albin.grataloup@umontpellier.fr}}

\thanks{
	\indent$\,\! $  The first author was supported by ANR-20-CE40-0016 HighAGT. The second author received funding from the European Research Council (ERC) under the European Union’s Horizon 2020 research and innovation programme (grant agreement No 768679)}


\begin{document}
	
	\begin{abstract}
		This is a survey on recent progress in algebraic deformation theory and the application of algebraic operads to its study. We review the classical homotopical tools in the theory of algebraic operads, namely Koszul duality. We give concrete examples of applications of such tools to various flavours of problems related to deformations of algebraic structures. We also study formal moduli problems and related notions from the operadic point of view. 
	\end{abstract}
	
	\maketitle
	
	\makeatletter
	\def\@tocline#1#2#3#4#5#6#7{\relax
		\ifnum #1>\c@tocdepth 
		\else
		\par \addpenalty\@secpenalty\addvspace{#2}%
		\begingroup \hyphenpenalty\@M
		\@ifempty{#4}{%
			\@tempdima\csname r@tocindent\number#1\endcsname\relax
		}{%
			\@tempdima#4\relax
		}%
		\parindent\z@ \leftskip#3\relax \advance\leftskip\@tempdima\relax
		\rightskip\@pnumwidth plus4em \parfillskip-\@pnumwidth
		#5\leavevmode\hskip-\@tempdima
		\ifcase #1
		\or\or \hskip 1em \or \hskip 2em \else \hskip 3em \fi%
		#6\nobreak\relax
		\hfill\hbox to\@pnumwidth{\@tocpagenum{#7}}\par
		\nobreak
		\endgroup
		\fi}
	\makeatother
	
\vspace{-1cm}
	
	\setcounter{tocdepth}{2}
	\tableofcontents

		\section*{Introduction}
	
	Deformation theory is classically understood to be the study of variations of an object $X$, to infinitesimally nearby objects $X_\epsilon$. 
	One of the most salient examples, dating back to the 50's, are the works of Kodaira--Spencer \cite{KS58,KS60} and Fr\"ohlicher--Nijenhuis \cite{FN57} on the deformation theory of complex manifolds and algebraic varieties where it is shown that ``first-order'' deformations of an algebraic variety are related to the tangent cohomology.
	
	The approach to deformation theory by Grothendieck, Mumford, Schlessinger (see \cite{Ar76}) consists of associating to a deformation problem, a ``deformation functor''
	\begin{align*}
		F\colon \tx{Parametrizations} &\to \tx{Sets}\\
		\text{Infinitesimal }\epsilon &\mapsto \{X_\epsilon\}\\
	\end{align*}
	
	such that when $\epsilon = 0$ we get back the original object $X_0 = X$. To make some sense of the association above, we can replace the source by the category of local Artinian commutative algebras (or their spectra). The prototypical example are the algebras $k[\epsilon]/(\epsilon^n)$.  From the point of view of algebraic geometry, these represent infinitesimal thickenings of a point. 
	
	If we want to capture the essence of a deformation problem, we should have a way to encode when and how two deformations can be equivalent, which suggests that we should take groupoids instead of sets to be the target of $F$. 

	It turns out that in characteristic zero one can often associate to a deformation problem a differential graded (dg) Lie algebra $\mathfrak g$ encoding the problem in the sense that the associated deformation functor takes the form
	\begin{equation}\label{eq:intro}
		R \mapsto \text{Maurer--Cartan elements of }\mathfrak g \otimes \mathfrak m_R,
	\end{equation}

where $\mathfrak m_R$ is the unique maximal ideal of $R$. These are degree $1$ solutions of the Maurer--Cartan equation $d\mu + \frac{1}{2}[\mu,\mu] =0$ and in fact there is a natural \emph{gauge group} acting on them such that the functor naturally lands in groupoids.

	To exemplify, let us consider a very classical algebraic situation:
	\begin{Ex}\label{Ex:intro}
		
		Let $A$ be an associative algebra. A first-order deformation of $A$ is given by a map $\mu\colon A\otimes A \to A$, such that the formula, $a\cdot b = ab + \mu_1(a,b)\epsilon$, defines an $\epsilon$-linear associative product on $A[\epsilon]/(\epsilon^2)$.
	The requirement that the product be associative is equivalent to
		\[
		\mu_1(a,b)c + \mu_1(ab,c) = a\mu_1(b,c) + \mu_1(a,bc), \quad \forall a,b,c \in A
		\]
More generally, a higher order deformation along $k[\epsilon]/(\epsilon^{n+1})$ is an associative product on $A[\epsilon]/(\epsilon^{n+1})$, which amounts to the choice of $\mu_1,\dots,\mu_n$ satisfying generalisations of the equation above. 

		The dg Lie algebra encoding this deformation problem is the Hochschild Lie algebra given in degree $n$ by $\mathfrak g^n = \Hom(A^{\otimes n+1},A)$ (with appropriate differential and bracket). A deformation of order $n$, $\mu= \epsilon\mu_1 + \dots +\epsilon^n\mu_n$, is precisely a Maurer--Cartan element of $g \otimes \epsilon k[\epsilon]/(\epsilon^{n+1})$.
		Two noteworthy takeaways are that  the Maurer--Cartan equation only involves dg Lie algebraic data and that the dg Lie algebra is the same regardless of the type of deformation.
	\end{Ex}

Examples such as the previous one led Deligne, Drinfeld and Quillen \cite{DrinfeldtoSchechtman}  to formulate the following informal slogan:
\begin{center}
	\noindent\fbox{%
		\parbox{9cm}{%
			\begin{center}
				Any deformation problem over a field $k$ of characteristic zero is
				controlled by a dg Lie $k$-algebra.
			\end{center}
		}%
	}
\end{center}
	\vspace{.2cm}
This slogan was fruitfully used in works of Kontsevich \cite{kontsevichdefquant}, Hinich \cite{Hi01}, Goldman--Millson \cite{GM88}, Manetti \cite{Ma02}. This leads us to the central questions that we wish to address:

	\begin{Que*}\label{Que:intro}\phantom{jump line}
		
		\begin{enumerate}
			\item Given a deformation problem, how to obtain the dg Lie algebra encoding it?
			\item Which deformations are ``equivalent''?
			\item What is the structure of the moduli space of deformations up to equivalence? 
		\end{enumerate}
	\end{Que*}

	The notion of equivalence tends to come naturally associated to the problem. For instance, take two deformations $\widetilde{X}$ and $\widetilde{X}'$ of a complex manifold or more generally a $k$-scheme, $X$. Then they are called equivalent if there is an isomorphism (or more generally a weak equivalence) that reduces to the identity on $X$ (up to homotopy). In the case of associative algebras, two deformations are equivalent if there is an isomorphism intertwining $\mu$ and $\mu'$.

The main purpose of this text is to present how the theory of algebraic operads gives us very efficient tools to address these questions.
Let us consider again the illustrative example of deforming associative algebras. An object of interest is the set (ideally a ``space'') of  associative algebra structures on a vector space, or more generally a cochain complex, $A$. These are precisely all possible structures of left modules on $A$ (with respect to a monoidal product $\circ$) over the operad $\mathbf{Ass}$  of associative algebras
$\mathbf{Ass} \circ A \to A$.

Crucially, the symmetric monoidal category of cochain complexes is closed which allows us to see $A$ as a representation of the operad $\mathbf{Ass}$: There is an endomorphisms operad of $A$, $\mathbf{End}_A$, such that the set of associative algebra structures on $A$ is

$$\{\text{Associative algebra structures on }A\} \cong \Hom_{\mathrm{Operads}}(\mathbf{Ass}, \mathbf{End}_A).$$

It follows that to understand deformations of a particular algebraic structure on $A$, we can equivalently study deformations of the corresponding map of operads $f\colon \mathbf{Ass} \to \mathbf{End}_A$. 

A homotopical perspective gives us some insight in how to progress further.  Indeed, using a model structure on operads, this Hom-set can be naturally upgraded to a mapping space. The deformations we care about live in an infinitesimal neighbourhood of the point $f$ in the mapping space. 
Here, Koszul duality plays a key role in constructing for us a simple cofibrant replacement of $\mathbf{Ass}$, which is the operad governing $A_\infty$-algebras (algebras only associative up to homotopy). 

Furthermore, this procedure expresses the mapping space naturally as the Maurer--Cartan space of a dg Lie algebra $\mathfrak g$, seen as the \emph{deformation complex of }$f$. In the present case $\mathfrak g$ is exactly the Hochschild dg Lie algebra of Example \ref{Ex:intro}.\\

In fact, the slogan itself that dg Lie algebras control deformation problems should be interpreted in an appropriate homotopical sense. On one hand, quasi-isomorphic dg Lie algebras encode the same deformation problem. On the other hand, rather than a set or a groupoid, the target of a deformation functor should be a space, concretely an $\infty$-groupoid.
A formalisation of a deformation functor was only achieved in the 2010's in the works of Lurie \cite{Lu11} which he called a \emph{formal moduli problem}, which is an $\infty$-categorical.

A formalisation of the slogan was by Lurie \cite{Lu11} and Pridham \cite{Pr10} in the form  an equivalence between the $\infty$-categories of formal moduli problems and dg Lie algebras. One of the equivalences is, informally speaking, given by the assignment \eqref{eq:intro}. 
In the other direction, the equivalence is given by the tangent complex functor $\Tt$ whose dg Lie algebra structure can be quite inexplicit.

Here too, the theory of algebraic operads provides a conceptual reason for why deformations are encoded by Lie algebras. It is an instance of Koszul duality between dg commutative algebras (the local Artinian algebras) and dg Lie algebras. \\

This survey splits into three chapters. In the first one, we review the theory of algebraic operads, with the goal of understanding Koszul duality and the structure of the model category of operads and of algebras over a fixed operad. 
In the second chapter we focus on the tools to treat deformation problems associated to a given dg Lie algebra, namely the Deligne groupoid associated to a Lie algebra and the more general problem of integrating homotopy Lie algebras. We review $\infty$-algebras and $\infty$-morphisms under this perspective as well as the homotopy transfer theorem. We will see how these methods apply to specific problems, such as Kontsevich's formality theorem or to rational homotopy theory.
In the final chapter we address the more (derived) geometrical approach to deformation theory, namely the notion of formal moduli problems and the formalization of the deformation slogan, under the algebraic operadist perspective.

 We will review the construction of the (co)tangent complex and understand problems that can be infinitesimally deformed along non-commutative algebras.
At the very end, we will review recent developments, that fall outside of the framework of this such as deformation theory in positive characteristic.
\subsection*{Acknowledgments}

This survey was based on a mini-course the first author gave on the Workshop on Higher Structures and Operadic Calculus organized at the CRM in Barcelona. We would like to thank all participants and organizers and Bruno Vallette in particular. We are deeply grateful to Vladimir Dotsenko and Joost Nuiten for a very thorough review of the paper. We would also like to thank Pelle Steffens for helpful discussions during the writing of this review and Miguel Barata and José Moreno-Fernandez for comments on the first ArXiV version.

This project has received funding from  the grant 
ANR-20-CE40-0016 HighAGT and from the European Research Council (ERC) under the
European Union’s Horizon 2020 research and innovation programme (grant agreement No
768679).

\subsection*{Conventions}
Throughout the text, unless otherwise specified we work over a field $k$ of characteristic zero.
We always work over cochain complexes and use cohomological conventions, regardless of whether a subscript or a superscript is used for the degrees. For example, both Hochschild homology and cohomology arise from a complex whose differential has degree $+1$.

All algebraic objects are considered differential graded by assumption and we will thus not use the prefix ``dg''. For instance, what we will call an associative algebra is what is more commonly called a differential graded algebra elsewhere. The few instances we will work with non-dg objects, we will explicitly use the phrasing ``graded'' to mean that there is no differential and ``concentrated in degree $0$'' for objects such as vector spaces, having nothing in non-zero degree. In particular, $k$-module and cochain complex are synonyms.

Whenever applicable, we distinguish the internal Hom from the Hom set by using $\iHom(V,W)$ for the first and $\Hom(V,W)$ for the latter.

A list of notations is present at the end.

\needspace{6\baselineskip}

	\section{Algebraic Operad Theory} \

	The goal of this section is to introduce the important elements of the theory of algebraic symmetric operads.  In Section \ref{Sec_Algebraic Operads, Cooperads and their (co)-Algebras}, we introduce the definitions of symmetric operads, cooperads and the main examples of (co)operads we are going to use in this review. In Section \ref{Sec_(Co)AlgebraicStructureOn(Co)Operads}, we discuss the notion of algebra and coalgebra over operads and cooperads. We also give the construction of the (co)free (co)algebra adjunctions. In Section \ref{Sec_Some Classical Constructions for Algebraic Operads} we describe the convolution operad, the bar-cobar constructions and the notion of twisting morphisms. 
	These three sections are meant to be a shortcut for the reader to the content of \cite{LV} up to Section 6, which is generally the reference we recommend for details.
	Finally, Sections \ref{Sec_Model Categorical Aspects}, \ref{Sec_KoszulDuality} and \ref{sec:algebraic-bar-cobar-adjunction} are devoted to model categorical aspects of operad theory, describing the model structure on operads and on algebras over an operad. We also discuss the notion of Koszul duality as a means to obtain simple cofibrant replacements of operads. \\

	\subsection{Algebraic Operads and Cooperads}\
	\label{Sec_Algebraic Operads, Cooperads and their (co)-Algebras}

	\medskip

Let us start with a motivating example.
A (differential graded or dg) associative algebra is a cochain complex $(A,d)$ together with a bilinear product $\cdot$ such that for all $a,b,c\in A$
\begin{enumerate}
	\item\label{operadic} $(a\cdot b)\cdot c = a\cdot (b \cdot c)$ (associativity),
	\item\label{leibniz} $d(a\cdot b) = da \cdot b + (-1)^{\deg a}a \cdot db$  (compatibility with the differential).	
\end{enumerate}
The operadic approach consists in removing the spotlight from the elements $a,b,c$ (which are arbitrary) and rather focus on the multiplication operation itself. One could instead define an associative algebra structure on a cochain complex $A$ to be a map $\mu \colon A\otimes A \to A$, such that $\mu(\mu,\id_A) = \mu(\id_A,\mu)$. Notice that this is an equality of maps $A\otimes A \otimes A \to A$.
An immediate advantage of this perspective is that we no longer need to encode property \eqref{leibniz}. Indeed, it is hidden in the assumption that $\mu \colon A\otimes A \to A$ is a map of complexes.
Property \ref{leibniz} is ``categorical'', it follows formally from working with the underlying category of cochain complexes and this is the reason why we choose to drop the prefix dg: a dg associative algebra is the same as an associative algebra in cochain complexes. 
From this perspective it also follows immediately from applying the homology functor to $\mu$ that $H(A)$ is an associative algebra (in graded vector spaces).
Property \ref{operadic} on the other hand is ``structural'' and is the one suitable to be studied with the operadic machinery. 
\\

Even if we start with an operation taking two inputs, by composing it, we can write identities that take place in $\End_A(n)\coloneqq \iHom(A^{\otimes n},A)$. We think of these as multiendomorphisms or \emph{arity} $n$ endomorphisms of $A$.
We could have considered instead \emph{unital} associative algebras, which would amount to furthermore choosing an arity $0$ endomorphism $1\colon k \to A$, such that $\mu(1,\id_A) = \mu(\id_A,1) = \id_A \colon A\to A$.\\

An operad $\_P$ will be the kind of object such that its representations, i.e. maps of operads $\_P \to \End_A$, correspond precisely to $\_P$-algebra structures on $A$. The important properties to retain are the existence of an \emph{arity}, representing the number of inputs of the operation, as well as a way to \emph{compose} operations.

Furthermore, a way to encode symmetries is necessary: if we want to consider commutative algebras this amounts to require $\mu\colon A\otimes A \to A$ to factor through the $\Sigma_2$-coinvariants of the source. Or in a Lie algebra, even ignoring the anti-symmetry of the bracket, if we wish to encode the Jacobi identity, we need to permute the inputs of $[a,[b,c]]$.\\

	The objects encoding the collection of $n$-ary operations of an operad will be taken in a fixed category. We will use the category $\chk$ of cochain complexes over $k$ in this survey, although much of what is discussed can be extended to any closed  symmetric monoidal model category such as the category of simplicial sets $\tx{sSet}$ or $A$-modules $\Mod_A$ for example. The first thing we define is a category whose objects are given by collections of $n$-ary operations. 

	\begin{Def}  \label{Def_SymmetricSequences} 
		
A \defi{symmetric sequence} is a collection of elements in $\chk$,  $M := \lbrace M(n) \rbrace_{n \geq 0}$ together with right actions of the symmetric groups $M(n)\curvearrowleft \Sigma_n$  for all $n \geq 0$. A morphism $f \colon M \rightarrow N$ between symmetric sequences is the data of a collection of maps in $\chk$, $f_n : M(n) \rightarrow N(n)$ that are invariant with respect to the actions of $\Sigma_n$ on $M(n)$ and $N(n)$.
The category of symmetric sequences is denoted $\symseq$.

\end{Def}

	\begin{RQ}
		
		Let $\Nn^\sim$ denote the category with objects the sets  $\{1,\dots,n\}$ for all $n \in \Nn$ and morphisms given by the permutations of these sets. 
		
	The category of symmetric sequences in $\chk$ is isomorphic to the category of functors, $\tx{Fun}\left( \mathbb N^\sim, \chk \right)$, that is, we have an equivalence\footnote{Sending a functor $F$ to the collection $(F(n))_{n \in \Nn}$ and the permutation of $\lbrace 1, \cdots, n\rbrace$ are the morphisms in $\Nn^\sim$ that induce the action of $\Sigma_n$ on $F(n)$.} of categories:
	\[\symseq\simeq \tx{Fun}\left( \Nn^\sim, \chk \right).\]
\end{RQ}

{\begin{RQ} \label{RQ_NS and Colored Operads}
		
		The most common variations on the previous definition is to consider: \begin{itemize}
			\item $\Nn^o \subset \Nn^\sim$ the sub-category with the same objects but with only identities as morphisms.
			
			\item Given a set $S$ of \emph{colors}, denote $\Nn_S$ the category whose objects are $k$-tuples of elements on $S$ for all $k \in \Nn$, and whose morphisms are given by all permutations of these tuples, $(a_1, \cdots, a_k) \to (a_{\sigma_1}, \cdots , a_{\sigma_k})$ that preserve the colors.

		\end{itemize} 
		We can use the functorial construction with these alternative categories, we obtain:
		\begin{itemize}
			\item For $\Nn^o$, the notion of \defi{non-symmetric operads} (with no action of $\Sigma_n$).
			\item For $\Nn_S$, the notion of \defi{colored operad}.    
		\end{itemize}  
\end{RQ}}

	\begin{RQ}
	We think of the cochain complex $M(n)$ as the space of all $n$-ary operations of our operad. An operation $m \in M(n)$ will be depicted by a rooted planar tree as follows:
		
		 \begin{equation}\label{GR_operation}
		  \begin{tikzpicture}[scale=0.7, baseline=-1ex]
		 \GraphInit[vstyle=Classic]
		 
		 \tikzset{VertexStyle/.style = {shape = circle,fill = black,minimum size = 1pt,inner sep=0pt}}
		 
		 \SetVertexNoLabel
		 
		 \Vertex[empty, x=0,y=-1]{2}
		 
		 \SetVertexLabel

		 \Vertex[LabelOut,Lpos=-20, L=$m$, x=0,y=0]{0}
		 \Vertex[LabelOut, L= $\hdots$,Lpos=90, x=0,y=1]{3}
		 \Vertex[LabelOut,Lpos=90, L= $2$, x=-1, y=1]{B}
		 \Vertex[LabelOut,Lpos=90, L= $1$,x=-2,y=1]{A}
		 \Vertex[LabelOut,Lpos=90, L= $n-1$, x=1, y=1]{C}
		 \Vertex[LabelOut,Lpos=90, L= $n$, x=2,y=1]{D}
		 
		 \Edge(A)(0)
		 \Edge(B)(0)
		 \Edge(3)(0)
		 \Edge(C)(0)
		 \Edge(D)(0)
		 \Edge(0)(2)
		 \end{tikzpicture}  
		 \end{equation}  
	
	The action of $\Sigma_n$ is the action that permutes the $n$ leaves of the tree and sends $m$ to an other operation $\sigma.m$.
	\end{RQ}
	
	The idea behind algebraic operads is that these trees correspond to operations with $n$ inputs and $1$ output. We will need to make sense of how to compose these trees by concatenating them as depicted by the following graph:
	
	\begin{equation*}
	\begin{tikzpicture}[scale=0.6, baseline=-1ex]
	\GraphInit[vstyle=Classic]
	
	\tikzset{VertexStyle/.style = {shape = circle,fill = black,minimum size = 1pt,inner sep=0pt}}
	
	\SetVertexNoLabel

	\Vertex[empty, x=0,y=-2]{O}
	
	\Vertex[empty,  x=0,y=-1]{TF}
	
	\SetVertexLabel
	
		\Vertex[LabelOut,Lpos=-180, L=$ M(i_1)\ni$, x=-2,y=1]{TG1}  
			\Vertex[LabelOut,Lpos=0, L=$\in M(i_k)$,  x=2,y=1]{TG2} 
	
	\Vertex[LabelOut, L= $\cdots$, Lpos=90, x=0,y=0]{0}
	\Vertex[LabelOut, L= $1$,  Lpos=180,  x=-2,y=0]{1}
	\Vertex[LabelOut, L= $k$, x=2,y=0]{2}

	\Vertex[LabelOut, L= $1$,Lpos=90, x=-3,y=2]{E1}
	
	\Vertex[LabelOut, L= $i_1$,Lpos=90, x=-1,y=2]{E3}
	\Vertex[LabelOut, L= $1$,Lpos=90, x=1,y=2]{E4}
	
	\Vertex[LabelOut, L= $i_k$,Lpos=90, x=3,y=2]{E6}
	
	\Vertex[LabelOut, L= $\cdots$,Lpos=90, x=-2,y=2]{E2}
	\Vertex[LabelOut, L= $\cdots$,Lpos=90, x=2,y=2]{E5}
	
	\Edge(TF)(O)
	
	\Edge(0)(TF)
	\Edge(1)(TF)
	\Edge(2)(TF)
	
	\Edge(TG1)(1)
	\Edge(TG2)(2)
	
	\Edge(E1)(TG1)
	\Edge(E2)(TG1)
	\Edge(E3)(TG1)

	\Edge(E4)(TG2)
	\Edge(E5)(TG2)
	\Edge(E6)(TG2)

	\end{tikzpicture} 
	\end{equation*}  
	
	 To make sense of having the set of all $k$-tuples of trees that we want as input for an other $k$-ary tree, we define a symmetric tensor product on $\symseq$.     		 In the tree above, the upper level should be interpreted as an arity $i_1+\dots +i_k$ element of $M(i_1)\otimes \dots \otimes M(i_k)$.
	
	\begin{Def}\label{Def_DayConvolution}
		
		Given $\_D$, a $\_V$-enriched category and $\_V$, a closed symmetric monoidal category with small limits and colimits. Then the \defi{Day convolution} of two functors $F,G : \_D \rightarrow \_V$ is defined by the coend construction (see \cite[Proposition 6.2.1]{Coend}):
		\[ F \otimes_{\tx{Day}} G (-) := \int^{(c,c') \in \_D \times \_D} \_D (c \otimes_\_D c', -) \otimes_\_V F(c) \otimes_\_V G(c') \]
		
	\end{Def}

Will will now describe the Day convolution using concrete formulas in the context we are interested in, that is when $\_D = \Nn^\sim$ and $\_V = \chk$.


We then have the following formula for $\otimes_\tx{Day}$ which can equally been taken as a definition for the reader less confortable with coend calculus.
	
	\begin{Prop}
 				For $\_D = \Nn^\sim$ and $\_V = \chk$, we obtain:
		\[  F \otimes_{\tx{Day}} G (n) =  \bigoplus_{p+q = n}\tx{Ind}_{\Sigma_p \times \Sigma_q}^{\Sigma_n} (F(p) \otimes G(q) ).\]
	{where $\tx{Ind}_{\Sigma_p \times \Sigma_q}^{\Sigma_n}$ denotes the induction\footnote{{Recall that given a morphism of associative algebras $f\colon A\to B$, we have an induction/extension of scalars functor $B\otimes_A - \colon \Mod_A \to \Mod_B$, left adjoint to the restriction of scalars along $f$.
			When $f$ arises from a group morphism, such as $\Sigma_p \times \Sigma_{q} \hookrightarrow \Sigma_n$ for $p+q=n$, we refer to the corresponding induction as the induced representation $$k[\Sigma_n]\otimes_{k[\Sigma_p\times \Sigma_{q}]} (-) = \tx{Ind}_{\Sigma_p \times \Sigma_{q}}^{\Sigma_n}\colon \left(\Sigma_p\times \Sigma_q\right)\text{-}\tx{Rep} \to \Sigma_n\text{-}\tx{Rep}.$$} }
		 from $\Sigma_p \times \Sigma_q$ representations to $\Sigma_n$-representations.}\nopagebreak
	\end{Prop}

	\begin{proof}
		For $\_D = \Nn^\sim$ we get\footnote{With $p \otimes q := p+q$.}: 
	\[ F \otimes_{\tx{Day}} G (d) = \int^{p,q \in \Nn} \iHom (p \otimes q, d) \otimes F(p) \otimes G(q)\]
		
		We obtain that $F \otimes_{\tx{Day}} G (d)$ is equal to:\\
		 
			\adjustbox{scale=0.85,center}{$ \tx{coeq} \left(\begin{tikzcd}
		\coprod\limits_{p,q \in \Nn} \iHom(p+q, d) \otimes F(p) \otimes G(q)  & \arrow[l, shift left] \arrow[l, shift right] \coprod\limits_{f: p \rightarrow p' , \ g: q \rightarrow q'}  \iHom(p + q, d) \otimes F(p') \otimes G(q') 
		\end{tikzcd} \right) $}\\
		
		$\iHom(p+q, d)$ is non-zero only if $p+q = d$, $p=p'$ and $q=q'$. Therefore we get: 
		 \[ \tx{coeq} \left(\begin{tikzcd}
		 \coprod\limits_{p+q =d} k[\Sigma_d] \otimes F(p) \otimes G(q)  & \arrow[l, shift left] \arrow[l, shift right] \coprod\limits_{p+q=d}\coprod\limits_{\Sigma_p \times \Sigma_q}  k[\Sigma_d] \otimes F(p) \otimes G(q) 
		 \end{tikzcd} \right) \]
		 \[=\coprod\limits_{p+q =d}\tx{coeq} \left(\begin{tikzcd}
		  k[\Sigma_d] \otimes F(p) \otimes G(q)  & \arrow[l, shift left] \arrow[l, shift right] \coprod\limits_{\Sigma_p \times \Sigma_q}  k[\Sigma_d] \otimes F(p) \otimes G(q) 
		 \end{tikzcd} \right) \]

		 We now need to relate this coequalizer to the tensor product over $k[\Sigma_p \times \Sigma_q]$. To do that we notice that for any $p,q \in \Nn$ with $p+q =d$, $N \in k[\Sigma_d]- \Mod$, a map	\[\Hom_{k[\Sigma_d]} \left(\tx{coeq} \left(\begin{tikzcd}[column sep = tiny]
		 	k[\Sigma_d] \otimes F(p) \otimes G(q)  & \arrow[l, shift left] \arrow[l, shift right] \coprod\limits_{\Sigma_p \times \Sigma_q}  k[\Sigma_d] \otimes F(p) \otimes G(q) 
		 	\end{tikzcd} \right), N \right) 
		  \]
		  
		  is exactly given by a map $f: k[\Sigma_d] \otimes F(p) \otimes G(q) \rightarrow N$ such that for all $(\alpha_p, \alpha_q) \in \Sigma_p \times \Sigma_q$, $\sigma \in \Sigma_d$, $f_p \in F(p)$ and $g_q \in G(q)$ we have:
		 \[ f(\sigma.(\alpha_p, \alpha_q) \otimes f_p \otimes g_q) = f(\sigma \otimes \alpha_p .f_p \otimes \alpha_q. g_q)\]
		  
		  But this is exactly the data of a map:
		   \[\Hom_{\Mod_{k[\Sigma_d]}}\left( k[\Sigma_d]\otimes_{k[\Sigma_p \times \Sigma_q]} F(p)\otimes G(q), N \right)\]
		   
		    This shows that:
		\[\begin{split}
		F \otimes_{\tx{Day}} G(d) &= \coprod_{p+q = d} (F(p) \otimes G(q)) \otimes_{k[\Sigma_p \times \Sigma_q]} k[\Sigma_d] \\
		&=   \coprod_{p+q = d}\tx{Ind}_{\Sigma_p \times \Sigma_q}^{\Sigma_d} (F(p) \otimes G(q) )
		\end{split}\]
	\end{proof}

	\begin{RQ}
		Taking successive Day tensor powers of a functor leads to the following formula: 
		\[ G^{\otimes_{\tx{Day}}k}(n)  = \bigoplus_{i_1 + \cdots + i_k = n}\tx{Ind}_{\Sigma_{i_1} \times \cdots \times \Sigma_{i_k}}^{\Sigma_n} (F(i_1) \otimes \cdots \otimes F(i_k))  \]
		
		Objects of arity $n$ in this symmetric sequence can be graphically understood as having $k$ trees in parallel with total arity $i_1 + \cdots + i_k = n$, with labels of different trees potentially permuted among them:
		 \[ \begin{tikzpicture}[scale=1, baseline=-1ex]
			\GraphInit[vstyle=Classic]
			
			\tikzset{VertexStyle/.style = {shape = circle,fill = black,minimum size = 1pt,inner sep=0pt}}
			
			\SetVertexNoLabel
			
			\Vertex[ x=0,y=0]{0}
			\Vertex[empty, x=0,y=-1]{2}
			
			\SetVertexLabel

			\Vertex[LabelOut, L= $\hdots$,Lpos=90, x=0,y=1]{3}
			\Vertex[LabelOut,Lpos=90, L= $\scriptstyle\sigma(1)$, x=-1, y=1]{B}
			\Vertex[LabelOut,Lpos=90, L= $\scriptstyle \sigma(i_1)$, x=1,y=1]{D}
		
			\Edge(B)(0)
			\Edge(3)(0)
			\Edge(D)(0)
			\Edge(0)(2)
		\end{tikzpicture} \qquad \cdots \qquad \begin{tikzpicture}[scale=1, baseline=-1ex]
		\GraphInit[vstyle=Classic]
		
		\tikzset{VertexStyle/.style = {shape = circle,fill = black,minimum size = 1pt,inner sep=0pt}}
		
		\SetVertexNoLabel
		
		\Vertex[ x=0,y=0]{0}
		\Vertex[empty, x=0,y=-1]{2}
		
		\SetVertexLabel

		\Vertex[LabelOut, L= $\hdots$,Lpos=90, x=0,y=1]{3}
		\Vertex[LabelOut,Lpos=90, L= $\scriptstyle \sigma(i_1+\dots +i_{k-1} +1)  \qquad \qquad$, x=-1, y=1]{B}
		\Vertex[LabelOut,Lpos=90, L= $\qquad \qquad \scriptstyle\sigma(i_1+\dots +i_k)$, x=1,y=1]{D}
		
		\Edge(B)(0)
		\Edge(3)(0)
		\Edge(D)(0)
		\Edge(0)(2)
	\end{tikzpicture}   \]

		The induction, $\tx{Ind}_{\Sigma_{i_1} \times \cdots \times \Sigma_{i_k}}^{\Sigma_n}$, describes the action of $\Sigma_n$ on the $n$ entries in a way that is compatible with the actions given on each individual tree,  $\Sigma_{i_1} \times \cdots \times \Sigma_{i_k}$.
	This consists in adding the $i_1,\dots, i_k$ shuffles $\mathrm{Sh}(i_1,\dots, i_k)$ which are canonical representatives of the cosets $\Sigma_n/{\Sigma_{i_1} \times \cdots \times \Sigma_{i_k}}$.	
	\end{RQ}

We want now to describe all the ways we can put $k$ trees (elements in $M^{\otimes_{\tx{Day}}k} $) as input for a $k$-ary tree so that we can compose them. This defines a new symmetric sequence described by the following non-symmetric monoidal product: 
  
	\begin{Def}\label{def:circ of SS}
		We define the \defi{composition of symmetric sequences} to be the monoidal structure (see \cite[Section 6.3]{Coend}) on $\symseq$ given by:
		\[ F \circ G  := \int^m F(m) \otimes G^{\otimes_{\tx{Day}}m}  \]
		
		We obtain the formula: 
		\[ F \circ G (n) = \bigoplus_{k\geq 0} F(k) \otimes_{k[\Sigma_k]} \left(\bigoplus_{i_1 + \cdots +i_k = n } \tx{Ind}_{\Sigma_{i_1} \times \cdots \times \Sigma_{i_k}}^{\Sigma_n} ( G(i_1) \otimes \cdots \otimes G(i_n) )\right)\]
		
	\end{Def}

		Graphically $F \circ G $ corresponds to the symmetric sequence of all possible concatenation of a $F$-tree with $k$-entries with $k$ $G$-trees on top on it: 
		
	  \begin{equation}\label{GR_ConcatenationOfTrees}
	  	\begin{tikzpicture}[scale=0.8, baseline=-1ex]
			\GraphInit[vstyle=Classic]
			
			\tikzset{VertexStyle/.style = {shape = circle,fill = black,minimum size = 1pt,inner sep=0pt}}
			
			\SetVertexNoLabel

			\Vertex[empty, x=0,y=-2]{O}

			\SetVertexLabel
			
		    \Vertex[LabelOut,Lpos=-10, L= $F$, x=0,y=-1]{TF}
			
		    \Vertex[LabelOut, L= $\cdots$, Lpos=90, x=0,y=0]{0}
		    	\Vertex[LabelOut, L= $1$,  Lpos=180,  x=-2,y=0]{1}
		    \Vertex[LabelOut, L= $k$, x=2,y=0]{2}

			\Vertex[LabelOut, L= $G_1$, x=-2,y=1]{TG1}
			\Vertex[LabelOut, L= $G_k$, x=2,y=1]{TG2}
			
			\Vertex[LabelOut, L= $1$,Lpos=90, x=-3,y=2]{E1}
			
			\Vertex[LabelOut, L= $i_1$,Lpos=90, x=-1,y=2]{E3}
			\Vertex[LabelOut, L= $1$,Lpos=90, x=1,y=2]{E4}
			
			\Vertex[LabelOut, L= $i_k$,Lpos=90, x=3,y=2]{E6}
			
			\Vertex[LabelOut, L= $\cdots$,Lpos=90, x=-2,y=2]{E2}
			\Vertex[LabelOut, L= $\cdots$,Lpos=90, x=2,y=2]{E5}
			
			\Edge(TF)(O)
			
			\Edge(0)(TF)
			\Edge(1)(TF)
			\Edge(2)(TF)
			
			\Edge(TG1)(1)
			\Edge(TG2)(2)
			
			\Edge(E1)(TG1)
			\Edge(E2)(TG1)
			\Edge(E3)(TG1)

			\Edge(E4)(TG2)
			\Edge(E5)(TG2)
			\Edge(E6)(TG2)

		\end{tikzpicture} \in (F\circ G)(i_1 + \cdots + i_k)
		\end{equation}

	Such concatenation will be called a $2$-leveled tree. In the picture above we display the unconcatenated labels, but keep in mind that the bottom labels disappear and the top level is now labeled by a permutation of $i_1+\dots +i_k$. In general, a $n$-fold product for $\circ$ will give a $n$-leveled tree. 
	
	The unit of this monoidal product will be denoted by $I$ and is the symmetric sequence given by $I(1)= k$ and $I(n)=0$ for $n\neq 1$. 
	
	The notion of composition of an operad amounts to sending a $2$-leveled tree, given by an element in $F \circ F$, to a new operation in $F$ viewed as the ``composition'' of these leveled trees. 
	
	\begin{Def}\label{Def_Operad}
				
		The category of \defi{operads} $\_P$ valued in $\chk$, denoted by $\Op$, is the category of monoid objects in $\symseq$ for $\circ$. Concretely an object $\_P \in \Op$ is a symmetric sequence $\_P$ together with an associative multiplication and a unit morphisms: 
		\[ \mu : \_P \circ \_P \rightarrow \_P \qquad \eta : I \rightarrow \_P \] 
		
		$\mu$ corresponds to composing the operations by sending any concatenation of  trees as depicted in \ref{GR_ConcatenationOfTrees} (with $G = F = \_P$) to a new element in $\_P$. In particular for each $n \in \Nn$ and $i_1, \cdots , i_n \in \Nn$, we get the composition map: 
		\[ \_P (n) \otimes \left( \_P(i_1) \otimes \cdots \otimes \_P (i_k) \right) \rightarrow \_P (i_1 + \cdots + i_k) \]
		
		Which pictorially gives: 
		 \[ \begin{tikzpicture}[scale=1, baseline=-1ex]
			\GraphInit[vstyle=Classic]
			
			\tikzset{VertexStyle/.style = {shape = circle,fill = black,minimum size = 1pt,inner sep=0pt}}
			
			\SetVertexNoLabel
			
			\Vertex[empty, x=0,y=-2]{O}

			\SetVertexLabel
			
			\Vertex[LabelOut, L= $F$, Lpos=-5, x=0,y=-1]{TF}
			
			\Vertex[LabelOut, L= $\cdots$, Lpos=90, x=0,y=0]{0}
			\Vertex[LabelOut, L= $1$,  Lpos=180,  x=-2,y=0]{1}
			\Vertex[LabelOut, L= $k$, x=2,y=0]{2}
	
			\Vertex[LabelOut, L= $F_1$, x=-2,y=1]{TG1}
			\Vertex[LabelOut, L= $F_k$, x=2,y=1]{TG2}
			
			\Vertex[LabelOut, L= $1$,Lpos=90, x=-3,y=2]{E1}
			
			\Vertex[LabelOut, L= $i_1$,Lpos=90, x=-1,y=2]{E3}
			\Vertex[LabelOut, L= $1$,Lpos=90, x=1,y=2]{E4}
			
			\Vertex[LabelOut, L= $i_k$,Lpos=90, x=3,y=2]{E6}
			
			\Vertex[LabelOut, L= $\cdots$,Lpos=90, x=-2,y=2]{E2}
			\Vertex[LabelOut, L= $\cdots$,Lpos=90, x=2,y=2]{E5}
			
			\Edge(TF)(O)
			
			\Edge(0)(TF)
			\Edge(1)(TF)
			\Edge(2)(TF)
			
			\Edge(TG1)(1)
			\Edge(TG2)(2)
			
			\Edge(E1)(TG1)
			\Edge(E2)(TG1)
			\Edge(E3)(TG1)

			\Edge(E4)(TG2)
			\Edge(E5)(TG2)
			\Edge(E6)(TG2)

		\end{tikzpicture}   \mapsto 	\begin{tikzpicture}[scale=1, baseline=-1ex]
		\GraphInit[vstyle=Classic]
		
		\tikzset{VertexStyle/.style = {shape = circle,fill = black,minimum size = 1pt,inner sep=0pt}}
		
		\SetVertexNoLabel
	
		\Vertex[empty, x=0,y=-1]{O}
		
		\SetVertexLabel
		
		\Vertex[LabelOut, L= $\mu(F \otimes \left(F_1 \otimes \cdots \otimes F_k\right) )$, Lpos=-5, x=0,y=0]{TF}
		
		\Vertex[LabelOut, L= $1$,Lpos=90, x=-2,y=1]{E1}
		\Vertex[LabelOut, L= $i_1+ \cdots + i_k$,Lpos=90, x=2,y=1]{E3}
		\Vertex[LabelOut, L= $\cdots$,Lpos=90, x=0,y=1]{E2}
		
		\Edge(TF)(O)
		\Edge(E1)(TF)
		\Edge(E2)(TF)
		\Edge(E3)(TF)
	\end{tikzpicture} \]

		These maps satisfy the natural associativity and unitality conditions coming from the definition of monoids. In particular this means that given a $n$-level tree in $\_P^{\circ n}$, it is canonically sent to a tree $\_P$ no matter in which order we chose to compose each levels. 
	\end{Def}

	\begin{Def}\label{def:partial comp}	Equivalently, an operadic structure can be completely described via the \defi{partial compositions}. For $\alpha \in \_P(n)$, $\beta \in \_P(m)$ and for all $1 \leq i \leq n$ we defined $\circ_{i}$ as: 
		\[ \alpha \circ_i \beta = \mu_{\_P} (\alpha, (I, \cdots, I, \beta, I, \cdots, I )) \in \_P (n+m-1)\]
		
		where $\beta$ is in the i-th position in the $n$-tuple $(I, \cdots, I, \beta, I, \cdots, I )$. This can be depicted as:   
		
		 \begin{equation*}
		\begin{tikzpicture}[scale=0.8, baseline=-1ex]
		\GraphInit[vstyle=Classic]
		
		\tikzset{VertexStyle/.style = {shape = circle,fill = black,minimum size = 1pt,inner sep=0pt}}
		
		\SetVertexNoLabel
		
		\Vertex[empty, x=0,y=-1]{2}

		\Vertex[LabelOut,Lpos=-20, L=$\mu$, x=-1,y=2]{E1}
		\Vertex[LabelOut,Lpos=-20, L=$\mu$, x=0,y=2]{E2}
		\Vertex[LabelOut,Lpos=-20, L=$\mu$, x=1,y=2]{E3}
		\SetVertexLabel

		\Vertex[LabelOut,Lpos=-20, L=$\alpha$, x=0,y=0]{0}
		\Vertex[LabelOut,Lpos=0, L= $\beta$, x=0,y=1]{3}
		\Vertex[LabelOut,Lpos=90, L= $i-1$, x=-1, y=1]{B}
		\Vertex[empty, x=-2,y=1]{A}
		\Vertex[LabelOut,Lpos=90, L= $i+1$, x=1, y=1]{C}
		\Vertex[empty, x=2,y=1]{D}
		
		\Edge(A)(0)
		\Edge(B)(0)
		\Edge(3)(0)
		\Edge(C)(0)
		\Edge(D)(0)
		\Edge(0)(2)
		\Edge(E1)(3)
		\Edge(E2)(3)
		\Edge(E3)(3)
		
		\end{tikzpicture}  
		\end{equation*}  

	\end{Def}
In the picture above we display the original label, but keep in mind that relabeling after composition will make the $i+1$ label into $i+m$.

As a consequence of this equivalent description, one can interpret an operad structure on a symmetric sequence $\_P$, to be a rule to assign to any tree $T$ (rooted and planar, but without any leveling), with internal vertices decorated by elements of $\_P$ in appropriate arity, a single element of $\_P(l)$, where $l$ is the number of leaves of $T$.

	\begin{Cons}\label{Cons_Operad}\
			
	\begin{itemize}
		\item \textbf{Free operad:} 		
The forgetful functor $\Op \rightarrow \symseq$ has a left adjoint $\T$. Given a symmetric sequence $E$, we denote by $\T E$ the \emph{free operad} generated by $E$.
		
	As a symmetric sequence, $\T E$ is spanned by all possible trees whose {internal} vertices are decorated by elements of $E$ of the appropriate arity and by $I$ of arity $1$ (which we interpret as having no {internal} vertices). The (cohomological) degree of the tree is the sum of the degrees of all {internal} vertices and the differential acts vertex by vertex. The operadic structure is done by grafting as depicted in the following example:

\[ \begin{tikzpicture}[scale=0.78, baseline=-1ex]
\GraphInit[vstyle=Classic]

\tikzset{VertexStyle/.style = {shape = circle,fill = black,minimum size = 1pt,inner sep=0pt}}

\SetVertexNoLabel

\Vertex[ x=0,y=0]{0}
\Vertex[empty, x=0,y=-1]{S}

\SetVertexLabel

\Vertex[LabelOut, L= $2$,Lpos=90, x=0,y=2]{2}
\Vertex[LabelOut,Lpos=90, L= $1$, x=-1, y=2]{1}
\Vertex[LabelOut,Lpos=90, L= $3$, x=1,y=2]{3}

\Edge(1)(0)
\Edge(3)(0)
\Edge(2)(0)
\Edge(0)(S)
\end{tikzpicture} \circ_2 \begin{tikzpicture}[scale=0.78, baseline=-1ex]
\GraphInit[vstyle=Classic]

\tikzset{VertexStyle/.style = {shape = circle,fill = black,minimum size = 1pt,inner sep=0pt}}

\SetVertexNoLabel

\Vertex[ x=0,y=0]{0}
\Vertex[empty, x=0,y=-1]{S}
\Vertex[empty, x=-1,y=1]{I1}

\SetVertexLabel

\Vertex[LabelOut, L= $2$,Lpos=90, x=0,y=2]{2}
\Vertex[LabelOut,Lpos=90, L= $1$, x=-2, y=2]{1}
\Vertex[LabelOut,Lpos=90, L= $3$, x=2,y=2]{3}

\Edge(1)(I1)
\Edge(3)(0)
\Edge(2)(I1)
\Edge(I1)(0)
\Edge(0)(S)
\end{tikzpicture} =  \begin{tikzpicture}[scale=0.55, baseline=-1ex]
\GraphInit[vstyle=Classic]

\tikzset{VertexStyle/.style = {shape = circle,fill = black,minimum size = 1pt,inner sep=0pt}}

\SetVertexNoLabel

\Vertex[ x=0,y=0]{0}
\Vertex[empty, x=0,y=-1]{S}
\Vertex[empty, x=0,y=1]{I1}
\Vertex[empty, x=-1,y=2]{I2}

\SetVertexLabel

\Vertex[LabelOut,Lpos=90, L= $1$, x=-3, y=3]{1}
\Vertex[LabelOut, L= $2$,Lpos=90, x=-2,y=3]{2}
\Vertex[LabelOut,Lpos=90, L= $3$, x=0,y=3]{3}
\Vertex[LabelOut,Lpos=90, L= $4$, x=2,y=3]{4}
\Vertex[LabelOut,Lpos=90, L= $5$, x=3,y=3]{5}

\Edge(1)(0)
\Edge(2)(I2)
\Edge(3)(I2)
\Edge(4)(I1)
\Edge(5)(0)
\Edge(I2)(I1)
\Edge(I1)(0)
\Edge(0)(S)
\end{tikzpicture}  \]

 We will denote by $\T^{(n)}E\subset \T E$ the subcomplex spanned by trees with exactly $n$ internal vertices.		
\\

		\item \textbf{Operad defined by generators and relations:} 		
		 Consider an operadic ideal\footnote{ $\_I \subset \_P$ is called an \defi{operadic ideal} if $\mu(l \otimes \left(l_1\otimes \cdots \otimes l_k\right)) \in \_I$ as soon as one of $l,l_1 , \cdots, l_k$ is in $\_I$.} $\_I \subset \_P$. In this situation, the cokernel of the map $\_I \rightarrow \_P$, denoted by $\faktor{\_P}{\_I}$, is called the \defi{quotient operad}. 
		 
		 If we consider $E \in \symseq $, $R \subset \T E$ and $(R)$ the operadic ideal generated by $R$ in $\T E$. Then the operad generated by $E$ with relations $R$ is defined by:
		\[ \_P(E,R) := \faktor{\T E}{(R)}. \]

			\end{itemize}
	\end{Cons}

\begin{Def}\label{def:quadratic}
	An operad presented in terms of generators and relations $\_P(E,R)$ as above is said to be \defi{binary} if the generating symmetric sequence is concentrated in arity $2$, i.e. $E(\ne 2)=0$.
	It is said to be \defi{quadratic} if the $R\subset \T^{(n)}E$, i.e. if relations always involve two operations. 
\end{Def}

Loosely speaking, operads are naturally given in terms of generators and relations, whenever the  algebras they describe are defined in terms of their defining operations and the properties they must satisfy.
	
	\begin{RQ}\label{warning}
		Let $E$ be a symmetric sequence such that $E(\ne 2) = 0$. While there are only two possible arity $3$ trees one can write using two binary vertices, due to the symmetries, $\T^{(2)} E \cong (E(2) \otimes E(2))^{\oplus 3}$. Indeed, there are exactly three $(2,1)$ shuffles. Pictorially the three types of elements of $\T^{(2)} E$ are: 
		
		\[ \begin{tikzpicture}[scale=0.5, baseline=-1ex]
		\GraphInit[vstyle=Classic]
		
		\tikzset{VertexStyle/.style = {shape = circle,fill = black,minimum size = 1pt,inner sep=0pt}}
		
		\SetVertexNoLabel
		
		\Vertex[ x=0,y=0]{0}
		\Vertex[empty, x=0,y=-1]{S}
		\Vertex[empty, x=-1,y=1]{I1}

		\SetVertexLabel

		\Vertex[LabelOut, L= $2$,Lpos=90, x=0,y=2]{2}
		\Vertex[LabelOut,Lpos=90, L= $1$, x=-2, y=2]{1}
		\Vertex[LabelOut,Lpos=90, L= $3$, x=2,y=2]{3}
		
		\Edge(1)(0)
		\Edge(3)(0)
		\Edge(2)(I1)
		\Edge(I1)(0)
		\Edge(0)(S)
		\end{tikzpicture}, \qquad \begin{tikzpicture}[scale=0.5, baseline=-1ex]
		\GraphInit[vstyle=Classic]
		
		\tikzset{VertexStyle/.style = {shape = circle,fill = black,minimum size = 1pt,inner sep=0pt}}
		
		\SetVertexNoLabel
		
		\Vertex[ x=0,y=0]{0}
		\Vertex[empty, x=0,y=-1]{S}
		\Vertex[empty, x=1,y=1]{I1}

		\SetVertexLabel

		\Vertex[LabelOut, L= $2$,Lpos=90, x=0,y=2]{2}
		\Vertex[LabelOut,Lpos=90, L= $1$, x=-2, y=2]{1}
		\Vertex[LabelOut,Lpos=90, L= $3$, x=2,y=2]{3}
		
		\Edge(1)(0)
		\Edge(3)(0)
		\Edge(2)(I1)
		\Edge(I1)(0)
		\Edge(0)(S)
		\end{tikzpicture} = \begin{tikzpicture}[scale=0.5, baseline=-1ex]
		\GraphInit[vstyle=Classic]
		
		\tikzset{VertexStyle/.style = {shape = circle,fill = black,minimum size = 1pt,inner sep=0pt}}
		
		\SetVertexNoLabel

		\Vertex[empty, x=0,y=-1]{S}
		\Vertex[empty, x=-1,y=1]{I1}

		\SetVertexLabel
		
		\Vertex[LabelOut, L= ${}_{(12)}$, x=0,y=0]{0}
		\Vertex[LabelOut, L= $3$,Lpos=90, x=0,y=2]{2}
		\Vertex[LabelOut,Lpos=90, L= $2$, x=-2, y=2]{1}
		\Vertex[LabelOut,Lpos=90, L= $1$, x=2,y=2]{3}
		
		\Edge(1)(0)
		\Edge(3)(0)
		\Edge(2)(I1)
		\Edge(I1)(0)
		\Edge(0)(S)
		\end{tikzpicture}, \qquad \begin{tikzpicture}[scale=0.5, baseline=-1ex]
		\GraphInit[vstyle=Classic]
		
		\tikzset{VertexStyle/.style = {shape = circle,fill = black,minimum size = 1pt,inner sep=0pt}}
		
		\SetVertexNoLabel
		
		\Vertex[ x=0,y=0]{0}
		\Vertex[empty, x=0,y=-1]{S}
		\Vertex[empty, x=-1,y=1]{I1}

		\SetVertexLabel

		\Vertex[LabelOut, L= $3$,Lpos=90, x=0,y=2]{2}
		\Vertex[LabelOut,Lpos=90, L= $1$, x=-2, y=2]{1}
		\Vertex[LabelOut,Lpos=90, L= $2$, x=2,y=2]{3}
		
		\Edge(1)(0)
		\Edge(3)(0)
		\Edge(2)(I1)
		\Edge(I1)(0)
		\Edge(0)(S)
		\end{tikzpicture}\]
	with internal vertices labeled by elements of $E(2)$.
			\end{RQ}

		\begin{Ex}\label{Ex_Operad}\ 
			
				\begin{itemize}
		\item \textbf{Trivial operad:}		
		The symmetric sequence $I$ defined by $I(1) = k$ and $I(k)=0$ for $k \neq 1$ is the unit for $\circ$ and possesses a natural operad structure given by the identity as monoidal product and unit:
		\[ I = I \circ I \to I \qquad I \to I\]

		\item  \textbf{Associative operad:}		
		The associative operad is the operad generated by two binary operations, $m$ and its permutation $m' = m \cdot (12)$, subject to the quadratic associativity relations which is the smallest {$\Sigma_3$-submodule} of $\T^{(2)} E$ containing: 
		 \[ 	\begin{tikzpicture}[scale=0.8, baseline=-1ex]
			\GraphInit[vstyle=Classic]
			
			\tikzset{VertexStyle/.style = {shape = circle,fill = black,minimum size = 1pt,inner sep=0pt}}
			
			\SetVertexNoLabel

			\Vertex[empty, x=0,y=-2]{O}

			\SetVertexLabel
			
			\Vertex[LabelOut, L= $m$, x=0,y=-1]{1}
			
			\Vertex[empty, x=-1,y=0]{A}
			\Vertex[LabelOut, L= $m$, x=1,y=0]{B}

			\Vertex[empty, x=-2,y=1]{E1}
			\Vertex[empty, x=0,y=1]{E2}
			\Vertex[empty, x=2,y=1]{E3}

			\Edge(1)(O)
			
			\Edge(B)(1)
			
			\Edge(E1)(1)
			\Edge(E2)(B)
			\Edge(E3)(B)

		\end{tikzpicture} - 	\begin{tikzpicture}[scale=0.8, baseline=-1ex]
		\GraphInit[vstyle=Classic]
		
		\tikzset{VertexStyle/.style = {shape = circle,fill = black,minimum size = 1pt,inner sep=0pt}}
		
		\SetVertexNoLabel

		\Vertex[empty, x=0,y=-2]{O}

		\SetVertexLabel
		
		\Vertex[LabelOut, L= $m$, x=0,y=-1]{1}
		
		\Vertex[LabelOut,L= $m$, x=-1,y=0]{A}
		\Vertex[empty, x=1,y=0]{B}

		\Vertex[empty, x=-2,y=1]{E1}
		\Vertex[empty, x=0,y=1]{E2}
		\Vertex[empty, x=2,y=1]{E3}

		\Edge(1)(O)
		
		\Edge(A)(1)
		
		\Edge(E1)(A)
		\Edge(E2)(A)
		\Edge(E3)(1)

	\end{tikzpicture} \]

The complex $\T^{(2)} E= (E(2) \otimes E(2))^{\oplus 3}$ is generated by all permutations of the partial compositions $m \circ_{1} m$ and $m \circ_{2} m$. The relations imposed for the associative algebra are generated by all the permutations of $m \circ_{1} m - m \circ_{2} m$:
	 \[ 	\begin{tikzpicture}[scale=0.8, baseline=-1ex]
	\GraphInit[vstyle=Classic]
	
	\tikzset{VertexStyle/.style = {shape = circle,fill = black,minimum size = 1pt,inner sep=0pt}}
	
	\SetVertexNoLabel

	\Vertex[empty, x=0,y=-2]{O}

	\SetVertexLabel
	
	\Vertex[LabelOut, L= $m$, x=0,y=-1]{1}
	
	\Vertex[empty, x=-1,y=0]{A}
	\Vertex[LabelOut, L= $m$, x=1,y=0]{B}

	\Vertex[LabelOut,L= $x_{\sigma(1)}$, x=-2,y=1]{E1}
    \Vertex[LabelOut,L= $x_{\sigma(2)}$, x=0,y=1]{E2}
    \Vertex[LabelOut,L= $x_{\sigma(3)}$, x=2,y=1]{E3}

	\Edge(1)(O)
	
	\Edge(E1)(1)
	\Edge(B)(1)
	
	\Edge(E2)(B)
	\Edge(E3)(B)

	\end{tikzpicture} - 	\begin{tikzpicture}[scale=0.8, baseline=-1ex]
	\GraphInit[vstyle=Classic]
	
	\tikzset{VertexStyle/.style = {shape = circle,fill = black,minimum size = 1pt,inner sep=0pt}}
	
	\SetVertexNoLabel

	\Vertex[empty, x=0,y=-2]{O}

	\SetVertexLabel
	
	\Vertex[LabelOut, L= $m$, x=0,y=-1]{1}
	
	\Vertex[LabelOut,L= $m$, x=-1,y=0]{A}
	\Vertex[empty, x=1,y=0]{B}

\Vertex[LabelOut,L= $x_{\sigma(1)}$, x=-2,y=1]{E1}
\Vertex[LabelOut,L= $x_{\sigma(2)}$, x=0,y=1]{E2}
\Vertex[LabelOut,L= $x_{\sigma(3)}$, x=2,y=1]{E3}

	\Edge(1)(O)
	
	\Edge(A)(1)
	
	\Edge(E1)(A)
	\Edge(E2)(A)
	\Edge(E3)(1)

	\end{tikzpicture} \]

	  The quotient vector space we obtain in arity $3$ is therefore $6$-dimensional and is isomorphic to $k[\Sigma_3 ]$. In general, we have that $\mathbf{Ass}(0)=0$ and for $n\geq 1$, $\mathbf{Ass}(n) = k[\Sigma_n]$ together with the natural action of $\Sigma_n$.\\
		
		\item \textbf{Commutative operad:} 		
		The commutative operad is defined as the operad generated by a single binary operation, invariant by the $\Sigma_2$ action (i.e. commutative) with quadratic relations encoding associativity. The free operad satisfies $\T^{(2)} E = (E(2)\otimes E(2))^{\oplus 3}$ and is therefore $3$-dimensional. If we write $xy$ for the product in $E(2)$, $\T^{(2)}E$ is freely generated by  $x(yz)$, $(xz)y$ and $(xy)z$. 
		
		So we only have to add the associativity relations $x(yz)-y(zx)$, $y(zx)-z(xy)= (zx)y-z(xy)$. We give a pictorial description for the first relation:

		\[	\begin{tikzpicture}[scale=0.75, baseline=-1ex]
		\GraphInit[vstyle=Classic]
		
		\tikzset{VertexStyle/.style = {shape = circle,fill = black,minimum size = 1pt,inner sep=0pt}}
		
		\SetVertexNoLabel

		\Vertex[empty, x=0,y=-2]{O}

		\SetVertexLabel
		
		\Vertex[LabelOut, L= $$, x=0,y=-1]{1}
		
		\Vertex[LabelOut,L= $$, x=-1,y=0]{A}
		\Vertex[empty, x=1,y=0]{B}

		\Vertex[LabelOut,L= $x$, x=-2,y=1]{E1}
		\Vertex[LabelOut,L= $y$, x=0,y=1]{E2}
		\Vertex[LabelOut,L= $z$, x=2,y=1]{E3}

		\Edge(1)(O)
		
		\Edge(A)(1)
		\Edge(B)(1)
		
		\Edge(E1)(A)
		\Edge(E2)(B)
		\Edge(E3)(1)

		\end{tikzpicture}  
%
%
%
%
%
%
%
%
%
%
%
%
%
%
%
%
-\begin{tikzpicture}[scale=0.75, baseline=-1ex]
		\GraphInit[vstyle=Classic]
		
		\tikzset{VertexStyle/.style = {shape = circle,fill = black,minimum size = 1pt,inner sep=0pt}}
		
		\SetVertexNoLabel

		\Vertex[empty, x=0,y=-2]{O}

		\SetVertexLabel
		
		\Vertex[LabelOut,L= $$, x=0,y=-1]{1}
		
		\Vertex[empty, x=-1,y=0]{A}
		\Vertex[LabelOut, L= $$, x=1,y=0]{B}

		\Vertex[LabelOut,L= $y$, x=-2,y=1]{E1}
		\Vertex[LabelOut,L= $z$, x=0,y=1]{E2}
		\Vertex[LabelOut,L= $x$, x=2,y=1]{E3}

		\Edge(1)(O)
		
		\Edge(A)(1)
		\Edge(B)(1)
		
		\Edge(E1)(1)
		\Edge(E2)(B)
		\Edge(E3)(B)

		\end{tikzpicture}  \]

		  Since there are $3$ generators and $2$ relations, the commutative operad is therefore $1$-dimensional in arity $3$. In general, we have that $\mathbf{Com}(n) = k$ for all $n \geq 1$ with the trivial action.  \\
		  
		  		\item \textbf{Unital commutative operad:} 
		  		
The unital commutative operad $\mathbf{uCom}$ is an example of an operad which is neither binary nor quadratic. Besides the binary generator there is one additional generator, corresponding to the unit, which is an arity zero element $E(0) = k$. The relation involving the unit involves a quadratic and a constant term

	\[ \begin{tikzpicture}[scale=0.7, baseline=-1ex]
	
	\GraphInit[vstyle=Simple]
	\tikzset{VertexStyle/.style = {shape = circle,fill = black,minimum size = 4pt,inner sep=0pt}}

	\SetVertexNoLabel
	
	\Vertex[x=-1,y=2]{T}
	\tikzset{VertexStyle/.style = {shape = circle,fill = black,minimum size = 1pt,inner sep=0pt}}
	
	\Vertex[x=1,y=1]{E2}

	\Vertex[x=0,y=-1]{O}
	
	\SetVertexLabel
	\Vertex[LabelOut, L = $1$, x=-1,y=1]{E1}
	\Vertex[LabelOut, L = $\mu$, x=0,y=0]{M}
	\Edge(M)(O)
	
	\Edge(E1)(M)
	\Edge(E2)(M)
	\Edge(T)(E1)
	
	

\end{tikzpicture} 
-\quad
  \begin{tikzpicture}[scale=0.7, baseline=-1ex]
	
	\GraphInit[vstyle=Simple]
	\tikzset{VertexStyle/.style = {shape = circle,fill = black,minimum size = 4pt,inner sep=0pt}}

	\SetVertexNoLabel
	
	\Vertex[empty, x=0,y=1]{T}
	\tikzset{VertexStyle/.style = {shape = circle,fill = black,minimum size = 1pt,inner sep=0pt}}
	
	\Vertex[ x=0,y=-.8]{O}

	\SetVertexLabel
	\Vertex[LabelOut, L = $\tx{id}_{\mathbf{uCom}}$,x=0,y=0]{E1}
	
	\Edge(E1)(O)
	\Edge(T)(E1)
	
	

\end{tikzpicture}\]

which lives in $\T^{(2)}  E\oplus\T^{(0)} E$. The operad $\mathbf{uCom}$ differs from $\mathbf{Com}$ only in arity $0$, $\mathbf{uCom}(0) = k$. Similarly, there is an operad of unital associative algebras which only differs from $\mathbf{Ass}$ in that $\mathbf{uAss}(0) = k$.
		  			  		\\
		\item \textbf{Lie operad:} 	
		The Lie operad, denoted by $\mathbf{Lie}$, is generated by a binary operation $[\text{-},\text{-}]$ acted upon by $\Sigma_2$ via $[\text{-},\text{-}]\cdot \tau = -[\text{-},\text{-}]$. It satisfies the associativity relation given by the Jacobiator. We have $\mathbf{Lie}(2) = k$ but the description of $\mathbf{Lie}(n)$ for $n>2$ is more complicated. We refer to  \cite[Section 13.2.3 and 13.2.4]{LV} for more details on the description of $\mathbf{Lie}$.    \\

		\item \textbf{Endomorphisms operad:} 		
		We will now see the first operad not naturally presented in terms of generators and relations.		
		 Take $V \in \chk$. We define the endomorphism operad of $V$, denoted $\mathbf{End}_V$, by:
		\[\mathbf{End}_V(n) := \iHom_{k}\left(V^{\otimes n}, V\right) \]	
		where the action of $\Sigma_n$ on $\mathbf{End}(n)$ is obtained from the natural action of $\Sigma_n$ on $V^{\otimes n}$. The composition is naturally given by the composition for $\iHom$ defined by $\mu(f, f_1, \cdots , f_m) = f \circ (f_1 \otimes \cdots \otimes f_m)$ with $f \in \mathbf{End}_V (m)$. \\ 
		
		Loosely speaking, $\mathbf{End}_V$ is the operad in which all possible operations on $V$ live. We will see in Section \ref{Sec_(Co)AlgebraicStructureOn(Co)Operads} that we can define a $\_P$-algebra structure on $V$ as  a morphism of operads $\mu_V : \_P \rightarrow \mathbf{End}_V$.
	
	\item \textbf{Coendomorphisms operad:}		
	Take $V \in \chk$. We define the coendomorphism operad of $V$ to be $\mathbf{coEnd}_V$ where:
	\[\mathbf{coEnd}_V(n) := \iHom(V, V^{\otimes n}) \]
	
	and the action of $\Sigma_n$ on $\mathbf{coEnd}(n)$ is obtained from the natural action of $\Sigma_n$ on $V^{\otimes n}$. The composition is naturally given by the composition for $\iHom$ given by $\mu(f, f_1, \cdots , f_m) = f_1\otimes \cdots \otimes  f_m \circ f$ with $f \in \mathbf{coEnd}_V (m)$.\\  
		Similar to the  previous example, $\mathbf{coEnd}_V$ can be used to define $\_P$-coalgebra structures on $V$.

%
	\end{itemize}
\end{Ex} 

	Now that we have discussed the notion of operad, which amounts to describing how to compose $n$-ary operations, we will describe the notion of cooperad, where instead of composing we will decompose an $n$-ary operation into a sum of $i$-leveled trees. 
	In a similar fashion to Definition \ref{def:circ of SS}, we define the \defi{complete composition} of symmetric sequences to be
		\[ F \hat{\circ} G = \int_m F(m) \otimes G^{\otimes_{\tx{Day}}m}.\]

\begin{Prop}\label{Prop_FormulaDualMonoidalStructure}	
	For $F$ and $G$ in $\symseq$, then we have the following formula:  
	\[ F \hat{\circ} G (n) = \prod_{k\geq 0} \left( F(k) \otimes_{k} \left(\bigoplus_{i_1 + \cdots +i_k = n } \tx{Ind}_{\Sigma_{i_1} \times \cdots \Sigma_{i_k}}^{\Sigma_n} ( G(i_1) \otimes \cdots \otimes G(i_k) )\right)\right)^{\Sigma_k} \]
\end{Prop}

Compared with $\circ$, $\hat\circ$ carries some disadvantages, one of which being that it is not monoidal (since direct products do not distribute along tensor products). 

In order to circumvent that problem, we define another monoidal product $\bar{\circ}$ as the sub-symmetric sequence of $\int_m F(m) \otimes G^{\otimes_{\tx{Day}}m}$ given by the direct sum instead. Using the formula obtained in Proposition \ref{Prop_FormulaDualMonoidalStructure} we can define the following monoidal product by replacing the product by a direct sum:
\begin{equation}\label{Eq_cocompostionsymmetricsequences}
	\resizebox{.9\hsize}{!}{$F \bar{\circ} G (n) = \bigoplus\limits_{k\geq 0} \left( F(k) \otimes_{k} \left(\bigoplus\limits_{i_1 + \cdots +i_k = n } \tx{Ind}_{\Sigma_{i_1} \times \cdots \Sigma_{i_k}}^{\Sigma_n} ( G(i_1) \otimes \cdots \otimes G(i_k) )\right)\right)^{\Sigma_k}$}
\end{equation}

\begin{RQ}\label{RQ_NormalisationIsomorphism}
	The normalization morphism $ V^G \rightarrow V_G$ from invariants to coinvariants induces a natural transformation $\bar{\circ} \rightarrow \circ $ sending the direct sum of invariants to the direct sum of coinvariants. Since we work in characteristic $0$ this is a natural isomorphism and therefore we will not worry too much about the distinction. In particular, $\bar{\circ}$ is indeed monoidal.
\end{RQ}

%
		
			\begin{Def}\label{Def_Cooperads}
				
			The category of \defi{cooperads in }$\chk$, denoted $\coOp$ is defined a the category of comonoids in $\symseq$ for  $\bar{\circ}$. Concretely an object $\_C \in \coOp$ is a symmetric sequence $\_C$ together with cocomposition and counit morphisms: 
			\[ \Delta : \_C \rightarrow \_C \bar{\circ} \_C \qquad \epsilon : \_C \rightarrow I\]

		\end{Def}

		\begin{RQ}	
				
		This pictorial description also involves trees for the $n$-ary operations in $M(n)$. This does not change compared to operad because of the natural isomorphism of Remark \ref{RQ_NormalisationIsomorphism}. However, the difference with cooperads is that the coalgebra structure will not compose trees but rather decompose them into a sum of concatenations of trees built via $\bar{\circ}$. 
		
		For example:
		\[ \begin{tikzpicture}[scale=0.7, baseline=-1ex]
		\GraphInit[vstyle=Classic]
		
		\tikzset{VertexStyle/.style = {shape = circle,fill = black,minimum size = 1pt,inner sep=0pt}}
		
		\SetVertexNoLabel

		[label="I"]\Vertex[empty, x=-1,y=1]{E1}
		\Vertex[empty, x=0,y=1]{E2}
		\Vertex[empty, x=1,y=1]{E3}
		
		\Vertex[empty, x=0,y=0]{M}
		
		\Vertex[empty, x=0,y=-1]{O}
		
		\SetVertexLabel

		\Edge(M)(O)
		
		\Edge(E1)(M)
		\Edge(E2)(M)
		\Edge(E3)(M)

		\end{tikzpicture} \mapsto \begin{tikzpicture}[scale=0.7, baseline=-1ex]
		\GraphInit[vstyle=Classic]
		
		\tikzset{VertexStyle/.style = {shape = circle,fill = black,minimum size = 1pt,inner sep=0pt}}
		
		\SetVertexNoLabel

		\Vertex[empty, x=-1,y=2]{E1'}
		\Vertex[empty, x=0,y=2]{E2'}
		\Vertex[empty, x=1,y=2]{E3'}
		
		\Vertex[empty, x=-1,y=1]{E1}
		\Vertex[empty, x=0,y=1]{E2}
		\Vertex[empty, x=1,y=1]{E3}
		
		\Vertex[empty, x=0,y=0]{M}
		
		\Vertex[empty, x=0,y=-1]{O}
		
		\SetVertexLabel

		\Edge(M)(O)
		
		\Edge(E1)(M)
		\Edge(E2)(M)
		\Edge(E3)(M)
		\Edge[label=$I$](E1')(E1)
		\Edge[label=$I$](E2')(E2)
		\Edge[label=$I$](E3')(E3)
		
		\end{tikzpicture} + \begin{tikzpicture}[scale=0.7, baseline=-1ex]
		\GraphInit[vstyle=Classic]
		
		\tikzset{VertexStyle/.style = {shape = circle,fill = black,minimum size = 1pt,inner sep=0pt}}
		
		\SetVertexNoLabel

		\Vertex[empty, x=-2,y=2]{E1}
		\Vertex[empty, x=0,y=2]{E2}
		\Vertex[empty, x=2,y=2]{E3}
		
		\Vertex[empty, x=-1,y=1]{M1}
		\Vertex[empty, x=0,y=0]{M2}
				
		\Vertex[empty, x=0,y=-1]{O}
		
		\SetVertexLabel

		\Edge(M)(O)
		
		\Edge(E1)(M1)
		\Edge(E2)(M1)
		\Edge(E3)(M2)
		\Edge(M1)(M2)
		
		\end{tikzpicture}+ \begin{tikzpicture}[scale=0.7, baseline=-1ex]
		\GraphInit[vstyle=Classic]
		
		\tikzset{VertexStyle/.style = {shape = circle,fill = black,minimum size = 1pt,inner sep=0pt}}
		
		\SetVertexNoLabel

		\Vertex[empty, x=-2,y=2]{E1}
		\Vertex[empty, x=0,y=2]{E2}
		\Vertex[empty, x=2,y=2]{E3}
		
		\Vertex[empty, x=0,y=0]{M1}
		\Vertex[empty, x=1,y=1]{M2}
		
		\Vertex[empty, x=0,y=-1]{O}
		
		\SetVertexLabel

		\Edge(M)(O)
		
		\Edge(E1)(M1)
		\Edge(E2)(M2)
		\Edge(E3)(M2)
		\Edge(M1)(M2)
		
		\end{tikzpicture} + \cdots\]

		Concretely, it is the data of maps:
		
		\[ \_C(n) \rightarrow  \bigoplus_{k\geq 0} \left(\_C (k) \otimes \left(\bigoplus_{i_1 + \cdots  +i_k = n}\tx{Ind}_{\Sigma_{i_1} \times \cdots \Sigma_{i_k}}^{\Sigma_n} \left( \_C(i_1) \otimes \cdots \otimes \_C(i_k) \right)\right)\right)^{\Sigma_k} \]

		together with a counit, satisfying some coassociativity and counitality conditions derived from the comonoid axioms. 
		\end{RQ}

Similar to the case of operads, cooperads also have a notion of partial cocomposition, which consists in cocomposing along a single edge $i$ into two vertex trees. This is given by maps $\Delta_i\colon \_C(n) \to \_C(n-k+1) \otimes \_C(k)$, for all $n\geq 0$, $i=1,\dots, n$ and $k\leq n$.

\begin{Def}\label{Def_ConilpotentCooperad}
	
	A cooperad $\_C$ is \defi{conilpotent} if for every $m \in \_C(n)$, there exists an $N$ such that the iteration of any choice of $N$ partial cocomposition maps is equal to $0$. The category of conilpotent cooperads will be denoted by $\mathbf{coOp}^{\tx{conil}}$
	
\end{Def}

All cooperads we will consider in this text will be conilpotent. Notice that any cooperad such that $\_C(0) = 0$ and $\_C(1) = k$ is conilpotent ($N$ can be taken to be $n-1$ for a given $m \in \_C(n)$).

	\begin{Cons}\label{Cons_Cooperads}\ 
	
	\begin{itemize}
		 \item \textbf{Cofree conilpotent cooperad:}		 
		 	The cofree conilpotent cooperad associated to a symmetric sequence $E$ is denoted $\Tc E$ is the image on $E$ via the cofree functor defined as the right adjoint of the forgetful functor $F : \coOp^{\tx{conil}} \rightarrow \symseq$. The underlying symmetric sequence of $\Tc E$ is isomorphic to the one of $\T E$ and the cocomposition of a tree is obtained by ``cutting'' along an internal edge.
Such a cooperad is conilpotent since a given tree with $e$ edges can be only non-trivially partially cocomposed up to $e$ times. \\

		\item \textbf{Cooperad by generators and relations:} 		
		The cooperad cogenerated by a symmetric sequence $E$ with corelations $R \subset \Tc E$, denoted $\_C(E, R)$, is defined to be the smallest subcooperad of $\Tc E$ containing $R$. In the quadratic binary case, we have $\_C(E, R)(2)= E$ and $\_C(E, R)(3)=R$.
%
		 \\
		
		\item \label{Ex_DualOperad}\textbf{Dual of an operad:}		
		 The linear dual of a symmetric sequence $M$ is a symmetric sequence $M^\vee$ defined in each arity by the linear dual $M^\vee(n)$. For any cooperad $\_C$, $\_C^\vee$ has the structure of an operad. Since $\bar{\circ}$ is not exactly the dual version of $\circ$, the converse is not always true. For an operad $\_P$ we only get a map, $\_P^\vee \rightarrow \_P^\vee \hat{\circ} \_P^\vee$ which might not factor through a comonoid $\_P^\vee \rightarrow \_P^\vee \bar{\circ} \_P^\vee$. If we further suppose that $\_P(0) = 0$  and that the pre-image of each element in $\_P$ by $\mu_{\_P}$ is finite dimensional then we get a cooperad structure on $\_P^\vee$ (see \cite[Section 5.8.2]{LV}).
		 	
	\end{itemize}
\end{Cons}

\subsection{Algebraic Structures over an Operad or Cooperad}\
\label{Sec_(Co)AlgebraicStructureOn(Co)Operads}

\medskip

In this section we are going to define the notion of algebra and coalgebra both over operads and cooperads. We already gave some definitions in the case of operads in Example \ref{Ex_Operad}, but there are definitions in terms of modules and comodules over the operads viewed as monoids. Proposition \ref{Prop_AgebraicStructureOperad} shows that both notions coincide. More generally, in this section will define several adjunctions involving the monoidal products defined earlier. This will allow us to describe the free $\_P$-algebra and cofree $\_P$-coalgebra functors associated to an operad $\_P$ (see Proposition \ref{Prop_FreeAlgebraOverAnOperad}). Finally we will discuss the naturality of the category of algebras with respect to morphisms of operads.

	\begin{Def} \label{Def_AlgebraCoalgebraOperad}
		
		Let $A \in \chk$ and $\_P \in \Op$. 
		\begin{itemize}
			\item	The structure of a \defi{$\_P$-algebra} on $A$ is the structure of a $\_P$-module on $A$ with respect to $\circ$, where we interpret $A$ as a symmetric sequence concentrated in arity $0$: 
			\[ \mu\colon \_P \circ A \rightarrow A \]
					
			Unraveling the 	definition, this amounts to associating to every $p\in \_P(n)$ and $a_1,\dots,a_n$ a single element in $A$, which we can shorten to $p(a_1,\dots,a_n)$ if the map $\mu$ is implicit.
			The category of $\_P$-algebras is denoted \[ \mathbf{Alg}_\_P := \_P\text{-}\Mod_\circ, \]
			using the convention that modules will always refer to sequences concentrated in arity $0$.
			We will see in Proposition \ref{Prop_AgebraicStructureOperad} that a $\_P$-algebra structure on $A$ is equivalent to the choice of a morphism of operads $\_P\to \mathbf{End}_A$.

			\item Perhaps the expected way to define a $\_P$-coalgebra would be in terms of a map $A \to \_P^\vee \hat{\circ} A$. However, as $\hat{\circ}$ is not monoidal, it makes no sense to ask for this map to be a comodule. As such, a \defi{$\_P$-coalgebra} structure on $A$ is defined as a map of operads: 
			\[ \_P \to \mathbf{coEnd}_A \]
			The category of $\_P$-coalgebras is denoted 
			 \[\mathbf{coAlg}_\_P := \Hom_{\Op}\left( \_P, \mathbf{coEnd}_A \right)\]
			
			\item 	We will see with Lemma \ref{Lem_closedmonoidalstructure} that a coalgebra $\_P \to \mathbf{coEnd}_A$ induces a map $A \to \_P^\vee \hat{\circ} A$, assuming $\_P$ is arity-wise dualizable\footnote{A symmetric sequence $M$ is called arity-wise dualizable if for each $n \in \Nn$, $M(n)$ is dualizable, which in our setting corresponds to bounded and finite dimensional in every degree.} and $\_P^\vee$ is a cooperad.
			
			 Therefore we can define a \defi{conilpotent coalgebra} as a comodule\footnote{This comodule extends to a map $A \to \_P^\vee \hat{\circ} A$ which is equivalent, thanks to Lemma \ref{Lem_closedmonoidalstructure}, to a map of symmetric sequences $\_P \to \mathbf{End}_A$. This turns out to also be a map of operad that defines the underlying coalgebra structure (by forgetting the conilpotency).}: \[A \to \_P^\vee \bar{\circ} A\]

			The category of conilpotent $\_P$-coalgebras is denoted by $\mathbf{coAlg}_\_P^{\tx{conil}}$ and we have:
			 \[ \mathbf{coAlg}_\_P^{\tx{conil}} = \_P^\vee-\tx{coMod}_{\bar{\circ}}\]
		\end{itemize}
 
	\end{Def}

	\begin{Def}
		\label{Def_AlgebraCoalgebraCoOperad}
		
		Let $A \in \chk$ and $\_C$ be a cooperad. The structure of a \defi{$\_C$-(co)algebra on $A$} is defined as a $\_C^\vee$-(co)algebra structure on $A$. We denote by $\mathbf{Alg}_{\_C}$ and $\mathbf{coAlg}_{\_C}$ the categories of all $\_C$-algebras and $\_C$-coalgebras respectively. We have:
	 \[ \mathbf{Alg}_\_C := \_C^\vee\text{-}\Mod_{\circ} \qquad \mathbf{coAlg}_\_C := \Hom_{\Op}\left(\_C^\vee, \mathbf{coEnd}_A \right)\]

Moreover, a conilpotent $\_C$-coalgebra can be defined as a comodule in $\_C-\tx{coMod}_{\bar{\circ}}$. We get:
\[ \mathbf{coAlg}_\_C^{\tx{conil}} := \_C-\tx{coMod}_{\bar{\circ}}\]

	\end{Def}

	\begin{RQ}\label{Def_AlgebraicStructureInRMod}
		If we replace $\chk$ by $\Mod_R$ (with $R$ a commutative algebra), to obtain (co)operads valued in $\Mod_R$, we obtain the notion of (co)algebra in $R$-modules, that is $R$-linear algebraic structures.  
	\end{RQ}
	
	In the Example \ref{Ex_Operad}, we described the endomorphism operad, $\mathbf{End}_A$, and defined a $\_P$-algebra structure on $A$ as a morphism of operad $\_P \rightarrow \mathbf{End}_A$. It turns out that this coincides with the previous definition as we will explain with Proposition \ref{Prop_AgebraicStructureOperad}. We start with a fundamental Lemma expressing some adjunctions between the monoidal product and the (co)endomorphism symmetric sequences. 
	
	\begin{Def}\label{Def_InternalHomSymSeq}
		Given $M, N \in \chk$, we define the following symmetric sequences: 
		\[ \mathbf{End}_N^M(n) = \iHom_{k}\left( M^{\otimes n}, N \right) \]
		\[ \mathbf{coEnd}_N^M(n) = \iHom_{k}\left( M, N^{\otimes n} \right). \]
	
%
%
	\end{Def} 
	
	An application of some end-coend gymnastics, together with the hom-tensor adjunction, give the following lemma.
	
	\begin{Lem}\label{Lem_closedmonoidalstructure} For all $\_P \in \symseq$ and for all $M,N \in \chk$, there is an isomorphism:
		\[ \Hom_{k} \left( \_P \circ M, N \right) \cong \Hom_{\symseq}\left( \_P, \mathbf{End}_N^M \right) \]		
		Morover if $\_P$ is arity-wise dualizable then we have an isomorphism: 
		\[ \Hom_{k} \left( M, \_P^\vee \hat{\circ} N \right) \cong \Hom_{\symseq}\left( \_P, \mathbf{coEnd}_N^M \right) \]
		
	\end{Lem}

	\begin{Prop}\label{Prop_AgebraicStructureOperad}
		
		Given $A \in \chk$ and $\_P \in \Op$, we have:  
		\[ \_P\text{-}\Mod_{\circ}(A) \cong \Hom_{\Op}\left( \_P, \mathbf{End}_A \right) \]
	\end{Prop}

	\begin{proof}
Following Lemma \ref{Lem_closedmonoidalstructure} it is a straightforward verification that the condition of $\_P\circ A \to A$ being a $\_P$-module structure on $A$ is equivalent to the map of symmetric sequences $\_P \to \mathbf{End}_A$ being a map of operads.	
	\end{proof}

	The previous results and Lemma \ref{Lem_closedmonoidalstructure} give us some adjunctions using the monoidal product $\circ$ and the $\mathbf{End}$ constructions. In particular, we obtain the following constructions:

	\begin{Prop}\label{Prop_FreeAlgebraOverAnOperad}\
			\begin{itemize}[itemsep=.2cm]
			\item		Given $\_P$ an operad, we have the following free-forgetful adjunction:
		\[ \_P(-) : \begin{tikzcd}
		\chk \arrow[r, shift left] & \arrow[l, shift left]  \mathbf{Alg}_{\_P}
		\end{tikzcd} : (-)^\sharp\]
		The left adjoint $\_P(-): \chk \rightarrow \mathbf{Alg}_{\_P}$ is the \defi{free $\_P$-algebra} functor that sends $A \in \chk$ to $\_P(A) := \_P \circ A$.		
		
	\item	Given $\_C$ an arity-wise dualizable cooperad, we have the following adjunction:
		\[  (-)^\sharp: \begin{tikzcd}
		\mathbf{coAlg}^\mathrm{conil}_{\_C} \arrow[r, shift left] & \arrow[l, shift left] \chk  
		\end{tikzcd} : \_C(-)\]
		The right adjoint $\_C(-): \chk \rightarrow \mathbf{coAlg}^\mathrm{conil}_{\_C}$ is the \defi{cofree conilpotent coalgebra} functor that sends $A \in \chk$ to $\_C(A) := \_C \bar{\circ} A$.
		
		In particular, if $\_P$ is arity-wise dualizable we get: 
		\[  (-)^\sharp: \begin{tikzcd}
		\mathbf{coAlg}^\mathrm{conil}_{\_P^\vee} \simeq \mathbf{coAlg}^\mathrm{conil}_\_P \arrow[r, shift left] & \arrow[l, shift left] \chk  
		\end{tikzcd} : \_P^\vee(-)\]
		this gives us the cofree $\_P$-coalgebra. 
		
	\end{itemize}
	
	\end{Prop}

	\begin{proof}
Let us consider the first adjunction.
Notice that $\_P(A)$ is naturally a $\_P$-algebra via the structure maps of $\_P$:  $(\_P\circ \_P)\circ A \to \_P \circ A$.

There is an inclusion of chain complexes $A = I \circ A \to \_P \circ A$ induced by the unit $I\to \_P$. Pulling back along this inclusion yields the ``restriction to generators'' map:
\[\Hom_{\mathbf{Alg}_\_P} \left( \_P(A), B \right) \rightarrow \Hom_{k}\left( A , B^\sharp \right)\] 
		
We shall show that this map is a bijection by constructing its inverse. Let $f\in\Hom_{k}\left( A , B^\sharp \right)$. We consider 

\[
\_P \circ A \stackrel{\id\circ f}{\to} \_P \circ B \to B,
\] 
	where the second map is the $\_P$-algebra structure on $B$. It is easy to check that this defines a map of $\_P$-algebras and establishes the inverse map.
	
The proof of the second bullet point is dual.
	\end{proof}

	\begin{Ex}\
		
		\begin{itemize}
			\item For $\mathbf{uAss}$, the free associative algebra functor is the tensor algebra functor. Indeed, given $V \in \chk$, we have that:
			\[\mathbf{Ass}(V) = \bigoplus_{n \geq 0} k[\Sigma_n] \otimes_{k[\Sigma_n]} V^{\otimes n} \simeq  \bigoplus_{n \geq 0} V^{\otimes n} = TV \]

			\item For $\mathbf{Com}$, the free commutative algebra functor is the  non-unital  symmetric algebra functor:
			\[\mathbf{Com}(V) = \bigoplus_{n \geq 1} k \otimes_{k[\Sigma_n]} V^{\otimes n} \simeq  \bigoplus_{n \geq 1} (V^{\otimes n})_{\Sigma_n}= \Sym_k^{\geq 1} V \]

		\end{itemize}
	\end{Ex}
\begin{Prop}\label{Prop_NaturalityAlgebraAdjunction}
	Given $f : \_P \rightarrow \_Q$ a morphism of operads we obtain an adjunction: 
	\[  \begin{tikzcd}
	 f_! : \mathbf{Alg}_\_P \arrow[r, shift left] & \arrow[l, shift left] \mathbf{Alg}_\_Q : f^*
	\end{tikzcd} \] 
	
	where $f^*$ sends the $\_Q$-structure $\_Q \rightarrow \mathbf{End}_A$ to the composition $\_P \to \_Q \to \mathbf{End}_A$ and $f_! (B)$, with $B \in \mathbf{Alg}_\_P$, is defined as the reflexive coequalizer: 
	
	\[ \_Q \circ_\_P B := \tx{coeq} \left( \begin{tikzcd}
	\_Q \circ \_P \circ B \arrow[r, shift left, "\mu_\_Q \circ f"] \arrow[r, shift right, "\mu_B"'] & \_Q \circ B
	\end{tikzcd}\right)\]

\end{Prop}

\begin{RQ}\label{RQ_RelativeMonoidalProductAndUniversalEnvAlgebra}
	The notation $\circ_\_P$ is suggestive. It is a relative version of the usual operation $\circ$. In fact for $f : I \rightarrow \_P$ this adjunction recovers the free-forget adjunction, we get $f^* A = A^\sharp$ and $f_! V = \_P \circ_I V = \_P \circ V$. 
\end{RQ}

\begin{RQ}\label{RQ_TrivalPAlgebra}
	When $\_P$ is an augmented operad, that is an operad together a retract $\_P \to I$ of the unit, we can define the trivial algebra functor given as the right adjoint $\tx{Triv} : \mathbf{Alg}_I = \chk \to \mathbf{Alg}_{\_P}$ coming from Proposition \ref{Prop_NaturalityAlgebraAdjunction} applied to the augmentation.  
\end{RQ}

\subsubsection{A short digression on the universal enveloping algebra}\
\label{sec:Lie digression}

\medskip

Recall that the universal enveloping algebra of a Lie algebra $\G_g$ is classically defined to be $\mathfrak U \G_g \coloneqq T(\G_g)/[x,y] = x\otimes y - y\otimes x$. This is exactly the induction functor along the operad morphism $f\colon\mathbf{Lie} \to \mathbf{Ass}$,
\[  \begin{tikzcd}
	\mathbf{Alg}_{\mathbf{Lie}} \arrow[r, shift left, "\G_U"] & \arrow[l, shift left, "f^*"] \mathbf{Alg}_{\mathbf{Ass}} 
\end{tikzcd} \]
where $f^* (A)$ is the Lie algebra structure on $A$ given by the Lie bracket $[x,y] = xy - yx$. 

It will be important later on to notice that there is a natural coproduct on $\G_U \G_g$ given by
$$\Delta(x) = 1\otimes x + x\otimes 1 \in \G_U \G_g \otimes \G_U \G_g, \tx{ for }\ x\in \G_g,$$

which can be extended to the whole $\G_U \G_g$ by making $\Delta(xy)=\Delta(x)\Delta(y)$.
Defining an antipodal map $S(x) = -x$, it follows that the universal enveloping algebra functor extends to a functor into (cocommutative) Hopf algebras.

\begin{Def}
Given a Hopf algebra $H$, we say that $g\in H\setminus\{0\}$ is \defi{grouplike} if $\Delta(g) = g\otimes g$. We say that $x\in H$ is \defi{primitive} if $\Delta(x) = 1\otimes x + x\otimes 1$.
\end{Def}

One can observe that for any Hopf algebra, the grouplike elements form a group, with inverse given by the antipode. As we will make precise in Section \ref{Sec_MaurerCartanSpace}, under some completion assumptions, when $H$ is a universal enveloping algebra $\G_U \G_g$, the grouplike elements can be interpreted as the exponential group of $\G_g$.

It is generally difficult to explicitly describe the cochain complex (in other words provide a basis) underlying the induction $f_!$ along an arbitrary morphism of operads $f\colon \_P \to \_Q$. In this particular case $f\colon \mathbf{Lie} \to \mathbf{Ass}$, the PBW (Poincar\'e--Birkoff--Witt) theorem gives a complete answer to this problem (in characteristic zero).

\begin{Th}[PBW, {\cite[Theorem B.2.3]{Qu69}}]\label{Th:PBW}
There is an isomorphism of cochain complexes: 
$$\mathfrak U \mathfrak g\cong \Sym \mathfrak g.$$
Furthermore, interpreting $\Sym \mathfrak g$ as the cofree cocommutative coalgebra generated by $\mathfrak g$, this is an isomorphism of coalgebras. 
\end{Th}
	
%

	\subsection{Classical Constructions for Algebraic Operads}\
	\label{Sec_Some Classical Constructions for Algebraic Operads}
	
	\medskip
	
	This section is a recollection of classical constructions and definitions on operads. We start by defining the convolution operad (Definition \ref{Def_ConvolutionOperad}) and the totalization of an operad (Definition \ref{Def_TotatlizationOfAnOperad}). These will be used to define the notion of (Koszul) twisting morphism (Definition \ref{Def_KoszulTwistingMorphisms}), which provides a convenient alternative description for the bar-cobar constructions (Definition \ref{Def_BarCobarConstruction}). The goal of this section is to give the main definitions around the bar-cobar adjunction, which will be used later in order to take good cofibrant replacements (Proposition \ref{Prop_BarCoBarCofibrantReplacement}), describe homotopy algebras (Section \ref{Sec_HomotopyAlgebrasAndMorphism}) and study deformations of algebraic structures (Section \ref{Sec_DeformationofAlgebraicStructures}).
 	
 	\begin{Def}
 		\label{Def_ConvolutionOperad}
		
		Given an operad $\_P$ and a cooperad $\_C$, we can construct the \defi{convolution operad}:
		\[ \tx{Conv}(\_C, \_P) (n) := \iHom_{k} \left( \_C (n), \_P (n) \right) \]

		The symmetric sequence structure on $\tx{Conv}(\_C, \_P)$ is given, for each $n\geq 0$, by the action by conjugation: $\sigma.f := x \in \_C(n) \mapsto \sigma. f(\sigma^{-1}.x)$. 
		
		The operadic structure is given by the map that sends
		\[ \phi := f \otimes(g_1 \otimes \cdots \otimes g_k) \in \tx{Conv}(\_C, \_P) \circ \tx{Conv}(\_C, \_P) (n),\] 
		with $g_l \in \tx{Conv}(\_C, \_P)(i_l)$ and $i_1 + \cdots + i_k = n$, to the element in $\tx{Conv}(\_C, \_P)$ of arity $n$ defined by the following composition:
				
		\adjustbox{scale=0.85,center}{
		\begin{tikzcd}[column sep= tiny]
		\_C \arrow[r, "\Delta"] & \_C \hat{\circ} \_C \arrow[r] & \_C(k) \otimes \_C(i_1) \otimes \cdots \otimes \_C (i_k) \arrow[r, "\phi"' ] &  \_P(k) \otimes \_P(i_1) \otimes \cdots \otimes \_P(i_k) \arrow[r, hookrightarrow] & \_P \circ \_P \arrow[r, "\mu"] & \_P 
		\end{tikzcd} }
	\end{Def}

	The object we are going to be interested in is not the convolution operad itself, but rather its totalization. The totalization of an operad is a pre-Lie algebra whose associated Lie algebra will  be important in defining twisting morphisms and describing deformations of algebraic structures. 

\begin{Def}[Pre-Lie and Permutative Operads]\label{Def_PermAndPreLie}\
	
\begin{itemize}
	\item 	A (right) \defi{permutative algebra} is an associative algebra $A$ with a binary operation satisfying:
	\[ x \cdot (y \cdot z) =(-1)^{|y||z|} x \cdot (z \cdot y)\]
	We denote by $\mathbf{Perm}$ the quadratic operad encoding permutative algebras.
	
	\item	A (right) \defi{pre-Lie algebra} is given by $\mathfrak g\in \chk$ with a bilinear operation $\star$ such that: 
	\[ (x\star y)\star z - x\star (y\star z)=(-1)^{|y||z|}\big( (x\star z)\star y - x\star (z\star y)\big)\]

	We denote by $\mathbf{preLie}$ the quadratic operad encoding pre-Lie algebras. $\mathbf{Perm}$ and $\mathbf{preLie}$ are related by the fact that they are Koszul dual of each other (Definition \ref{Def_KoszulDualOperad}).
\end{itemize}
\end{Def}

Pre-Lie algebras are so called since, similarly to associative algebras, the formula $[p,q] = p \star q - q \star p$ defines a Lie algebra structure on $\mathfrak g$. Moreover, there is a map $\mathbf{preLie} \to \mathbf{Ass}$ so that all associative algebras are pre-Lie algebras. In this text, the fundamental instances of pre-Lie algebras that do not come from an associative algebra will arise as totalizations of   operads.

\begin{Def}\label{Def_TotatlizationOfAnOperad}
	Given $\_P$ an operad we can produce a pre-Lie algebra called the \defi{totalization of $\_P$}: 
\[ \tx{Tot}(\_P) := \prod_{n\geq 0} \_P(n)^{\Sigma_n} \in \mathbf{Alg}_{\mathbf{preLie}} \]

The pre-Lie product $\alpha\star\beta $ is defined as the sum of all possible ways to compose $\alpha$ and $\beta$. Pictorially, if $\alpha\in \_P(3)^{\Sigma_3}$ and $\beta\in \_P(2)^{\Sigma_2}$,

	\[
\begin{tikzpicture}[scale=1, baseline=-1ex]
	\GraphInit[vstyle=Classic]
	
	\tikzset{VertexStyle/.style = {shape = circle,fill = black,minimum size = 1pt,inner sep=0pt}}
	
	\SetVertexNoLabel
	
	\Vertex[empty, x=0,y=-1]{1}
	
	\Vertex[empty, x=-1,y=1]{E1}
	\Vertex[empty, x=0,y=1]{E2}
	\Vertex[empty, x=1,y=1]{E3}
	
	\SetVertexLabel

	\Vertex[LabelOut, L=$\alpha$, x=0,y=0]{A}

	\Edge(A)(1)
	
	\Edge(E1)(A)
	\Edge(E2)(A)
	\Edge(E3)(A)

\end{tikzpicture} \star  \begin{tikzpicture}[scale=1, baseline=-1ex]
	\GraphInit[vstyle=Classic]
	
	\tikzset{VertexStyle/.style = {shape = circle,fill = black,minimum size = 1pt,inner sep=0pt}}
	
	\SetVertexNoLabel
	
	\Vertex[empty, x=0,y=-1]{1}
	
	\SetVertexLabel

	\Vertex[LabelOut, L=$\beta$, x=0,y=0]{A}

	\Vertex[empty, x=-1,y=1]{E1}
	
	\Vertex[empty, x=1,y=1]{E3}

	\Edge(A)(1)
	
	\Edge(E1)(A)
	
	\Edge(E3)(A)

\end{tikzpicture} =    \begin{tikzpicture}[scale=0.90, baseline=-1ex]
	\GraphInit[vstyle=Classic]
	
	\tikzset{VertexStyle/.style = {shape = circle,fill = black,minimum size = 1pt,inner sep=0pt}}
	
	\SetVertexNoLabel
	
	\Vertex[empty, x=0,y=-1]{1}

	\Vertex[empty, x=0,y=1]{E2}
	\Vertex[empty, x=1,y=1]{E3}
	
	\Vertex[empty, x=-2,y=2]{I1}
	\Vertex[empty, x=0,y=2]{I2}
	
	\SetVertexLabel
	
	\Vertex[LabelOut, L=$\alpha$, x=0,y=0]{A}
	
	\Vertex[LabelOut, L=$\beta$, x=-1,y=1]{E1}

	\Edge(A)(1)
	
	\Edge(E1)(A)
	\Edge(E2)(A)
	\Edge(E3)(A)
	
	\Edge(I1)(E1)
	\Edge(I2)(E1)
\end{tikzpicture} +   \begin{tikzpicture}[scale=0.90, baseline=-1ex]
	\GraphInit[vstyle=Classic]
	
	\tikzset{VertexStyle/.style = {shape = circle,fill = black,minimum size = 1pt,inner sep=0pt}}
	
	\SetVertexNoLabel
	
	\Vertex[empty, x=0,y=-1]{1}
	
	\Vertex[empty, x=-1,y=1]{E1}
	
	\Vertex[empty, x=1,y=1]{E3}
	
	\Vertex[empty, x=-1,y=2]{I1}
	\Vertex[empty, x=1,y=2]{I2}
	
	\SetVertexLabel
	
	\Vertex[LabelOut, L=$\alpha$, x=0,y=0]{A}
	\Vertex[LabelOut, L=$\beta$, x=0,y=1]{E2}

	\Edge(A)(1)
	
	\Edge(E1)(A)
	\Edge(E2)(A)
	\Edge(E3)(A)
	
	\Edge(I1)(E2)
	\Edge(I2)(E2)
\end{tikzpicture} +  \begin{tikzpicture}[scale=0.90, baseline=-1ex]
	\GraphInit[vstyle=Classic]
	
	\tikzset{VertexStyle/.style = {shape = circle,fill = black,minimum size = 1pt,inner sep=0pt}}
	
	\SetVertexNoLabel
	
	\Vertex[empty, x=0,y=-1]{1}
	
	\Vertex[empty, x=-1,y=1]{E1}
	\Vertex[empty, x=0,y=1]{E2}
	
	\Vertex[empty, x=0,y=2]{I1}
	\Vertex[empty, x=2,y=2]{I2}
	
	\SetVertexLabel
	
	\Vertex[LabelOut, L=$\alpha$, x=0,y=0]{A}
	\Vertex[LabelOut, L=$\beta$, x=1,y=1]{E3}

	\Edge(A)(1)
	
	\Edge(E1)(A)
	\Edge(E2)(A)
	\Edge(E3)(A)
	
	\Edge(I1)(E3)
	\Edge(I2)(E3)
\end{tikzpicture} 	\]

To be precise, the incoming edges should be labeled by all possible shuffles (this is similar to Remark \ref{warning}), we refer to \cite[Section 5.4.3]{LV} for details.

 For non-symmetric operads the formula is exactly the sum of partial compositions: \[\alpha \star \beta = \sum\limits_{i=1}^n \alpha \circ_i \beta \]
\end{Def}

It should not be surprising that one can construct pre-Lie algebras out of ``tree-shaped compositions''. Indeed, the pre-Lie operad is isomorphic to the operad of rooted trees \cite{CL01}.

	\begin{Def} 
	\label{Def_MC}
	
	Given a Lie algebra $\G_g$, we define its \defi{set of Maurer--Cartan elements} of $\G_g$ to be:
	\[ \mathrm{MC}(\G_g) := \left\lbrace x \in \G_g_1 \vert \ dx + \frac{1}{2}[x,x]=0 \right\rbrace \]
	
	Moreover, $\mathrm{MC}$ defines a functor from Lie algebras to sets. 
	
\end{Def}

\begin{Def}[Twisting Morphisms]\label{Def_TwistingMorphisms}
	Given a cooperad $\_C$ and an operad $\_P$, the set of \defi{twisting morphisms} is the set of Maurer--Cartan elements in $\tx{Tot}(\tx{Conv}(\_C, \_P))$:
	\[\tx{Tw}(\_C, \_P) := \mathrm{MC}(\tx{Tot}(\tx{Conv}(\_C, \_P)))\]
	It corresponds to the set of morphisms $\alpha : \_C \dashrightarrow \_P$ of symmetric sequences of degree $1$ satisfying the Maurer--Cartan equation, $d\alpha + \alpha \star \alpha =0$. 
	
%
\end{Def}

\begin{Def}[{\cite[Section 6.5.3]{LV}}] \label{Def_BarCobarConstruction}
	
	There is an adjunction called the \defi{Bar-cobar adjunction}: 
	\[ \Omega :\begin{tikzcd} \coOp^{\tx{conil}} \arrow[r, shift left] & \arrow[l, shift left] \Op^{\tx{aug}}
	\end{tikzcd}: \mathbf{B}\]

	between coaugmented conilpotent cooperads and augmented operads defined by $\Omega (\_C) = \left( \T (\overline{\_C}[-1]), d_\_C+d_{\Omega} \right)$ with $\overline{\_C}:= \faktor{\_C}{I}$ for $\Omega$, and with $\mathbf{B}\_P :=  \left( \Tc (\overline{\_P}[1]), d_\_P+d_\mathbf{B} \right)$.
	
\end{Def}

The underlying symmetric sequence of $\Omega (\_C)$ depends only on the underlying symmetric sequence of $\_C$ and the cooperad structure only plays a role for the differential $d_\Omega$. The differential $d_\Omega$ is defined on the generators, i.e. $1$-vertex trees labeled by $c\in \overline{\_C}[-1]$, as the sum over all possible partial cocompositions of $c$, and then extended to an arbitrary tree by derivations, which amounts to sum this formula for all vertices. The ``coassociativity'' property of partial cocomposition, together with the signs produced by the degree shifts, guarantees that $d_\Omega$ squares to zero (see \cite[Section 6.5.2]{LV} for more details).

Similarly, $d_\mathbf{B}$ sums over all possible ways of contracting an edge on a tree using the partial composition of $\_P$, see \cite[Section 6.5.1]{LV}.
 We will see in Sections \ref{Sec_Model Categorical Aspects} and \ref{Sec_KoszulDuality} that this adjunction can provide a cofibrant replacement of an operad and will be used to define homotopy algebras. \\
	
In fact, maps $\Omega \_C \to \_P$ (or equivalently $\_C \to \mathbf{B}\_P$) can be described in terms of twisting morphisms. This point of view will prove to be very important when it comes to the description of the deformations of algebraic structures (see Section \ref{Sec_MaurerCartanSpace}).

\begin{Prop}[Rosetta Stone]\label{Prop_RosettaStone}
	There are natural equivalences: 
	\[ \Hom_{\Op}(\Omega \_C, \_P) \simeq \tx{Tw}(\_C, \_P) \simeq \Hom_{\coOp^{\tx{conil}}}(\_C, \mathbf{B}\_P)  \]
\end{Prop}
	\begin{proof}
		Take $\alpha : \Omega \_C \rightarrow \_P$. Ignoring differentials, this corresponds to a map from the free operad on $\overline{\_C}[-1]$ and is thus equivalent to having a (co)unit preserving map $\tilde{\alpha} : \_C  \dashrightarrow \_P $ of degree $1$ of symmetric sequences.  Moreover $\alpha$ is compatible with the differentials if and only if $\tilde{\alpha}$ is a Maurer--Cartan element. 
		
		Indeed, we have the following equivalences: 
		\[ \begin{split}
		\alpha (d_{\Omega \_C}) =d_{\_P} \circ \alpha  &\Leftrightarrow \alpha (d_{\Delta} + d_\_C  ) =d_{\_P} \circ \alpha \\
		& \Leftrightarrow  (\alpha \circ d_{\Delta} + \alpha \circ d_\_C ) -d_{\_P} \circ \alpha =0\\
		& \Leftrightarrow \alpha \star \alpha + \partial \alpha = 0 \\
		& \Leftrightarrow \frac{1}{2}[\alpha, \alpha] + \partial \alpha = 0
		\end{split}\]
	\end{proof}

\begin{RQ}[A remark on augmentations]\label{rem:A remark on augmentations}
		
	In Definition \ref{Def_BarCobarConstruction} we consider (co)operads that are (co)augmented which forces us to remove the (co)unit. Instead, we could have worked from the beginning with the equivalent categories of non-(co)unital (co)operads, but this makes the descriptions of (co)algebras, namely the (co)free one, more complicated. 
		
While we will mostly hide such details under the rug, in practice this means that strictly speaking for propositions such as the Rosetta Stone to hold, one should require the (co)augmentation ideals the twisting morphisms to be maps compatible with the (co)augmentation, should they exist.
\end{RQ}

	\subsection{Model Categorical Aspects}\
	\label{Sec_Model Categorical Aspects}

\medskip

Given an operad $\_P$, we are interested in understanding the category of $\_P$-algebras up to homotopy, i.e. together with the notion of quasi-isomorphisms as ``weak equivalences''. The theory of model categories gives us a way to handle these weak equivalences.
In this section we will see that algebras over operads, and operads themselves form model categories.
\subsubsection{Model categories and homotopical algebra} \ \label{sec:model-categories-and-homotopical-algebra}

\medskip

Model categories are a tool introduced by Quillen in \cite{Qu67} in order to study homotopy theory. The main problem is to understand objects of a category up to a notion of weak equivalences which are typically non-invertible morphisms. 

A naive way to proceed is to formally invert all the weak equivalences (see \cite[Definition 1.2.1]{Ho07}). For a category $\_C$ with a class of morphisms which we call weak equivalences $\_W$, this formal localization will be denoted $\tx{h}\_C := \_C [\_W^{-1}]$. But in general, descriptions of $\_C [\_W^{-1}]$ are very difficult to deal with. The category $\tx{h}\_C$ can however be described via a universal property: There is a natural functor $\_C \to \tx{h}\_C$ such that any functor $\_C \to \_D$ sending morphisms in $\_W$ to isomorphisms in $\_D$ factors uniquely through $\tx{h}\_C$. This universal property characterizes $\tx{h}\_C$ (see \cite[Lemma 1.2.2]{Ho07}) up to  equivalence of categories. 
Model structures are a tool to get workable descriptions of $\tx{h}\_C$. In our framework, we will consider $\_C$ a category having all small limits and small colimits which, besides a class of \defi{weak equivalences} is equipped with two other classes of morphisms called \defi{fibrations} and \defi{cofibrations} and satisfying certain axioms making $\_C$ a \defi{model category}. For a detailed account we refer to \cite{Ho07}. 
We use the phrase \emph{trivial (co)fibration} to refer to a (co)fibration which is also a weak equivalence. Let us present the main features of model categories:
\begin{itemize}
	\item Any morphism $A \to B$ can be factorized both as a cofibration followed by a trivial fibration or as a fibration followed by a trivial cofibration: 
	\[ \begin{tikzcd}
	A \arrow[r, "\tx{cof.}"] & B' \arrow[r,"\tx{triv. fib.}"]& B
	\end{tikzcd} \]
	\[ \begin{tikzcd}
	A \arrow[r, "\tx{triv. cof.}"] & A' \arrow[r,"\tx{fib.}"]&B
	\end{tikzcd} \]  
	\item An object $A$ is called \defi{fibrant} (resp. \defi{cofibrant}) if the map to (resp. from) the terminal (resp. initial) object is a fibration (resp. cofibration). The factorization axioms ensure that all objects are weakly equivalent to a fibrant-cofibrant object (such an object is called a fibrant-cofibrant replacement). 
	
\item {The class of all cofibrations (respectively fibrations) is completely determined by the classes of weak equivalences and  fibrations (respectively weak equivalences and  cofibrations) (see \cite[Lemma 1.1.10]{Ho07}).  For example, a morphism $f : A \to B$ is a cofibration if and only if it has the \emph{left lifting property} with respect all trivial fibrations. This means that for all trivial fibrations $g : A' \to B'$ and all commutative squares: }
\[ \begin{tikzcd}
A\arrow[d, "f"] \arrow[r] & A' \arrow[d, "g"] \\
B \arrow[r] \arrow[ur, dashed, "h"] & B' 
\end{tikzcd}\]
there exists a lift $h : B \dashrightarrow A'$ making the diagram commute. In particular, in many cases in this survey we will describe a model structure by describing the class of weak equivalences and fibrations, knowing that cofibrations are determined by them, even though they are usually more difficult to describe.

Note that similarly, fibrations and trivial fibrations are determined by trivial cofibrations and cofibrations via the right lifting property.
\end{itemize}

A priori, a weak equivalence may not have an inverse or even a quasi-inverse\footnote{A quasi-inverse (or homotopy inverse) is a map in the opposite direction inducing an inverse in the homotopy category.}. This means that to define the notion of ``being weak equivalent'', the only sensible to do is to ask for two objects, $A$ and $B$, to be weakly equivalent if there is a zig-zag of weak equivalences connecting them:

\[ \begin{tikzcd}
 &\arrow[dl] \arrow[dr] A_1 && \arrow[dl] \arrow[dr] \cdots  && \arrow[dl] \arrow[dr] A_n& \\
 A & & \cdots & & \cdots && B
\end{tikzcd} \]

We denote this equivalence by $A \sim B$.
On the other hand, {a consequence of the axioms of a model category is that} a weak equivalence between objects which are both fibrant and cofibrant admits a ``homotopy-inverse'' which makes ``being weakly equivalent'' a simpler equivalence relation, and it means that we have a nicer description of the homotopy category, $\faktor{\_C^\mathrm{cf}}{\sim} \simeq \tx{h}\_C$ (\cite[Theorem 1.2.10]{Ho07}) where $\_C^\mathrm{cf}$ denote the full sub-category of $\_C$ given by objects that are both fibrant and cofibrant. 

The most basic example in this survey is the model structure on cochain complexes.
\begin{Th}[{\cite[Theorem 2.3.11]{Ho07}}]\label{Th_ModelStructureModR}
	Let $R$ be a ring. There is a model structure on $\Mod_R$ such that a map is:
	\begin{itemize}
		\item a weak-equivalence if it is a quasi-isomorphism.
		\item a fibration if it is degree-wise surjective.
	\end{itemize}
\end{Th}

In $\chk$, all objects are both fibrant and cofibrant (because all objects are projective) and therefore all quasi-isomorphisms are quasi-invertible. Later on, we will transfer this model structure into other categories (such as the category of algebras over an operad) which \emph{will not} satisfy this property. \\

Even though fibrant and cofibrant objects have good homotopical properties, they can be quite big and complicated. For an arbitrary object in a model category, we may want to consider a class of objects that are \defi{minimal}, in which any weak equivalence, $f : A \to B$, between two minimal objects must be an isomorphism. For $\chk$ we will consider the class of objects given by the cohomologies of the complexes, that is the minimal objects are graded complexes with no differentials. 

\begin{Def}\label{Def_FormalityGeneral}
	Take $\_C$ a  model category and a class of minimal objects. Then we say that
 an object $A$ is \defi{formal} if there is a weak equivalence between $A$ and a minimal object $\tilde{A}$.
\end{Def}

We did not make the class of minimal objects explicit in that definition since we will always work with model categories where the class of minimal objects is clear.  

\begin{Prop}\label{Prop_FormalityOfCHKAndDeformationRetractToCohomology}
Every object in $\chk$ is formal. Moreover, every cochain complex has a \emph{deformation retract} to its cohomology, i.e. there exists a diagram
	\[\begin{tikzcd}
	\arrow[loop left, "h"] (A, d_A) \arrow[r, "p", shift left] & \arrow[l, shift left, "i"] (H(A), 0)
	\end{tikzcd}\] 
 such that $pi =\tx{id}_{H(A)}$ and $h$ is a \defi{homotopy} between $\tx{Id}_A $ and $i \circ p$ of degree $-1$, meaning that $\tx{Id}_A - i p = d_A h - h d_A$.
\end{Prop}

\begin{proof}
Using the fact that $k$ is a field we can non-canonically decompose $A$ as a direct sum $A \cong B[1] \oplus \underbrace{B\oplus H(A)}_{\ker d_A}$, where the only non-zero piece of the differential is the ``identity'' $\mathrm{shift}\colon B[1]\to B$.
The maps $p,i$  are defined in the obvious way and $h\coloneqq \mathrm{shift}^{-1}$.
\end{proof}

The natural notion of morphisms between different model categories is the notion of \defi{Quillen adjunctions}. Given $\_C$ and $\_D$ two model categories, a Quillen adjunction is an adjunction:
\[ \begin{tikzcd}
L : \_C \arrow[r, shift left] & \arrow[l, shift left] \_D : R
\end{tikzcd}\]
such that either $L$ preserves cofibrations and trivial cofibrations or equivalently $R$ preserves fibrations and trivial fibrations. This condition is enough to ensure that $L$ and $R$ induce an adjunction between the homotopy categories. Moreover a Quillen adjunction is called a \defi{Quillen equivalence} if it induces an equivalence between the homotopy categories. We refer to \cite[Corollary 1.3.16]{Ho07} for a more practical description of Quillen equivalences. \\

Given an ordinary adjunction:
\[ \begin{tikzcd}
L : \_C \arrow[r, shift left] & \arrow[l, shift left] \_D : R
\end{tikzcd}\]
{
such that $\_D$ (respectively $\_C$) is a model category. We ask whether we can \emph{transfer} the model structure on $\_C$ (respectively $\_D$) such that the adjunction is Quillen. To do so, we define the following classes on $\_C$ (respectively $\_D$):
\begin{itemize}
	\item $f$ is a cofibration in $\_C$ if and only if $L(f)$ is a cofibration in $\_D$ (respectively  $f$ is a fibration in $\_D$ if and only if $R(f)$ is a fibration in $\_C$).

	\item $f$ is a weak equivalence in $\_C$ if and only if $L(f)$ is a weak equivalence in $\_D$ (respectively  $f$ is a weak equivalence in $\_D$ if and only if $R(f)$ is a weak equivalence in $\_C$).
	\item The fibrations are determined by the lifting properties.
\end{itemize} 
If these classes define a model structure (which might not always be the case), then the model structure on $\_C$ is called the \defi{left transferred model structure} (respectively the model structure on $\_D$ is called the \defi{right transferred model structure}). Note that for these choices of classes of maps, the adjunction is automatically Quillen.}\\

Even given a Quillen adjunction $L \dashv R$  the functors $L$ and $R$ might not have good homotopical properties (such as preserving  weak equivalences). However, the properties of a Quillen adjunction ensure that these functors can be ``derived'' to a new adjunction
\[ \begin{tikzcd}
\Ll(L) : \_C \arrow[r, shift left] & \arrow[l, shift left] \_D : \Rr(R)
\end{tikzcd}\]
where $\Ll(L)$ is the \defi{left derived functor of $L$} and  $\Rr(R)$ is the \defi{right derived functor of $R$}. These derived functors preserve weak-equivalences and are obtained by precomposing $L$ and $R$ by an appropriate weakly equivalent replacement (a cofibrant replacement for $L$ and a fibrant replacement for $R$). In the rest of this survey, whenever a functorial construction is ``derived'' or when taking a ``homotopy'' pullback or pushout, the reader may keep in mind that it is the classical construction applied to an appropriate replacement. \\

Although model structures are very convenient tools to handle weak equivalences, a model structure is not always available. In such situations, namely those coming up in Section \ref{Sec_Operadic Deformation Theory}, we need to use $\infty$-categories (see \cite{Lu09}).

\subsubsection{Model structures for (co)algebras}\

\medskip

\begin{Th}[{\cite[Theorem 4.1.1]{Hi97}}]\label{thm:hinich model str}
		
	For an operad $\_P \in \Op$, consider the adjunction of Proposition \ref{Prop_FreeAlgebraOverAnOperad}, 
	\[ \_P(-): \begin{tikzcd}
	 \chk \arrow[r, shift left] & \arrow[l, shift left] \mathbf{Alg}_\_P 
	\end{tikzcd} :(-)^\sharp  \]
	
	Then we can right transfer the model structure on $\chk$ (Theorem \ref{Th_ModelStructureModR}) to a model structure on $\mathbf{Alg}_\_P$ such that a map $f$ is weak equivalence (respectively a fibration) in $\mathbf{Alg}_\_O (\chk)$ if and only if $f^\sharp$ is a quasi-isomorphism (respectively a fibration). With this model structure this adjunction is Quillen.   	
\end{Th}

\begin{RQ}\label{RQ_NonFormalityPAlgebra}
Let us note that if $\_P$ has no differential, then $H(A)$ naturally has a $\_P$-algebra structure.
	While every cochain complex is formal, we can ask whether a $\_P$-algebra is formal (Definition \ref{Def_FormalityGeneral}), that is, do we have a zig-zag of quasi-isomorphisms from $A$ to its homology \emph{in the category of $\_P$-algebras}.  
			It turns out that the answer is that not all $\_P$-algebras are formal. 
		 On the other hand, we will see that the homotopy transfer theorem (Section \ref{Sec_HTT}) will provide a \emph{homotopy} $\_P$-algebra structure on $H(A)$, which in a way quantifies the defect of formality to be satisfied. This extended structure on $H(A)$ will be weakly equivalent to $A$.

\end{RQ}


The model structures on algebras over operads are compatible with change of operads in the following sense: Recall from Proposition \ref{Prop_NaturalityAlgebraAdjunction} that a morphism of operads $f\colon \_P \to \_Q$ induces an adjunction:
	\[  \begin{tikzcd}
	f_! : \mathbf{Alg}_\_P \arrow[r, shift left] & \arrow[l, shift left] \mathbf{Alg}_\_Q : f^*
\end{tikzcd} \] 

\begin{Th}[{\cite[Theorem 4.6.4]{Hi97}}]\label{thm:q.i. operads}
	
	The adjunction above is Quillen. Furthermore, if $f\colon \_P \to \_Q$ is a quasi-isomorphism, this adjunction is a Quillen equivalence.
\end{Th}

The first assertion follows from the fact that the right adjoint preserves the underlying complex. Therefore it sends fibrations to fibrations and weak equivalences to weak equivalences.

\begin{War}\label{war:abelianization}
While the right adjoint will always preserve quasi-isomorphisms, the left adjoint has no reason to do so.

As an example, take the canonical map $f\colon \mathbf{Ass} \to \mathbf{Com}$. The induction $f_!$ takes an associative algebra to its abelianisation. Let $A$ be the quasi-free\footnote{Free as a graded algebra, if we forget the differential.} associative algebra,
\[
A= T_k(x,y,z), dx=dy=0, dz=xy-yx.
\]
The algebra $A$ is quasi-isomorphic to its own homology $H(A) = \Sym_k(x,y)$. But the (homology of the) abelianization of $A$ is $\Sym_k(x,y,z)$.
\end{War}

\subsubsection{Model structure on operads}\

\medskip

In this Section, we introduce a model structure on operads. This model structure is obtained via transfer from the model structure on $\symseq$. {This discussion would also extend when replacing $\chk$ by a ``good enough'' model  category such as $\mathbf{sSet}$ or $\Mod_A$ for example (see \cite{BM03}). }

%

\begin{Th}[{\cite[Theorem 3.2]{BM03}}] \label{Th_ModelStructureOnOperads}

	There is a a cofibrantly generated model structure on the category $\Op$ such that a  morphism $\_P \rightarrow \_Q$ is a weak equivalence (respectively fibration) if for all $n \in \Nn$ the maps $\_P(n) \rightarrow \_Q(n)$ are weak-equivalences (respectively fibrations) in the projective model structure on $\chk$.
\end{Th}

\begin{RQ}
	The free operad functor is left adjoint adjoint to the forgetful functor $\Op \rightarrow \symseq$ where the model structure $\symseq$ is the projective model structure on the functor category $\tx{Fun}\left( \mathbb N^\sim, \chk \right)$, where fibrations and weak-equivalences are defined object-wise. This forgetful functor clearly preserves fibrations and trivial fibrations which makes it a right Quillen functor. In fact the model structure on $\Op$ is left transferred from the model structure on $\symseq$ via the free-forget adjunction. 
\end{RQ}

\begin{RQ} The original version on Theorem 3.2 in \cite{BM03} is a more general version of this theorem giving a model structure on the category of operads valued in any  closed model category (with some additional technical assumptions).  
\end{RQ}

\begin{RQ}\label{RQ_ModelStructureOnCooperad}
	There is no natural model structure on cooperads having weak equivalences given by all quasi-isomorphisms and making the bar-cobar adjunction Quillen. This comes from the fact that the cobar construction does not preserve quasi-isomorphisms. 
\end{RQ}

Despite this last remark, the bar-cobar adjunction is well behaved with respect to the model structure on operads. 

\begin{Th}[{\cite[Theorem 6.6.3]{LV}}] \label{Th_BarCobarUnitCounitQI}

	Both the unit $\epsilon : \Omega \mathbf{B} \_P \rightarrow \_P$ and the counit $\eta : \_C \rightarrow \mathbf{B} \Omega \_C$ of the bar-cobar adjunction are quasi-isomorphisms. 
\end{Th}
\begin{proof}[Sketch of proof]
Basis elements of $\Omega \mathbf B \_P$ can be seen as trees whose vertices are themselves (``inner'') trees whose vertices are labeled by $\_P$. Filtering by the number of inner edges we recover at the level of the associated graded only the piece of the differential corresponding to the one from $\_P$ and a second one making an inner edges into an outer edge. 
One can check that the associated graded retracts into $\_P$ by constructing a homotopy that makes an outer edge into an inner edge.
A similar argument shows that $\_C \to \mathbf{B} \Omega \_C$ is a quasi-isomorphism.
\end{proof}

\begin{Prop}[{\cite[Proposition 6.5.3]{LV}}]\label{Prop_BPreservesQI} 	
	The functor $\mathbf{B}$ preserve quasi-isomorphisms.
\end{Prop}

In general $\Omega$ will not preserve all quasi-isomorphisms. It preserves quasi-isomorphisms between non-negatively graded cooperads such that $\_C(0) = 0$ and $\_C(1) = k$ (see \cite[Proposition 6.5.6]{LV}). 

\begin{Def} \label{Def_KoszulTwistingMorphisms}\ 
	\begin{itemize}
		\item A twisting morphism (Definition \ref{Def_TwistingMorphisms}) $\alpha : \_C \dashrightarrow \_P$ is said to be \defi{Koszul} if the induced map $\Omega \_C \rightarrow \_P$ is a quasi-isomorphism. The full sub-category of $\mathbf{Tw}$ generated by Koszul twisting morphisms will be denoted by $\mathbf{Kos}$.
		\item A twisting morphism $\alpha : \_C \dashrightarrow \_P$ is said to be \defi{weakly Koszul} if the induced map $ \_C \rightarrow \mathbf{B}\_P$ is a quasi-isomorphism.
		\item The counit morphism $\Omega \mathbf{B} \_P \rightarrow \_P$ is a quasi-isomorphism (see Theorem \ref{Th_BarCobarUnitCounitQI}) and induces the \defi{universal twisting morphism} $ \mathbf{B}\_P \rightarrow \_P \in \mathbf{Tw}$ (see Section \ref{sec:algebraic-bar-cobar-adjunction}).
	\end{itemize}
\end{Def}

Since the bar construction preserves quasi-isomorphisms (see Proposition \ref{Prop_BPreservesQI}), Koszul morphisms induce a quasi-isomorphism $\mathbf{B}\Omega \_C \rightarrow \mathbf{B}\_P $ and are therefore weakly Koszul. There is also a version of the bar-cobar adjunction between (co)algebras associated to a given twisting morphism.  

\begin{Prop}\label{Prop_BarCoBarCofibrantReplacement}
	If $\_C \in \coOp$ satisfies $\_C(0)=0$ and $\_C(1)= k$ then $\Omega \_C$ is cofibrant. In particular, if $\_P \in \Op$ satisfies the same conditions, the unit map $\Omega \mathbf{B} \_P \rightarrow \_P$ is a cofibrant resolution of $\_P$.   
\end{Prop}

\subsection{Koszul Duality} \label{Sec_KoszulDuality}\

\medskip

{In Theorem \ref{thm:q.i. operads} we saw that from the homotopical point of view it is the same to consider algebras over an operad $\_P$ or over a cofibrant replacement $\_Q$  of such an operad. If $\_P$ has no differential, algebras over $\_Q$ have a typically more complicated structure. On the other hand, algebras over cofibrant operads, especially those of the form $\Omega \_C$, have very nice properties that we will discuss in the next section. One of such properties was foreshadowed in the Rosetta Stone \ref{Prop_RosettaStone}, where we saw that maps out of $\Omega \_C$ amount to solving the Maurer--Cartan equation. }

This section is all about finding a ``simple'' cofibrant replacement of operads using the methods of Koszul duality, originally introduced by Ginzburg and Kapranov \cite{GinzburgKapranov}. This is inspired in the more classical Koszul duality for associative algebras, introduced by Priddy \cite{Pr70} as a tool to produce small resolutions for a large class of algebras.
We will assume in this Section that $\_P(0)= 0$ and $\_P(1)= k$.

\begin{Def}	
	Given an operad $\_P$, we say that a cooperad $\_C$ is the \defi{generalized Koszul dual cooperad} of $\_P$ if there is a quasi-isomorphism $\Omega \_C \rightarrow \_P$. In particular it is given by a Koszul twisting morphism $\_C \dashrightarrow \_P$ (Definition \ref{Def_KoszulTwistingMorphisms}). 
\end{Def}

\begin{RQ}\label{RQ_ExistenceUniquenessOfGeneralisedKoszulDual}
	If $\_P(0)= 0$ and $\_P(1)= k$, a generalised Koszul dual always exists since, thanks to Proposition \ref{Prop_BarCoBarCofibrantReplacement}, $\mathbf{B}\_P$ is a generalized Koszul dual of $\_P$ (given by the universal twisting morphism from Definition \ref{Def_KoszulTwistingMorphisms}).  
	
	Moreover, the generalized Koszul dual is  unique up to weak equivalence. Given the cofibrant resolution $\Omega \_C \rightarrow \_P$, we can always find the following sequence of quasi-isomorphisms (thanks to Theorem \ref{Th_BarCobarUnitCounitQI} and Proposition \ref{Prop_BPreservesQI}): 
	\[ \begin{tikzcd}
	\_C \arrow[r, "\tx{counit}"] & \mathbf{B}\Omega \_C \arrow[r] & \mathbf{B}\_P
	\end{tikzcd} \]
	
\end{RQ}

    The main problem is that $\mathbf{B}\_P$ is typically a very big cooperad. We will therefore be interested in a class of operads who admit Koszul dual cooperads which are easy to compute. This is the class of Koszul operads, and these operads are in particular quadratic operads (Definition \ref{def:quadratic}).\\
    
    Using the notation from Constructions \ref{Cons_Operad} and \ref{Cons_Cooperads}, notice that the degree shift isomorphism $\T^{(1)} E[1] \to \T^{(1)}E$ induces a canonical twisting morphism  $\kappa : \_C \left( E[1], R[2] \right) {\dashrightarrow \_P(E,R)}$ (see \cite[Section 7.4.1]{LV}).

\begin{Def}\label{Def_KoszulOperad}
	
	A quadratic operad $\_P = \_P(E,R)$ is \defi{Koszul} if $$\Omega \_C \left( E[1], R[2] \right) \rightarrow \_P$$ is a quasi-isomorphism. Equivalently, $\_P$ is Koszul if  $\_C \left( E[1], R[2] \right)$ is a generalized Koszul dual of $\_P$.  In that case, we define the \defi{Koszul dual cooperad} of $\_P$ to be $\_P^\antishriek := \_C \left( E[1], R[2] \right)$.
\end{Def}

Classically, Koszul duality is a duality between operads and operads (rather than operads and cooperads). To recover such a duality we could in principle simply define the Koszul dual operad of a Koszul operad $\_P$ to be $(\_P^\antishriek)^\vee$, which is equivalent data if the generators are arity-wise finite dimensional.

In practice, since many important Koszul operads (for example $\mathbf{Ass}$, $\mathbf{Com}$ and $\mathbf{Lie}$) are binary and concentrated in degree zero, one typically introduces a degree shift so that their Koszul duals are also concentrated in degree zero. A way to achieve this is to tensor it in each arity with the endomorphisms operad of $k[1]$, which has the effect of raising the degree of the arity $n$ piece by $n-1$:

$$\_Q \lbrace -n \rbrace \coloneqq \_Q \otimes \mathbf{End}_{k[n]}$$

This is mostly for psychological reasons, since the categories of algebras are equivalent via the map shifting the degree of the underlying vector space.

\begin{Def}
	\label{Def_KoszulDualOperad}

	Let $\_P = \_P(E,R)$ be a Koszul operad. The \defi{ operad Koszul dual} to $\_P$ is:
	\[\_P^! \coloneqq 
	(\_P^\antishriek)^\vee\lbrace -1\rbrace
	  \]
	If $\_P$ is binary, this takes the form
	$\_P = \_P \left(\mathrm{sgn} \otimes E^\vee, R^{\perp}\right)$,
	where $\mathrm{sgn}$ denotes the sign representation of $\Sigma_2$ and $R^{\perp}$ denotes the orthogonal complement $R^{\perp}\cong \left(\faktor{\T^{(2)} E}{R}\right)^\vee$.
\end{Def}

\begin{Prop} 
	
	We have that $(\_P^!)^! = \_P$ for $\_P$ an arity-wise dualizable Koszul operad. Furthermore, $\_P$ is Koszul if and only if $\_P^!$ is Koszul.
\end{Prop}
\begin{proof}
	The first part is immediate from the definition. For the second statement, if $\Omega \_P^\antishriek \to \_P$ is a quasi-isomorphism, taking $\mathbf B$ on both sides yields a quasi-isomorphism $\_P^\antishriek \to \mathbf B \_P$. Dualizing and shifting gives the result.
\end{proof}

While not all quadratic operads are Koszul, this should be taken as a very common property for all operads appearing in practice. In particular, we have:

\begin{Th}\label{Th_ExampleKoszulDuality}
	
The operads $\mathbf{Ass},\mathbf{Com},\mathbf{Lie},\mathbf{preLie}$ and $\mathbf{Perm}$ are Koszul.
Furthermore, we have the following Koszul dualities: 
\[ \mathbf{Ass}^! = \mathbf{Ass} \qquad \mathbf{Com}^! = \mathbf{Lie} \qquad \mathbf{preLie}^! = \mathbf{Perm}\]

\end{Th}

When $\_P$ is a Koszul operad with no differential, it is easy to convince oneself that $\Omega \_P^\antishriek$ is the smallest possible quasi-free resolution of $\_P$. This motivates the following.

\begin{Conv}\label{Not_PinftyOperad}
	From here onwards, we will tacitly suppose that a Koszul operad  $\_P$ has no differential, even if it might be concentrated in non-zero degrees. We will use the notation
	$$\_P_\infty \coloneqq \Omega \_P^\antishriek.$$
 Algebras over $\_P_\infty$ are sometimes called \emph{homotopy $\_P$-algebras}. For the associative, Lie and commutative operads, we will rather use the classical notations $A_\infty$, $L_\infty$ and $C_\infty$ respectively.

\end{Conv}

\subsection{Bar-Cobar Adjunction for Algebras} \label{sec:algebraic-bar-cobar-adjunction}

\medskip

\begin{Def}[{\cite[Section 11.3]{LV}}] \label{Def_BarCobarAlgebra}
	
	Given a twisting morphism $\alpha : \_C \dashrightarrow \_P$, we can define a \defi{bar-cobar adjunction associated to $\alpha$}:
	\[ \begin{tikzcd}
	\Omega_\alpha : \mathbf{coAlg}_\_C^{\tx{conil}} \arrow[r, shift left] & \arrow[l, shift left] \mathbf{Alg}_\_P : \mathbf{B}_\alpha
	\end{tikzcd} \]
	
	where $\Omega_\alpha$ is defined as the quasi-free $\_P$-algebra, given on objects by $\Omega_\alpha(C) = (\_P\circ C,d_C+d_{\_P}+d_{\Omega})$ together with the cobar differential extending the differential on $\_P$, on $C$ and with $d_\Omega$ the unique differential extending the following composition:
	\[ \begin{tikzcd}
	C \arrow[r, "\Delta_C"] & \_C \circ C \arrow[r, "\alpha \circ \id"] & \_P \circ C
	\end{tikzcd} \]
	$\mathbf{B}_\alpha$ is the quasi-cofree $\_C$-coalgebra $\mathbf{B}_\alpha(A) = (\_C\circ A,d_A+d_{\_C}+d_{\mathbf{B}})$ with  $d_\mathbf{B}$ the unique codifferential on $\_C \circ A$ extending the following composition:
	\[ \begin{tikzcd}
	\_C \circ A \arrow[r, "\alpha \circ \id"] & \_P \circ A \arrow[r, "\mu_A"] & A
	\end{tikzcd}\]
		
\end{Def}

\begin{Ex}\
	
	\begin{itemize}
		\item For $\G_g$ a Lie algebra, the bar construction associated to the canonical twisting morphism gives the Chevalley--Eilenberg chain complex up to a degree shift (which is in fact a shifted cocommutative coalgebra): \[\mathbf{B}_\kappa \G_g = \left( \Sym_k^{\geq 1} \left(\G_g[1]\right)[-1], \delta_{\mathrm{CE}} \right)= \mathrm{CE}_{*}(\G_g). \]
		\item For $A$ an associative algebra, we obtain the Hochschild chain complex with   coefficients in itself (a shifted coassociative coalgebra): 
		\[ \mathbf{B}_\kappa A = \bigoplus_{n\geq 1} A^{\otimes n}[n-1]=\mathrm{CH}_{*} (A). \]
	\end{itemize}
\end{Ex}

\begin{RQ}
	There is also a notion of twisting morphism from a $\_C$-coalgebra $C$ to a $\_P$-algebra $A$, given by \cite[Definition 11.1.1]{LV}. This leads to a version of the Rosetta Stone for algebras (see \cite[Proposition 11.3.1]{LV}). 
\end{RQ}

\begin{RQ}\label{RQ_NaturalityAlgebraicBarCobar}
	The bar-cobar adjunction is natural with respect to morphisms of twisting morphism (Definition \ref{Def_TwistingMorphisms}) in the following sense.
	Given a morphism of twisting morphisms i.e., a commutative square: 
	\[ \begin{tikzcd}
	\_C \arrow[r,dashed, "\alpha"] \arrow[d,"f"]& \_P \arrow[d,"g"]\\
	\_D \arrow[r, dashed, "\beta"] & \_Q
	\end{tikzcd}\]    
	we have the following commutative diagrams {of left and right adjoint functors}.
	
\begin{center}
		\begin{tikzcd}
	\mathbf{coAlg}_{\_C} & \mathbf{Alg}_{\_P} \arrow[l, "\mathbf{B}_\alpha"]  \\
	\mathbf{coAlg}_{\_D} \arrow[u, "f^!"]  & \mathbf{Alg}_{\_Q} \arrow[l, "\mathbf{B}_\beta"] \arrow[u, "g_*"] 
	\end{tikzcd}\quad \quad
\begin{tikzcd}
	\mathbf{coAlg}_{\_C}\arrow[d, "f^*"] \arrow[r, "\Omega_\alpha"]& \mathbf{Alg}_{\_P}\arrow[d, "g_!"]    \\
	\mathbf{coAlg}_{\_D} \arrow[r, "\Omega_\beta"]  & \mathbf{Alg}_{\_Q}  
\end{tikzcd}
\end{center}

If we ignore the differentials, this is just follows from the fact that the composition of adjoints is an adjoint. One can then check that the differentials agree. See \cite[Lemma 3.25]{CT20} for a proof of the commutativity the first diagram (which is equivalent to saying that the diagram of right adjoints is also commutative).
\end{RQ}

	The condition of being Koszul gives extra properties to the adjunction, in particular Theorems 11.3.3 and 11.3.4 in \cite{LV} relate the criteria of being Koszul to having a unit and counit of the bar-cobar adjunction that are quasi-isomorphisms. 

\begin{Prop}[{\cite[Corollary 11.3.5]{LV}}]\label{RQ_UnitCouniQIAlgebraBarcoBar}
	
For $\_P$ as Koszul operad, the counit $\Omega_\kappa \mathbf{B}_\kappa A \rightarrow A$ is a quasi-free resolution of $A$ and in particular $\Omega_\kappa \mathbf{B}_\kappa$ gives a canonical cofibrant replacement on the category of $\_P$-algebras.
Furthermore, if $C$ is a conilpotent  $\_P^\antishriek$-coalgebra, the unit $C \rightarrow \mathbf{B}_\kappa \Omega_\kappa C$ is a quasi-isomorphism.
\end{Prop}

Notice that the bar-cobar adjunction give us a direct connection between algebras over a Koszul operad $\_P$ and algebras over it's Koszul dual $\_P^!$. We may  define the \emph{Koszul dual algebra} of a $\_P$-algebra $A$ to be $\mathfrak D(A)$, where $\mathfrak D$ is given by:
	\begin{align*}
	\mathfrak D \colon	\mathbf{Alg}_\_P \stackrel{\mathbf B_\kappa}{\longrightarrow} \mathbf{coAlg}_{\_P^\antishriek}  \stackrel{\vee}{\to} {\mathbf{Alg}_{\_P^!\{-1\}}}^{\mathrm{op}} \stackrel{A\mapsto A[-1]}{\longrightarrow}{\mathbf{Alg}_{\_P^!}}^{\mathrm{op}}
	\end{align*}

	\begin{Ex}\label{Ex:CE and HH cochain cx}\
		
		\begin{itemize}
			\item 	The	Koszul dual commutative algebra of a Lie algebra $\mathfrak g$ is
			 \[\mathfrak{D} (\G_g) = \left( \Sym_k^{\geq 1} \left(\left(\G_g[1]\right)^\vee\right), d \right), \]
			where the differential $d$ is the unique derivation extending the dual of the Lie bracket $\ell^\vee \colon \G_g^\vee\to \G_g^\vee\otimes \G_g^\vee$.
		When $\G_g$ is a finite dimensional  Lie algebra in degree $0$, up to to the term $\Sym_k^0(\G_g^\vee)=0$, this is precisely the classical Chevalley--Eilenberg cochain complex of $\G_g$ with coefficients in the trivial module $k$, usually written: \[\mathrm{CE}^*(\G_g)= \left( \Hom(\Lambda^*\G_g,k) , d_{\mathrm{CE}} \right) \]
			
			\item Similarly, in the finite dimensional and degree $0$ case and up to a term $k$, the dual algebra of an associative algebra $A$ is its Hochschild cochain complex with coefficients in the trivial module $k$:
			\[\mathfrak{D} (A) = \mathrm{CH}^*(A)= \left( \Hom(A[1]^{\otimes *},k), d_{\mathrm{Hoch}} \right) \]

		\end{itemize}

	\end{Ex}
	
	\begin{War}\label{war:dual is not dual}
		Suppose the space of generators of $\_P$ is finite dimensional.
Despite the terminology ``dual'' there is no direct natural transformation between $\mathfrak D^2$ and $\id_{\mathbf{Alg}_{\_P}}$ in any direction. However, if $A$ is a degreewise finite dimensional $\_P$-algebra, we have ${\mathbf B (A)}^\vee = \Omega(A^\vee)$, and in that case, there is a map $\mathfrak D^2 A \to A$ which is equivalent to the bar-cobar resolution of $A$. In general,  $\mathfrak D^2 A$ and  $A$ might not even have isomorphic homology.
	\end{War}

The bar-cobar adjunction has nice properties with respect to a model structure on coalgebras.

\begin{Th}[{\cite[Theorem 3.11]{DGH16}} or  {\cite[Theorem 2.1]{Val20}}]
	\label{Th_ModelStructureCoalgebra}
	
	Given any twisting morphism $\alpha : \_C \dashrightarrow \_P$, we can define the $\alpha$-model structure on $\mathbf{coAlg}_\_C$ as the model structure obtain via the left transfer along the bar-cobar adjunction for $\_P$-algebras (see Definition \ref{Def_BarCobarAlgebra}): 
	\[ \Omega_\alpha :\begin{tikzcd}
	\mathbf{coAlg}_\_C \arrow[r, shift left] & \arrow[l, shift left] \mathbf{Alg}_\_P  
	\end{tikzcd} : \mathbf{B}_\alpha \]
	
	This construction gives us a Quillen adjunction which is an equivalence if the twisting morphism is Koszul. This adjunction is natural in the choice of $\alpha$, $\_P$ and $\_C$ (see \cite[Remark 3.12]{DGH16}).
\end{Th}

\begin{RQ}
	If we take $\alpha : \_C \rightarrow I$ the augmentation on $\_C$, we obtain a model structure transferred from $\mathbf{Alg}_I = \chk$, where weak equivalences are quasi-isomorphisms. However we will not make use of this model structure, since it is not compatible with the one on $\_P$-algebras. In fact, we will only be interested in the case where $\alpha$ is Koszul.
\end{RQ}

\begin{Cor}\label{Prop:PooAlg==PAlg}
	Given a Koszul operad $\_P$, there are Quillen equivalences: 
	
	\[ \begin{tikzcd}
		\mathbf{coAlg}_{\_P^\antishriek}  \arrow[dr, shift left]  &\\
		&	\mathbf{Alg}_{\_P} \arrow[ul, shift left] \arrow[dl, shift left] \\
		\mathbf{Alg}_{\_P_\infty}  \arrow[ur, shift left] &
	\end{tikzcd}\]
\end{Cor}

\begin{proof}
	The first equivalence follows from Theorem \ref{Th_ModelStructureCoalgebra} and the second one from Theorem \ref{thm:q.i. operads}. 
\end{proof}

\subsubsection{A comment on operads as algebras} \label{sec:a-comment-on-operads-as-algebras}\

\medskip

The reader might wonder what is the relation between the bar-cobar adjunction for algebras and for operads. 

		In a roundabout way, non-symmetric operads can themselves be seen as algebras over a colored operad $\mathcal O^{ns}$ with set of colors $S= \Nn$ (that is, $\mathcal O^{ns}$ is a monoid in $\tx{Fun}(\Nn_{\Nn}, \chk)$, see Remark \ref{RQ_NS and Colored Operads}). The colored operad $\mathcal O^{ns}$ is generated by binary operations consisting of the partial compositions $\circ_i$, see \cite{Va04}. This observation can be used to show that operads form a model category by applying a colored version of Theorem \ref{thm:hinich model str}. There is also an operad $\mathcal O$ encoding symmetric operads which is colored in groupoids, to take properly into account the symmetric group actions, see \cite{Wa22}.
		
		Furthermore, the operad $\mathcal O$ is quadratic and Koszul self dual $\mathcal O^! = \mathcal O$. The (colored enhancement of the) bar-cobar construction from Definition \ref{Def_BarCobarAlgebra} with respect to the canonical $\mathcal O^\antishriek \dashrightarrow \mathcal O$ corresponds precisely to the operadic bar-cobar construction.

	\begin{Ex}
It is a curious observation that given a permutative algebra $A$, we can define an operad $\_P$ where $\_P(n)=A$ for all $n$ and setting all partial compositions to be the multiplication on $A$\footnote{In fact, this produces a non-unital operad, since $A$ is a priori non-unital. This is consistent with the detail hidden under the rug that $\mathcal O$ encodes in fact non-unital operads.}.

Using that $\mathbf{preLie}^! = \mathbf{Perm}$ we can use this observation to recover conceptually (albeit in an overcomplicated way) the pre-Lie structure on the totalization $\tx{Tot}(\_P)$ of an operad as follows:

There is an adjunction between operads and groupoid colored operads
\[ \begin{tikzcd}
						L: \Op \arrow[r, shift left] & \arrow[l, shift left] \mathbf{gcOp} : R 
\end{tikzcd}\]  

in which $L(\_Q)(c_1,\dots, c_n;c_0) = \_Q(n)$ for any choice of colors. The observation above manifests itself as a map of operads $\mathcal O \to L(\mathbf{Perm})$, which, by taking Koszul duals yields a map $L(\mathbf{preLie} ) \to \mathcal O^! = \mathcal O$. 

The data of an operad structure on $\_P$ is equivalent to a map $\mathcal O\to \mathbf{End}_{\_P}$. Composing these maps and applying the adjunction, we obtain a map

\[\mathbf{preLie}\to R(\mathbf{End}_{\_P}), \]
One can show that $ R(\mathbf{End}_{\_P})=\mathbf{End}_{\tx{Tot}(\_P)}$ so the map above recovers precisely the pre-Lie structure on $\tx{Tot}(\_P)$. 

For more details, see \cite[Section 3.3]{CCN20}.
\end{Ex}

\needspace{6\baselineskip}

	\section{Algebraic Deformation Theory}\label{Sec:2}
	
As mentioned in the introduction, it is an old heuristic that deformation problems are in correspondence with Lie algebras. We will not define  precisely for now what we mean by deformation problem, a formalisation of this correspondence will be treated in detail in Section \ref{Sec_Operadic Deformation Theory}. Nevertheless we point out that there is no general explicit procedure to produce a Lie algebra out of a deformation problem. 

In the other direction, though, this heuristic tells us that if, for some reason, we have access to the Lie algebra $\G_g$ controlling a deformation problem, then the set Maurer--Cartan elements of $\G_g$, $\tx{MC}(\G_g)$, corresponds to such deformations. Furthermore, a deformation problem should come with some notion of equivalence between two deformations which corresponds on the Lie algebra side, to an action of the ``gauge group'' or ``exponential group'' of $\G_g$ on $\tx{MC}(\G_g)$.\\ 

The goal of this section is to present how the operadic machinery developed in the previous section give us very efficient tools to produce Lie algebras associated to certain algebraic deformation problems.	{Throughout we will work with a cooperad $\_C$ such that $\_C(0) = 0$ and $\_C(1) = k$ so that $\Omega \_C$ is cofibrant. This condition implies in particular that $\_C$ is conilpotent. }\\
	 
	We will start, in Section \ref{Sec_AlgebraicStructuresUpToHomotopy}, by recalling the definitions of homotopy $\_P$-algebras, $\infty$-morphisms and the homotopy transfer theorem. It turns out that these homotopy algebras provide a good context in which we can deform algebraic structures. In Section \ref{Sec_DeformationofAlgebraicStructures}, we discuss the construction of spaces of deformations as the Maurer--Cartan space of the convolution Lie algebra. In that section, we will define two groupoids, one given by some Maurer--Cartan element together with gauge action between them, and the other given by deformations of algebraic structures together with equivalences of such deformations. The main result, Theorem \ref{Th_DeformationAndDeligneGroupoids}, says that those groupoids are in fact equivalent. Then, in Section \ref{Sec_CommutativeAndAssociativeStructures}, we will focus on the interplay between commutative and associative structures.
	
	\subsection{Algebraic Structures Up to Homotopy}\
	\label{Sec_AlgebraicStructuresUpToHomotopy}
	

\medskip

As we mentioned in Section \ref{Sec_Model Categorical Aspects}, the non-invertibility of quasi-isomorphisms of $\_P$-algebras makes the study of the question of ``$A$ being quasi-isomorphic to $B$'' rather unpleasant since it, in principle, forces us to consider all possible zig-zags of quasi-isomorphisms of algebras between $A$ and $B$. 

On the other hand, the category of $\_P$-algebras forms a model category in which every object is fibrant, so this already reduces that question to the existence of a single quasi-isomorphism $Q(A)\stackrel{\sim}{\to} B$, where $Q(A)$ is any cofibrant replacement of $A$. This is not always a very efficient way to address this question or more generally to study the homotopy category of $\_P$-algebras.\\

{{In this section, we will always assume that $\_P$ is a Koszul operad}}. We will see that the class of morphisms of $\_P$-algebras can be enlarged to an explicit notion of \emph{$\infty$-morphism}, in which $\infty$-quasi-isomorphisms admit quasi-inverses, which gives us great control over the homotopy category. This is summarized by the following theorem whose proof we will sketch at the end of the section.

\begin{Th}\label{Th_MainPropertiesOfInftyMaps}
		
The faithful inclusion from $\_P$-algebras with ordinary morphisms to $\_P$-algebras with $\infty$-morphisms induces an equivalence at the level of the homotopy categories.

In particular, two $\_P$-algebras are weakly equivalent if and only if there exists a direct $\infty$-quasi-isomorphism between them. 
\end{Th}

	\subsubsection{Homotopy algebras and morphisms}\
	\label{Sec_HomotopyAlgebrasAndMorphism}

\medskip

	We will in fact start by enlarging the objects of the category of $\_P$-algebras rather than the morphisms, by going to the category of $\_P_\infty$-algebras, which recovers the classical notions of $A_\infty$- and $L_\infty$-algebras. The real advantages of doing so will only appear in Section 	\ref{Sec_HTT}, but for the moment notice that from the homotopical point of view, nothing is gained or lost thanks to Corollary \ref{Prop:PooAlg==PAlg}. 

	There are several different equivalent ways to define $\_P_\infty$-algebras. Given our presentation, the most natural definition is to define a $\_P_\infty$-structure on $A$ to be a morphism $\_P_\infty := \Omega \_P^\antishriek \rightarrow \mathbf{End}_A $. But using Proposition \ref{Prop_RosettaStone}, we have the equivalences: 
		\[ \Hom_{\Op} \left( \_P_\infty, \mathbf{End}_A \right) \simeq \tx{Tw}(\_P^\antishriek, \mathbf{End}_A) \simeq \Hom_{\coOp^{\tx{conil}}}\left( \_P^\antishriek, \mathbf{B}\mathbf{End}_A \right)\]

	\begin{Lem}\label{Lem_PInftyStructureCodifferential}
		
		A $\_P_\infty$-structure on $A$ is exactly given by a codifferential  on the cofree $\_P^\antishriek$-algebra $\_P^\antishriek (A)$. We denote by $\tx{codiff}(\_P^\antishriek (A))$ the set of such codifferentials. 
	\end{Lem}

By codifferential on $\_P^\antishriek (A)$, we mean a degree $1$ square zero coderivation $D$, such that the underlying map $A\to A$ is just the differential on $A$: 
\[D = d_A + D^{\geq 2}\]
	
\begin{proof}
			This is essentially the content of \cite[Proposition 10.1.11]{LV}. In general a coderivation on $\_P^\antishriek (A)$ is completly determined (because of cofreeness of $\_P^\antishriek (A)$) by a map in $\Hom_{\symseq} \left( \_P^\antishriek, \mathbf{End}_A \right)$ of degree $1$ (see \cite[Proposition 6.3.8]{LV} together with Lemma \ref{Lem_closedmonoidalstructure}). For this map to correspond to a $\_P_\infty$-structure on $A$, it needs to be a twisting morphism via the Rosetta Stone. It turns out that to satisfy the Maurer--Cartan equation in $\Hom_{\symseq} \left( \_P^\antishriek, \mathbf{End}_A \right)$ is equivalent to saying that the associated coderivation on $\_P^\antishriek(A)$ squares to zero.     
\end{proof}

\begin{RQ}
	Notice that for $\_P = \_P(E,R)$ Koszul, the Lie algebra $\tx{Tot}(\mathrm{Conv}(\_P^\antishriek, \mathbf{End}_A))$ is filtered complete, with filtration given by the vertex filtration on $\mathfrak T^c(E)$, which descends to $\_P^\antishriek$. 
	
	This filtration splits as a weight grading by the number of vertices:
		\[\tx{Tot}(\mathrm{Conv}(\_P^\antishriek, \mathbf{End}_A))= \prod_{n\geq 1}\Hom_\Sigma(\_P^{\antishriek(n)}, \mathbf{End}_A))\]
		 A twisting morphism then decomposes in $\phi = \phi_1 + \phi_2 + \cdots$ and the Maurer--Cartan equation becomes: 
\[ - \sum_{\underset{p<n, q<n}{p+q =n} } \phi_k \star \phi_l = \partial \phi_n  \] 

where $\partial$ and $\star$ give the pre-Lie structure of the convolution pre-Lie algebra.
\end{RQ}
	\begin{Ex} \ \label{Ex_InfityAlgebras}
		
		\begin{itemize}

			\item $\mathbf{As}_\infty$ and $A_\infty$-algebras: 
			
			For simplicity reasons let us consider $\mathbf{As}$ be the \emph{non-symmetric} version of the associative operad and let us describe homotopy algebras over it. 
			An $A_\infty$-algebra structure is given by a twisting morphism in $\tx{Tw}(\mathbf{As}^\antishriek, \mathbf{End}_A)$. Forgetting for now the Maurer--Cartan condition, this is given by a set of maps from $\mathbf{As}^\antishriek (n)$ to  $\mathbf{End}_A(n)[-1]$ for each $n\in \Nn$. We have that:  \[ \mathbf{As}^! (n) = (\mathbf{As}^\antishriek (n))^\vee \otimes \mathbf{End}_{k[-1]}(n) \simeq (\mathbf{As}^\antishriek (n))^\vee[n-1]\]
			
			 Therefore we get $\mathbf{As}^\antishriek(n) = (\mathbf{As}^!(n))^\vee[1-n] = \mathbf{As}(n)^\vee[1-n]$ (since $\mathbf{As}^! = \mathbf{As}$). Therefore we have an isomorphism $\mathbf{As}^\antishriek (n)  \cong k[1-n]$ for $n\geq 1$ and such a morphism is given for each $n$ by an element in $\mathbf{End}_A(n)$ of degree $2-n$, that is an element of degree $0$ in $\iHom_{k}\left( A^{\otimes n}, A \right)[n-2]$. This means that for each $n \geq 2$, we get  an $n$-ary operation of degree $2-n$:
			 \[ \mu_n : A^{\otimes n} \rightarrow A \] 
			 satisfying some conditions corresponding to the Maurer--Cartan equation satisfied by the twisting morphism:
			 \[ \partial \alpha + \alpha \star \alpha = 0, \]
			 where $\partial$ is the differential on $\tx{Tw}(\mathbf{As}^\antishriek, \mathbf{End}_A)$ given by $\partial \alpha = \alpha \circ d_{\mathbf{As}^\antishriek} + \partial_A \circ \alpha$. But we have  $\partial_{\mathbf{As}^\antishriek} = 0$ and we get: 
			 \[ \partial_A \circ \alpha + \alpha \star \alpha = 0 \]

			 From the codifferential point of view, $\_P^\antishriek (A)$ is simply the shifted cofree tensor coalgebra on the graded module $A[-1]$ given by the tensor algebra $(T(A[-1]))[1]$. 			 
			 Then the maps $\mu_n$ correspond to a map, $D$, on $\_P^\antishriek (A)$ with
			   $$D= d_A+ \mu_2 + \cdots \colon \_P^\antishriek (A) \rightarrow A$$
			 defining a derivation $D$ on $\_P^\antishriek (A)$. Then, if we write $\mu_1 \coloneqq d_A$, the condition $D^2 = 0$ recovers the classical equations encoding the associativity  up to homotopy of the $A_\infty$-algebras: 
			 \[ \sum_{p+q +r = n} (-1)^{p+qr} \mu_{p+r+1} \circ (\tx{id}^{\otimes p} \otimes \mu_q \otimes \tx{id}^{\otimes r}) = 0, \quad \forall n\geq 1\] 
			    
			\item $\mathbf{Lie}_\infty$ and $L_\infty$-algebras:
			
			Similarly to the associative situation, we can see that $\mathbf{Lie}^\antishriek (n) \cong k[1-n]$ with the sign action of $\Sigma_n$. An $L_\infty$-algebra structure on $A$ is given by a sequence of graded skew-symmetric maps for each $n \geq 2$:
			\[ \ell_n :  A^{\otimes n} \rightarrow A[n-2]\]
			
			The cofree $\mathbf{Lie}^\antishriek$-coalgebra generated by $A$ is given by $(\Sym A[-1])[1]$ together with the natural decomposition product. If $D : \mathbf{Lie}^\antishriek (A) \rightarrow A$ is the degree 1 codifferential associated to the sequence of $\ell_n$, and $d_A$ the differential on $A$, then $D^2 = 0$ is equivalent to the following conditions: 
			\[ \sum_{\underset{p>1, q>1}{p+q =n} } \sum_{\sigma \in \tx{Unsh}(p,q)} \tx{sgn}(\sigma)(-1)^{(p-1)q} (\ell_p \circ \ell_q)^\sigma = \partial_A (\ell_n )\]
		\end{itemize}
	\end{Ex}
	
	Such homotopy algebraic structures, being algebras over an operad $\_P_\infty$ have a natural notion of morphism, which commutes with all the $\_P_\infty$ operations. Such morphisms are sometimes called \emph{strict} morphisms, in opposition with a weaker notion of morphisms called  \emph{$\infty$-morphism}. Even when the $\_P_\infty$ structure arises from a strict $\_P$-algebra structure, this is a novel notion. These new type of morphisms give us a better control of the homotopy theory of $\_P_\infty$-algebras, have very nice invertibility properties and naturally appear in the homotopy transfer theorem (Theorem \ref{Th_HHT}).  

	\begin{Def}\label{Def_InftyMorphisms}
		
		Let $A$ and $B$ be two $\_P_\infty$-algebras given by codifferentials $\left( \_P^\antishriek (A), \Delta_A \right)$ and $\left( \_P^\antishriek (B), \Delta_B \right)$ (see Lemma \ref{Lem_PInftyStructureCodifferential}). An \defi{$\infty$-morphism} from $A$ to $B$, denoted $f: A \rightsquigarrow B$, is defined to be a morphism of $\_P^\antishriek$-coalgebras preserving the codifferential: 
		\[ f: \left( \_P^\antishriek (A), \Delta_A \right) \rightarrow \left( \_P^\antishriek (B), \Delta_B \right) \]
	\end{Def}

Let us unravel a bit the definition above. 	By virtue of being a map into a cofree conilpotent coalgebra, an $\infty$-morphism $f: A \rightsquigarrow B$ is in fact completely determined by the projection $\_P^\antishriek (A) \rightarrow B$ or equivalently, thanks to Lemma \ref{Lem_closedmonoidalstructure}, by a map $\tilde f : \_P^\antishriek \rightarrow \mathbf{End}_B^A$. Decomposing $\tilde f $ by arity, we see that an $\infty$-morphism is determined by a series of maps $A^{\otimes n}\to B$, being enough to specify one such map for each generator of $\_P^\antishriek(n)$ as a $\Sigma_n$-module. 

	We call $f_1\coloneqq \tilde f(I) \colon A \to B$ the \emph{base map} of the $\infty$-morphism. Notice that this is only a morphism of cochain complexes and we interpret the other components of $f$ as the homotopies controlling the failure of $f_1$ to intertwine the $\_P$ or $\_P_\infty$ structures. If $\tilde f$ vanishes in arities $\geq 2$, then $f_1$ is a strict morphism of $\_P_\infty$-algebras.

	\begin{Ex}\cite[Section 10.2.6]{LV}\
		
		\begin{itemize}
			\item $A_\infty$-morphisms:
			
			Let $A$ and $B$ be $A_\infty$-algebras described by maps $\mu_i^A$ and $\mu_i^B$ respectively. An $\infty$-morphism $f : A \rightsquigarrow B$ is given by maps \[f_n : A[1]^{\otimes n} \rightarrow B[1]\] for all each $n \geq 1$. The condition that the induced map $\mathbf{Ass}^\antishriek (A) \rightarrow \mathbf{Ass}^\antishriek (B)$ preserves the differential is, for all $n \geq 1$: 
			\[ \sum_{0 \leq i+j \leq n} f_{i+j+1} \circ (1^{\otimes i} \otimes \mu_{n-i-j}^A \otimes 1^{\otimes j}) = \sum_{r \geq 1} \sum_{i_1 + \cdots + i_r = n} \mu_r^B \circ (f_{i_1} \otimes \cdots \otimes f_{i_r}) \]
			
			\item $L_\infty$-morphisms:
			
			Let $A$ and $B$ be $L_\infty$-algebras described by maps $\ell_i^A$ and $\ell_i^B$ respectively (with $\ell_1^A = d_A$ and $\ell_1^B = d_B$). An $\infty$-morphism $f : A \rightsquigarrow B$ is given by maps $f_n : \Sym^n (A[1]) \rightarrow B[1]$ for all each $n \geq 1$. The condition that the induced map $\mathbf{Lie}^\antishriek (A) \rightarrow \mathbf{Lie}^\antishriek (B)$ preserves the differential is a symmetrization of the $A_\infty$ situation and for all $n \geq 1$: 
		
			\[ \sum_{i+j = n} \sum_{\sigma \in \tx{Sh}(i,j)} \epsilon(\sigma) (f_{j+1} \circ_1 \ell_{i}^A)^\sigma = \sum_{r \geq 1} \sum_{i_1 + \cdots + i_r = n} \sum_{\sigma \in \tx{Sh}(i_1, \cdots, i_r)} \epsilon(\sigma) \ell_r^B \circ (f_{i_1} \otimes \cdots \otimes f_{i_r})^\sigma \]	
		
		\end{itemize}
	\end{Ex}

	\begin{Def}
		\label{Def_InftyIsotopy}
		
		We say that $f$ is an \defi{$\infty$-(quasi)-isomorphism} if the induced map $f_1 : A \rightarrow B$ is a (quasi)-isomorphism. An \defi{$\infty$-isotopy} between two $\_P_\infty$-structures on $A$ given by $\alpha, \beta : \_P_\infty \rightarrow \mathbf{End}_A$ is an $\infty$-morphism $f : \_P^\antishriek \rightarrow \mathbf{End}_A^A = \mathbf{End}_A$ between $\alpha$ and $\beta$ such that $f_1 := \tx{id}_A$. 
	\end{Def}
	
	\begin{RQ}\label{RQ_MorphismExtension}
		A morphism $f : A \rightarrow B$ of complexes between two $\_P_\infty$-algebras is said to extend to a $\infty$-morphism of $\_P_\infty$-algebra if there exists an $\infty$-morphism $f_\infty : A \rightsquigarrow B$ such that $f = (f_\infty)_1$.

	\end{RQ}

	\subsubsection{Homotopy transfer theorem}\label{Sec_HTT}\
	
	\medskip
	
As a general heuristic, there is a trade-off between how small a $\_P$-algebra is and how simple its structure is, while preserving the quasi-isomorphism type. For example, a quasi-free resolution of an algebra $A$ (such as $\Omega \mathbf B A$) is simple in the sense that it is there are no relations, but free objects tend to be quite big.

Trying to go in the other direction, the smallest possible $\_P$-algebra that has any chance of being quasi-isomorphic to $A$ is its own homology. In case $A$ is not formal, the homotopy transfer theorem tells us that in that case we can insist on working with the cochain complex $H(A)$ but the price to pay is to have to endow $H(A)$ with a \emph{different} $\_P_\infty$-algebra structure on it.

It might look odd to consider a $\_P_\infty$-structure on a complex with trivial differential, but this just means that while there exist higher homotopies, the $\_P$-algebra equations are satisfied on the nose. In that situation, the $\_P_\infty$ structure on $H(A)$ will recover the canonical $\_P$-algebra structure.\\

 Given a quasi-isomorphism $A \rightarrow B$ of cochain complexes, we would like to know when we can transfer a $\_P$-algebra structure (or even $\_P_\infty$-structure) from $A$ to $B$ (or vice versa). 
  We start by describing the conditions on our quasi-isomorphism that permits such a transfer. This procedure gives a way to produce $\_P_\infty$-structures on cohomology, transferring from the deformation retract of cochain complexes on cohomology (see Proposition \ref{Prop_FormalityOfCHKAndDeformationRetractToCohomology}). 
	
	\begin{Def}\label{Def_HomotopyandDeformationRetracts}
		
		We say that $(B, d_B)$ is a \defi{homotopy retract} of $(A, d_A)$ if there are maps of cochain complexes: 
		\[\begin{tikzcd}
		\arrow[loop left, "h"] (A, d_A) \arrow[r, "p", shift left] & \arrow[l, shift left, "i"] (B, d_B)
		\end{tikzcd}\]
		
		such that $i$ is a quasi-isomorphism and $\tx{Id}_A - i p = d_A h + h d_A$. This pair is a \defi{deformation retract} if moreover $pi =\tx{id}$, and a \defi{strong deformation retract} if we have the side conditions $ph = hi = h^2 = 0$. 
	\end{Def} 

To go from a deformation retract to a strong one, it suffices to change the homotopy $h$ to $h'' = -h'd_Ah'$, where $h' = (d_Ah+hd_A)h(d_Ah+hd_A)$.

	\begin{Th}[Homotopy Transfer Theorem, {\cite[Theorem 10.3.1]{LV}}]\ 
		\label{Th_HHT}
		
		Let $\_P$ be a Koszul operad and let $(B, d_B)$ a homotopy retract of $(A, d_A)$. Then, any $\_P_\infty$-algebra structure on $A$ can be transferred to a $\_P_\infty$-algebra structure on $B$ such that $i$ extends to a $\infty$-quasi-isomorphism $i_\infty$. Moreover, the $\_P_\infty$ structure on $A$ and $i_\infty$ can be explicitly described (see Remarks \ref{RQ_PInftyStructure} and \ref{RQ_iInfty}).
	\end{Th}
	
	The proof is based on the Rosetta Stone (Proposition \ref{Prop_RosettaStone}) so that finding a $\_P_\infty$-structure on $B$ boils down to finding a map $\_P^\antishriek \rightarrow \mathbf{B}\mathbf{End}_B$, which can be reduced to finding a map, $\Psi : \mathbf{B}\mathbf{End}_A \rightarrow \mathbf{B}\mathbf{End}_B$ that will make the following diagram commute:
	
	\[\begin{tikzcd}
	\_P^\antishriek \arrow[r] \arrow[dr] & \mathbf{B}\mathbf{End}_A \arrow[d, "\Psi"] \\
	 & \mathbf{B}\mathbf{End}_B
	\end{tikzcd}\] 
	
 This is the goal of the following Lemma. 
	
	\begin{Lem}[{\cite[Section 10.3.2]{LV}}]  \label{Lem_MapsBetweenBarConstructionofEndOperad}
		
		We can construct a map of cooperads $\Psi : \mathbf{B}\mathbf{End}_A \rightarrow \mathbf{B}\mathbf{End}_B$ that extends (by cofreeness) the map sending $m \in \mathbf{End}_A (n)$ to $p \circ m \circ i^{\otimes n} \in \mathbf{End}_B(n)$. 
	\end{Lem}

	\begin{RQ} \label{RQ_PInftyStructure}We will denote by $sm$ the element of a cochain complex $V[-1]$ corresponding to $m \in V$ but in degree 1 higher. To understand the map obtained in Lemma \ref{Lem_MapsBetweenBarConstructionofEndOperad}, recall from Definition \ref{Def_BarCobarConstruction} that $\mathbf{B}\mathbf{End}_A $ is the cofree cooperad on $\mathbf{End}_A[-1]$ together with the bar differential. The idea is to define a map that will send a tree in $\mathbf{B}\mathbf{End}_A$ to an element in $\mathbf{End}_B$ by adding $i$ to each leaf of the tree, $p$ to the root, and $h$ to each internal edge: 
	\[
	\begin{tikzpicture}[scale=0.6, baseline=-1ex]
	\GraphInit[vstyle=Classic]
	
	\tikzset{VertexStyle/.style = {shape = circle,fill = black,minimum size = 1pt,inner sep=0pt}}
	
	\SetVertexNoLabel
	
	\Vertex[empty,x=-4,y=3]{A}
	\Vertex[empty,x=-2,y=3]{B}
	\Vertex[empty,x=-1,y=3]{C}
	\Vertex[empty,x=1,y=3]{D}
	\Vertex[empty,x=3,y=3]{E}
	\Vertex[empty,x=5,y=3]{F}

	\Vertex[empty, x=0,y=-3]{S}
	
	\SetVertexLabel

	\Vertex[LabelOut,Lpos=-110, L=$sm_1$, x=-3,y=1.5]{M1}
	\Vertex[LabelOut,Lpos=-110, L=$sm_2$, x=0,y=1.5]{M2}
	\Vertex[LabelOut,Lpos=-20, L=$sm_3$, x=4,y=1.5]{M3}
	\Vertex[LabelOut,Lpos=-20, L=$sm_4$, x=2,y=0]{M4}
	\Vertex[LabelOut,Lpos=-20, L=$sm_5$, x=0,y=-1.5]{M5}
	
	\Edge(A)(M1)
	\Edge(B)(M1)
	\Edge(C)(M2)
	\Edge(D)(M2)
	\Edge(B)(M1)
	\Edge(E)(M3)
	\Edge(F)(M3)
	\Edge(M1)(M5)
	\Edge(M2)(M4)
	\Edge(M3)(M4)
	\Edge(M4)(M5)
	\Edge(M5)(S)

	\end{tikzpicture} \mapsto s \left( \begin{tikzpicture}[scale=0.4, baseline=-1ex]
	\GraphInit[vstyle=Classic]
	
	\tikzset{VertexStyle/.style = {shape = circle,fill = black,minimum size = 1pt,inner sep=0pt}}
	
	\SetVertexNoLabel
	
	\Vertex[empty,x=-4,y=3]{A}
	\Vertex[empty,x=-2,y=3]{B}
	\Vertex[empty,x=-1,y=3]{C}
	\Vertex[empty,x=1,y=3]{D}
	\Vertex[empty,x=3,y=3]{E}
	\Vertex[empty,x=5,y=3]{F}
	
	\Vertex[empty,x=-4,y=4.5]{A'}
	\Vertex[empty,x=-2,y=4.5]{B'}
	\Vertex[empty,x=-1,y=4.5]{C'}
	\Vertex[empty,x=1,y=4.5]{D'}
	\Vertex[empty,x=3,y=4.5]{E'}
	\Vertex[empty,x=5,y=4.5]{F'}
	
	\Vertex[empty, x=0,y=-3]{S}
	\Vertex[empty, x=0,y=-4.5]{S'}
	
	\SetVertexLabel

	\Vertex[LabelOut,Lpos=-110, L=$m_1$, x=-3,y=1.5]{M1}
	\Vertex[LabelOut,Lpos=-110, L=$m_2$, x=0,y=1.5]{M2}
	\Vertex[LabelOut,Lpos=-20, L=$m_3$, x=4,y=1.5]{M3}
	\Vertex[LabelOut,Lpos=-20, L=$m_4$, x=2,y=0]{M4}
	\Vertex[LabelOut,Lpos=-20, L=$m_5$, x=0,y=-1.5]{M5}

	\Edge[label=$i$](A')(A)
	\Edge[label=$i$](B')(B)
	\Edge[label=$i$](C')(C)
	\Edge[label=$i$](D')(D)
	\Edge[label=$i$](E')(E)
	\Edge[label=$i$](F')(F)
	\Edge[label=$p$](S)(S')
	
	\Edge(A)(M1)
	\Edge(B)(M1)
	\Edge(C)(M2)
	\Edge(D)(M2)
	\Edge(B)(M1)
	\Edge(E)(M3)
	\Edge(F)(M3)
	\Edge[label=$h$](M1)(M5)
	\Edge[label=$h$](M2)(M4)
	\Edge[label=$h$](M3)(M4)
	\Edge[label=$h$](M4)(M5)
	\Edge(M5)(S)
		
	\end{tikzpicture}\right)
	\]  
	
	Composing all the maps in the tree on the right hand side gives an element in $\mathbf{End}_B$ of same degree as the left hand side, and the map $\mathbf{B}\mathbf{End}_A$ to $\mathbf{End}_B$ naturally extends to the map of Lemma \ref{Lem_MapsBetweenBarConstructionofEndOperad}. 
		
	\end{RQ}
	
	\begin{RQ} \label{RQ_iInfty} The $\infty$-morphism $i_\infty$ is constructed by defining a map $\_P^\antishriek \rightarrow \mathbf{End}_A^B$ as:	
		\[ i_\infty : \begin{tikzcd}
	\overline{\_P}^\antishriek \arrow[r, "\Delta"] & \Tc \overline{\_P}^\antishriek    \arrow[r, "\Tc s\mu_A"] & \Tc (\mathbf{End}_A[-1]) \arrow[r, "\tilde{\Psi}"] & \mathbf{End}_A^B
	\end{tikzcd} \]
	\[ i_\infty : I \mapsto i \in \Hom_{k}\left( B,A \right) \subset \mathbf{End}_A^B \]
	
		 The morphism $\tilde{\Psi} : \Tc (\mathbf{End}_A[-1]) \rightarrow \mathbf{End}_A^B$ is defined just as $\Psi$, except for the root which is not labeled by $p$: 
	 	\[ 
	 	\begin{tikzpicture}[scale=0.55, baseline=-1ex]
	 	\GraphInit[vstyle=Classic]
	 	
	 	\tikzset{VertexStyle/.style = {shape = circle,fill = black,minimum size = 1pt,inner sep=0pt}}
	 	
	 	\SetVertexNoLabel
	 	
	 	\Vertex[empty,x=-4,y=2]{A}
	 	\Vertex[empty,x=-2,y=2]{B}
	 	\Vertex[empty,x=-1,y=2]{C}
	 	\Vertex[empty,x=1,y=2]{D}
	 	\Vertex[empty,x=3,y=2]{E}
	 	\Vertex[empty,x=5,y=2]{F}

	 	\Vertex[empty, x=0,y=-2]{S}
	 	
	 	\SetVertexLabel

	 	\Vertex[LabelOut,Lpos=-110, L=$sm_1$, x=-3,y=1]{M1}
	 	\Vertex[LabelOut,Lpos=-110, L=$sm_2$, x=0,y=1]{M2}
	 	\Vertex[LabelOut,Lpos=-20, L=$sm_3$, x=4,y=1]{M3}
	 	\Vertex[LabelOut,Lpos=-20, L=$sm_4$, x=2,y=0]{M4}
	 	\Vertex[LabelOut,Lpos=-20, L=$sm_5$, x=0,y=-1]{M5}
	 	
	 	\Edge(A)(M1)
	 	\Edge(B)(M1)
	 	\Edge(C)(M2)
	 	\Edge(D)(M2)
	 	\Edge(B)(M1)
	 	\Edge(E)(M3)
	 	\Edge(F)(M3)
	 	\Edge(M1)(M5)
	 	\Edge(M2)(M4)
	 	\Edge(M3)(M4)
	 	\Edge(M4)(M5)
	 	\Edge(M5)(S)

	 	\end{tikzpicture} \mapsto s \left( \begin{tikzpicture}[scale=0.58, baseline=-1ex]
	 	\GraphInit[vstyle=Classic]
	 	
	 	\tikzset{VertexStyle/.style = {shape = circle,fill = black,minimum size = 1pt,inner sep=0pt}}
	 	
	 	\SetVertexNoLabel
	 	
	 	\Vertex[empty,x=-4,y=2]{A}
	 	\Vertex[empty,x=-2,y=2]{B}
	 	\Vertex[empty,x=-1,y=2]{C}
	 	\Vertex[empty,x=1,y=2]{D}
	 	\Vertex[empty,x=3,y=2]{E}
	 	\Vertex[empty,x=5,y=2]{F}
	 	
	 	\Vertex[empty,x=-4,y=3]{A'}
	 	\Vertex[empty,x=-2,y=3]{B'}
	 	\Vertex[empty,x=-1,y=3]{C'}
	 	\Vertex[empty,x=1,y=3]{D'}
	 	\Vertex[empty,x=3,y=3]{E'}
	 	\Vertex[empty,x=5,y=3]{F'}
	 	
	 	\Vertex[empty, x=0,y=-2]{S}
	 	\Vertex[empty, x=0,y=-3]{S'}
	 	
	 	\SetVertexLabel

	 	\Vertex[LabelOut,Lpos=-110, L=$m_1$, x=-3,y=1]{M1}
	 	\Vertex[LabelOut,Lpos=-110, L=$m_2$, x=0,y=1]{M2}
	 	\Vertex[LabelOut,Lpos=-20, L=$m_3$, x=4,y=1]{M3}
	 	\Vertex[LabelOut,Lpos=-20, L=$m_4$, x=2,y=0]{M4}
	 	\Vertex[LabelOut,Lpos=-20, L=$m_5$, x=0,y=-1]{M5}

	 	\Edge[label=$i$](A')(A)
	 	\Edge[label=$i$](B')(B)
	 	\Edge[label=$i$](C')(C)
	 	\Edge[label=$i$](D')(D)
	 	\Edge[label=$i$](E')(E)
	 	\Edge[label=$i$](F')(F)
	 	\Edge(S)(S')
	 	
	 	\Edge(A)(M1)
	 	\Edge(B)(M1)
	 	\Edge(C)(M2)
	 	\Edge(D)(M2)
	 	\Edge(B)(M1)
	 	\Edge(E)(M3)
	 	\Edge(F)(M3)
	 	\Edge[label=$h$](M1)(M5)
	 	\Edge[label=$h$](M2)(M4)
	 	\Edge[label=$h$](M3)(M4)
	 	\Edge[label=$h$](M4)(M5)
	 	\Edge(M5)(S)
	 	
	 	\end{tikzpicture}\right)
	 	\]  
	 	
	\end{RQ}
	
	Focusing on the case where $B$ is the chain complex given by the homology of $A$, Proposition \ref{Prop_FormalityOfCHKAndDeformationRetractToCohomology} guarantees the existence of a deformation retract.
	
	\begin{Prop}[{\cite[Proposition 10.3.9]{LV}}]\label{Prop_HTTHomology} 
		
	There exists a strong deformation retract from $A$ to its homology: 
			\[\begin{tikzcd}
		\arrow[loop left, "h"] (A, d_A) \arrow[r, "p", shift left] & \arrow[l, shift left, "i"] (H(A), 0)
		\end{tikzcd}\]
		
	 In this case, $p$ can also be extended to an $\infty$-quasi-isomorphism of $\_P_\infty$-algebras: 
		
		\[ p_\infty : A \rightsquigarrow H(A) \]

	\end{Prop}

In particular, the homology of a $\_P$-algebra always carries a transferred $\_P_\infty$-structure. This structure depends on the choice of deformation retract, and therefore it is only unique up to $\infty$-isomorphism.

\begin{Ex}\label{Ex_MasseyProduct}
	For any topological space $X$, the cup product $\cup$ endows the singular cochain complex $C_{\tx{Sing}}^* (X)$ with an associative algebra structure. 
	Choosing a deformation retract to the cohomology, we transfer an $A_\infty$ structure to $H^*(X)$. For $n=3$ the product $\mu_3$ recovers essentially the Massey products of $X$ \cite{Ma58}, which are generalized to operadic algebras in \cite{muro}. Classically, the Massey product of classes $x,y,z\in H^*(X)$ is only defined if $xy=0$ or $yz=0$ and the result takes value in a quotient. The disappearance of this ambiguity in our case comes from the choice of deformation retract.
	More generally, the higher ($\mu_{n\geq 4}$) products  obtained from the homotopy transfer theorem are also related to the so-called higher Massey products, but the relation is a bit more subtle \cite{flynnconnolly2023higher}.
	
\end{Ex}

The homotopy transfer theorem enables a proof of Theorem \ref{Th_MainPropertiesOfInftyMaps} without appealing to any model categorical arguments. More generally, the same arguments show that all inclusions we can write between the four categories given by strict/homotopy $\_P$-algebras with strict/$\infty$ morphisms induce equivalences at the level of homotopy categories, appealing to the property known as \emph{rectification}: Every $\_P_\infty$-algebra $A$ is equivalent to a strict $\_P$-algebra, namely $\Omega_\kappa \mathbf B_\kappa A$.

It can be shown (and we invite the reader to do so) that Theorem \ref{Th_MainPropertiesOfInftyMaps} is a formal consequence of Theorem \ref{Th_ModelStructureCoalgebra}.
\begin{proof}[Sketch of proof of Theorem \ref{Th_MainPropertiesOfInftyMaps}]
The proof of this result rests on two key points. Given an $\infty$-quasi-isomorphism of $\_P$-algebras $f\colon X\rightsquigarrow Y$:

\begin{enumerate}
	\item There is a zig-zag of quasi-isomorphisms connecting $X$ to $Y$,
	
	\item There is  $g\colon Y\rightsquigarrow X$ such that the two composites are the identity in homology.
\end{enumerate}

For the first point, notice that by the bar-cobar adjunction $f\colon \mathbf B X \to \mathbf B Y$ is equivalent to a map $\Omega \mathbf B X \to Y$. One can therefore consider the short zig-zag $X \stackrel{\sim}{\leftarrow} \Omega \mathbf B X \stackrel{\sim}\to Y$.

For the second point, one first consider the case where $d_X = 0 = d_Y$, i.e., we need to know that $\infty$-isomorphisms of $\_P_\infty$ are indeed invertible. This can be shown by an argument similar to the invertibility of formal power series (see \cite[Theorem 10.4.1]{LV}).

In case of non-trivial differentials, by the homotopy transfer theorem we can reduce the problem to the homology. More precisely, the composite of $\infty$-quasi-isomorphisms

\begin{align}\label{Eq:oomap}
	H(X) \stackrel[i_\infty]{\sim}{\rightsquigarrow} X \stackrel[f]{\sim}{\rightsquigarrow} Y \stackrel[p_\infty]{\sim}{\rightsquigarrow} H(Y) 
\end{align}
is an $\infty$-isomorphism, therefore we set $g$ to be

$$Y \stackrel[p_\infty]{\sim}{\rightsquigarrow} H(Y) \stackrel{\sim}{\rightsquigarrow} H(X) \stackrel[i_\infty]{\sim}{\rightsquigarrow} X,$$

where the middle map is the inverse of the composition given by \eqref{Eq:oomap}.
\end{proof}
	 For a general study of the operadic homotopy transfer theorem  in the spirit of the homological perturbation lemma see \cite{Be14}. This framework gives an extension of $i$, $p$ and $h$.

	\subsection{Deformation of Algebraic Structures}\
\label{Sec_DeformationofAlgebraicStructures}

\medskip

This section is devoted to the study of deformations of algebraic structures. Given $A \in \chk$, the set of $\_P_\infty$-structures on $A$ is by definition given by $ \tx{Tw}\left(\_P^\antishriek, \mathbf{End}_A \right)$, the set of Maurer--Cartan elements of the convolution Lie algebra (Definition \ref{Def_ConvolutionPreLieAlgebra} and \ref{Def_TwistingMorphisms}). In Section \ref{Sec_MaurerCartanSpace}, we define a  set, groupoid and $\infty$-groupoid of Maurer--Cartan elements of a Lie algebra, commonly called the Deligne ($\infty$-)groupoid. This is a priori unclear what these objects have to do with deformations, but in Section \ref{Sec_DeformationComplex}, we define a set, groupoid and $\infty$-groupoid of $\G_R$-deformations of a $\_P_\infty$-structure on $A$ with $\G_R$ a local (Artinian) commutative ring (for example $\G_R = k[\epsilon]$, the ring of dual numbers). Theorem \ref{Th_DeformationAndDeligneGroupoids} then tells us that the deformation $\infty$-groupoid is homotopy equivalent to a Deligne $\infty$-groupoid of a certain twisted Lie algebra. 

\subsubsection{Maurer--Cartan space and gauge group}\ \label{Sec_MaurerCartanSpace} 

\medskip

The set of Maurer--Cartan elements (Definition \ref{Def_MC}) is not only good to define twisting morphism but is known to describe the set of deformations controlled by a Lie algebra $\mathfrak g$. We are interested in such deformations up to equivalences.

	Such equivalences usually coincide with the general notion of \emph{gauge} equivalences: Just as a finite dimensional classical (i.e. $\G_g=\G_g^0$) Lie algebra integrates to a Lie group via an exponential map, a (dg) Lie algebra $\G_g$ formally integrates to the so-called \email{gauge group} $\exp(\G_g^0)$. The gauge group will act naturally on $\mathrm {MC}(\mathfrak g)$.

To talk about exponentials and deal with convergence issues, we will need to use complete filtered Lie algebras. Taking complete filtered Lie algebras will ensure that all the ensuing infinite sums will converge.


\begin{Def}\label{Def_Filtration}
	We say that $\G_g$ is a \defi{complete positively filtered Lie algebras} or just a \defi{complete Lie algebra} if it is equipped with a complete Hausdorff descending filtration: 
	\[\G_g := F^1 \G_g \supset F^2 \G_g \supset  \cdots   \]

Being complete Hausdorff here means that there is a equivalence $\G_g \rightarrow \lim_n \faktor{\G_g}{F^n \G_g}$. To any filtered Lie algebra, one can associate a complete Lie algebra, $\hat{\G_g} = \lim_n \faktor{\G_g}{F^n \G_g}$, called its \defi{completion}.
	
	We also assume that the filtration is compatible with the differential and the Lie bracket: 	
	\[ d(F^n \G_g) \subset F^n \G_g \qquad \left[F^p\G_g, F^q \G_g \right] \subset F^{p+q}\G_g  \]
	
	We will denote by $\G_g_0$ the set of degree $0$ elements. 
\end{Def}

For now, let us assume that the Lie structure on $\G_g$ is obtained from a unital associative product $[x,y] = x \star y - y \star x$. Then, assuming convergence, we can write down the usual formula for the exponential: 

\begin{center}
	\begin{tabular}{cccccl}
		$\exp$ : & $\G_g_0 $& $\longrightarrow$ &$ \G_g_0$    \\ 
		& $X$ & $\mapsto$ & $e^X = \sum\limits_k \frac{X^{\star k}}{k!}$  \\ 
	\end{tabular}
\end{center}

The image of this map forms a group with respect to $\star$ and unit $1$, which we would like to call the \emph{gauge group} $G$. Taking the formal logarithm, we can equivalently define the gauge group as a group structure on the underlying set of $\mathfrak g_0$:
	$$ G = (\mathfrak g_0, X\cdot_G Y  \coloneqq \log(e^X\star e^Y), 0).$$

However, this product does not make sense for a general Lie algebra. To address this issue, we can value this exponential in the \emph{complete} universal enveloping algebra of $\widehat{\G_g_0}$: 
\begin{center}
	\begin{tabular}{cccccl}
		$\exp$ : & $\G_g_0 $& $\longrightarrow$ &$ \widehat{\G_U} \G_g_0$    \\ 
		& $X$ & $\mapsto$ & $e^X = \sum\limits_k \frac{X^{\star k}}{k!}$  \\ 
	\end{tabular}
\end{center}

In fact, the tentative formula that we wrote before in the associative case  $\log(e^X\star e^Y)$ turns out to be expressible purely in terms of the Lie bracket of $\mathfrak g_0$ in a universal way.
This formula is called the BCH (Baker--Campbell--Hausdorff) formula, and its initial terms are:

$$\tx{BCH}(X,Y) = X+ Y +\frac{1}{2}[X,Y] +\frac{1}{12}\left([X,[X,Y]] +[Y,[Y,X]]\right)+ \dots .$$

\begin{Prop}
	\label{Prop_BCH}
	There exist a universal BCH formula, such that:
	\[ e^X e^Y = e^{\tx{BCH}(X,Y)}\]
	
	In particular, the exponential map induces a bijection between $\G_g_0$ and its image:
	\[  \exp : \G_g_0 \rightarrow  G \subset \widehat{\G_U}\G_g_0
	\]
\end{Prop}
	We call the group $G\coloneqq \exp (\G_g_0) \cong (\G_g_0, \tx{BCH},0)$, the \defi{gauge group} of $\mathfrak g$.

	\begin{proof}
			We follow an approach close to \cite[Section 4.1]{BF11}. 	
		To show a universal BCH formula, it is enough to show that such a formula exists on the free complete Lie algebra on two generators $\widehat{\mathbf{Lie}}(X,Y)$. Note that its universal enveloping algebra is the free complete associative algebra on $X,Y$, i.e. the algebra of non-commutative formal  power series $\widehat{\mathbf{Ass}}(X,Y )= \widehat{k\langle X,Y\rangle}$.
		
		It is easy to check that the exponential map is one to one into the set of formal power series with constant term equal to $1$ and indeed an inverse is given by the formal logarithm.
				\[ \begin{tikzcd}
					\exp:  \widehat{k\langle X,Y\rangle}  \arrow[r, shift left] & \arrow[l, shift left] 1+ \widehat{k\langle X,Y\rangle}_{\geq 1} : \log.
		\end{tikzcd}\]

	We claim that $v\in  \widehat{k\langle X,Y\rangle}$ is primitive if and only if $\exp v$ is grouplike  (see \ref{sec:Lie digression}). One implication is a direct verification and for the other one can observe that 
		
		\[
		\Delta(\exp v) = \exp v \otimes \exp v \Rightarrow \exp(\Delta v) = \exp(1\otimes v + v\otimes 1),
		\]
		 where the exponentials on the right hand side are being taken in the algebra $\widehat{k\langle X,Y\rangle}\hat\otimes \widehat{k\langle X,Y\rangle}$. The claim follows from the injectivity of the exponential.

		By the complete version of the Poincaré--Birkoff--Witt theorem, the set  of primitive elements of $\widehat{k\langle X,Y\rangle}$ is precisely the Lie algebra $\widehat{\mathbf{Lie}}(X,Y)$ and the group-like elements form a group with respect to the product and unit of the universal enveloping algebra.
		
		The BCH formula is therefore obtained by taking the exponentials of $X, Y \in \widehat{\mathbf{Lie}}(X,Y) $,  multiplying them into a group-like element $e^X\cdot e^Y$ and taking the inverse of the exponential.
		
	\end{proof}

\begin{Prop}\label{Prop_DeformationLieAlgebraMaurerCartan}
	Given a Lie algebra $\G_g$ and $\phi \in \tx{MC}(\G_g)$, we can twist $\G_g$ by $\phi$ and define:
	\[\G_g^\phi = \left( \G_g, [-,-]_\G_g, d_\phi = d + [\phi, -]_\G_g \right) \]
	
	This is also a Lie algebra and a Maurer--Cartan elements in $\G_g^\phi$ is exactly an element $\psi \in \G_g$ such that $\phi + \psi$ is a Maurer--Cartan element of $\G_g$. In other words, we can think of them as the ``deformations'' of the Maurer--Cartan element $\phi$. 
	
\end{Prop}

We now wish to describe an action of the gauge group on the Maurer--Cartan set of a complete Lie algebra $\mathfrak g$. As a motivation, if we assume for a bit $\mathfrak g_1$ and $\mathfrak g_2$ to be finite dimensional, the solutions of the Maurer--Cartan equation $\tx{MC}(\mathfrak g)$, form an affine variety (an intersection of quadrics) and therefore its tangent space at a Maurer--Cartan element $x$ is $T_x\tx{MC}(\mathfrak g) = \{\alpha\in \mathfrak g_1 | d_x \alpha = 0 \}$. It follows that any element $\lambda \in \mathfrak g_0$ induces a tangent vector $d_x \lambda \in T_x\tx{MC}(\mathfrak g)$. Varying $x$, we get that $\lambda$ induces a vector field on $\tx{MC}(\mathfrak g)$. Two Maurer--Cartan elements connected by the flow of such $\lambda$ and said to be \emph{gauge equivalent}.

\begin{Def}\label{Def_DeligneGroupoid}
	Given a complete Lie algebra, there is an action of the gauge group, $\exp(\G_g) = (\G_g_0, \tx{BCH})$, on the set of Maurer--Cartan elements $\tx{MC}(\G_g)$, defined by: 	
	\[ e^\lambda\cdot x = x + \frac{e^{\tx{ad}_\lambda} - \tx{id}}{\tx{ad}_\lambda}  (d_x \lambda) := x - \sum_{k=0}^n \left(\frac{[\lambda, -]^n}{(n+1)!}\right)(d_x \lambda) \]
	
	where $d_x := [x,-] + d$ is the differential on $\G_g$ twisted by the Maurer--Cartan element $x$. The associated action groupoid on the set of Maurer--Cartan elements is called the \defi{Deligne groupoid}, denoted $\underline{\tx{Del}}(\G_g)$. 
\end{Def}

To relate the gauge action formula with the discussion above, we point out that one can consider the \emph{differential trick}, which gets rid of the differential of a dg Lie algebra by making it internal. We consider the graded Lie algebra $\mathfrak g^+ = \mathfrak g \oplus k\delta$, where $\delta$ is a degree $1$ nilpotent element such that for $x\in \mathfrak g$, $[\delta, x] = d_{\mathfrak g} x$. It it easy to see that the Maurer--Cartan elements of $\mathfrak g$ correspond to the set of degree $1$ square-zero elements in $g^+$ of the form $\delta + x$.

The gauge action formula from Definition \ref{Def_DeligneGroupoid} follows from applying the exponential of the operator $\mathrm{ad}_\lambda$ to $\delta +x$. Indeed, since $\mathrm{ad}_\lambda$, its exponential is a morphism of graded Lie algebras and therefore preserves square-zero elements, which justifies that $e^\lambda\cdot x$ is still a Maurer--Cartan element. For more details see \cite[Theorem 1.53]{DSV23}.\\

When dealing with homotopy theory, we often need an $\infty$-groupoid to obtain a full ``space'' with higher homotopies. It turns out that there is an $\infty$-groupoid version of the Deligne groupoid.

\begin{Def}[Deligne $\infty$-groupoid, {\cite[Section 2.2]{DVR}}]  \label{Def_HigherDeligneGroupoid}
	
Consider the simplicial commutative $k$-algebra $\Omega [\Delta^\bullet]$ (sometimes written $\Omega_\bullet$) defined by:
	\[ \Omega [\Delta^n] := \faktor{k[t_0, \cdots, t_n, dt_0, \cdots , dt_n]}{\left( \sum t_i = 1, \ \sum dt_i =0 \right)}\]

	with the differential sending $t_i$ to $dt_i$. 
	
	For any Lie algebra $\G_g$, $\G_g \otimes \Omega[\Delta^\bullet]$ is a simplicial Lie algebra where for all $n \in \Nn$, $\G_g \otimes \Omega[\Delta^n]$ is the Lie algebra with Lie bracket:  
	\[ [v \otimes \alpha, w \otimes \beta ] = [v,w]_\G_g \otimes \alpha . \beta \] 
	
	Then the \defi{Deligne $\infty$-groupoid} is the simplicial set defined by:
	\[ \mathbf{Del} (\G_g):= \tx{MC}(\G_g \hat{\otimes} \Omega[\Delta^\bullet])\]
	
	with  $\G_g \hat{\otimes} \Omega[\Delta^\bullet] := \lim_n \left( \faktor{\G_g}{F_n \G_g} \hat{\otimes} \Omega[\Delta^\bullet] \right)$. This definition from \cite{DVR} enables an extension of Theorem \ref{Th_GoldmannMilson} to Theorem \ref{Th_GoldmanMilsonExtended}. This simplicial set is an $\infty$-groupoid (i.e. a Kan complex) thanks to \cite[Proposition 4.1]{DVRC17}.
\end{Def}

\begin{Prop}[{\cite[Lemma B.2]{DVR}}] \label{Prop_Pi0DeligneGroupoid}

	There is a bijection between the connected components of the Deligne groupoid and those of the $\infty$-groupoid of Maurer--Cartan elements:
	\[ \faktor{\tx{MC}(\G_g)}{\sim} := \pi_0 (\underline{\tx{Del}}(\G_g)) \simeq \pi_0 \left( \mathbf{Del} (\G_g) \right) \]
	
	The equivalence relation $\sim$ is given by the gauge equivalences obtained from the action of the gauge group. 
\end{Prop}

\begin{Th}[Goldman--Milson, \cite{DVR}]\label{Th_GoldmannMilson}
	
	Given $f: \G_g \rightarrow \G_h$, a filtered quasi-isomorphism (see \ref{RQ:filtered q.i.}) of filtered complete Lie algebras, we have a bijection:
	\[\faktor{\tx{MC}(\G_g)}{\sim}  \rightarrow \faktor{\tx{MC}(\G_g)}{\sim} \]
\end{Th}

From the deformation theoretical perspective, this result tells us that two Lie algebras that are quasi-isomorphic encode equivalent deformation problems. In practice, one might be interested in showing that a Lie algebra is formal in order to have the ``simplest possible description'' of the deformation problem at hand.

\begin{RQ}\label{RQ:filtered q.i.}
The hypothesis that $f\colon\G_g \rightarrow \G_h$ is a \emph{filtered} quasi-isomorphism,  meaning that $f$ induces a quasi-isomorphism {on the associated graded}, is a stronger requirement than just a quasi-isomorphism.

As an example, let $\G_g$ be the two dimensional Lie algebra spanned by an element $x$ in degree $1$, and an element $[x,x]$ in degree $2$, such that $dx = -\frac{1}{2}[x,x]$. This Lie algebra is quasi-isomorphic to $0$, but not filtered quasi-isomorphic to $0$ for any possible complete filtration. Indeed, $\G_g$ has two Maurer--Cartan elements $0$ and $x$ which are not equivalent, since the gauge group is trivial.
\end{RQ}

One can suspect that the Goldman--Milson Theorem should generalise to the $L_\infty$ setting, since $\infty$-morphisms are equivalent to zig-zags of strict morphisms and $L_\infty$-algebras are rectifiable to strict Lie algebras. Indeed, given a complete $L_\infty$-algebra $\mathfrak g$, the corresponding Maurer--Cartan equation is:
\[
\mathrm{MC}(\mathfrak g) = \{x\in \G_g_1 \ | \  dx + \frac{1}{2!}\ell_2(x,x)+\frac{1}{3!}\ell_3(x,x,x) + \dots = 0\}.
\]

This is functorial with respect to any $\infty$-morphism $f\colon \mathfrak g \rightsquigarrow \mathfrak h$ with components $(f_n)_{n\geq 1}$. We get a map:
\begin{align*}
f\colon \mathrm{MC}(\mathfrak g)  &\to\mathrm{MC}(\mathfrak h)\\
x \hspace{.2cm} &\mapsto\sum_{n\geq 1} \frac{1}{n!}f_n(x,\dots,x).
\end{align*}

The Goldman--Milson Theorem indeed generalises to this setting and furthermore, thanks to Proposition \ref{Prop_Pi0DeligneGroupoid}, we can interpret the Goldman--Milson theorem as the $\pi_0$ case of an equivalence of the Deligne $\infty$-groupoid:

\begin{Th}[{\cite[Theorem 1.1]{DVR}}] \label{Th_GoldmanMilsonExtended}
	Let $f\colon L \rightarrow L'$ be an $\infty$-morphism between complete positively filtered $L_\infty$-algebras compatible with the filtration. If the linear term $f_1 : L \rightarrow L'$ is a filtered quasi-isomorphism, then it induces a weak homotopy equivalence: 
	\[\mathbf{Del}(L) \rightarrow \mathbf{Del}(L') \]
\end{Th}

\subsubsection{A digression on variations of the Deligne groupoid}\label{sec:a-digression-on-variations-of-the-deligne-groupoid}\

\medskip

	Given a complete Lie algebra $\mathfrak g$, it follows from $\Omega[\Delta^0] = k$ that the $0$-simplices of the Deligne groupoid $\mathbf{Del}(\mathfrak g)$ are precisely the Maurer--Cartan elements of $\mathfrak g$.
	In light of Proposition \ref{Prop_Pi0DeligneGroupoid}, it would be natural to expect the $1$-simplices of $\mathbf{Del}(\mathfrak g)$ to correspond to the gauge equivalences. This is however not true: the set of $1$-simplices is too big (even though it encodes the same gauge equivalence relation).
	In \cite{Ge09}, Getzler solved this problem by considering the normalised cellular complex of the $n$-simplex $C_n\coloneqq C^*(\Delta^n)$, which has the advantage of being much smaller than $\Omega[\Delta^n]$ and the disadvantage of not being a commutative algebra.

	\begin{Prop}[Dupont contraction]\label{prop:Dupont}
There exists a simplicial deformation retract:
	\[\begin{tikzcd}
	\arrow[loop left, "h_\bullet"] \Omega[\Delta^\bullet] \arrow[r, "p_\bullet", shift left] & \arrow[l, shift left, "i_\bullet"] C_\bullet.
\end{tikzcd}\] 
	\end{Prop}
The word ``simplicial'' in the statement means that $i_\bullet, p_\bullet$ and $h_\bullet$ commute with the respective simplicial structures.
	
	\begin{Th}[{\cite{Ge09}}]
		Given a complete Lie algebra $\mathfrak g$, the simplicial set 
		\[
		\gamma_\bullet (\mathfrak g)\coloneqq \mathrm{MC}(\G_g \hat{\otimes} \Omega[\Delta^\bullet]) \cap \ker h_\bullet
		\]
is isomorphic to the nerve of the Deligne groupoid of $\mathfrak g$. In particular it
is a $\infty$-groupoid and the inclusion $\gamma_\bullet(\G_g) \to \mathbf{Del}(\G_g)$ is a homotopy equivalence. We call it the \emph{Getzler $\infty$-groupoid}.
	\end{Th}	
Despite its apparent simplicity, the condition $h_\bullet(x)=0$ hides the difficult combinatorics of the Dupont contraction, which make some categorical properties of 	$\gamma_\bullet$ more difficult to show.

Let us briefly describe the operadic approach that led Robert-Nicoud and Vallette \cite{RV20} to give a more conceptual presentation of Getzler's $\infty$-groupoid.

The usual construction of the Deligne $\infty$-groupoid involves taking a simplicial framing of a Lie algebra $\mathfrak g$ by associating to it the simplicial Lie algebra $\mathfrak g \otimes \Omega [\Delta^\bullet]$. Notice that its Maurer--Cartan elements can be identified with a mapping space. Let $k\epsilon$ be the $1$-dimensional $\mathbf{Lie}^\antishriek$ coalgebra, with $\deg \epsilon =1$ and $\Delta \epsilon = \frac{1}{2}\epsilon \otimes \epsilon$. Then if $\kappa \colon \mathbf{Lie}^\antishriek\to \mathbf{Lie}$ is the canonical Koszul morphism, we have
\[
\mathbf{Del}(\mathfrak g) = \Hom_{\mathbf{Alg}_{\mathbf{Lie}}}(\Omega_\kappa(k\epsilon),\mathfrak g \otimes \Omega [\Delta^\bullet]).
\]

Another way to cook up a simplicial set out of a Lie algebra $\mathfrak g$ is to consider all maps from a cosimplicial Lie algebra into $\mathfrak g$. The cosimplicial Lie algebra we will consider is denoted $\mathfrak{mc}^\bullet$ is in fact an $L_\infty$-algebra that is Koszul dual to the $C_\infty$-algebra structure on $C_\bullet$ obtained by applying the homotopy transfer theorem (on every simplicial degree) to the Dupont contraction (see \cite[Section 2.2]{RV20}).

\begin{Th}[{\cite[Theorem 2.17]{RV20}}]
	
There is an isomorphism
\[
\gamma_\bullet(\mathfrak g) \cong \Hom_{L_\infty}(\mathfrak{mc}^\bullet,\mathfrak g).
\]
\end{Th}
This description automatically gives some functoriality properties of $\gamma_\bullet$ relative to some refined version of $\infty$-morphisms of Lie algebras and immediately produces the left adjoint $\mathfrak L$ of $\gamma_\bullet\colon \bf{Alg}_{\_L_\infty} \to \tx{sSets}$. 

As a beautiful application, let us point out that one of the key steps in \cite{Ge09} is showing that $\gamma_\bullet(\mathfrak g)$ is an $\infty$-groupoid. Using the explicit formula for $\mathfrak L$ one can easily show that given any horn $\Lambda^n_k \to \gamma_\bullet(\mathfrak g)$, the set of corresponding horn fillers is in bijection with the degree $-n$ elements $\mathfrak g_{-n}$.
It follows that there always exist horn fillers since $\mathfrak g_{-n}$ is a vector space and is therefore non-empty! Furthermore, a canonical horn filler exists: $0\in\mathfrak g_{-n}$. Such objects are called \defi{algebraic Kan complexes}.

\begin{Prop}[{\cite[Section 5.2]{RV20}}]\label{Prop_BCHHornFilling}
Let $x,y\in \exp g$ be elements of the gauge group of $g$. They define a horn $\Lambda_1^2 \to \gamma_\bullet(\mathfrak g)$. The horn filler correponding to $0$ produces the BCH formula:

\begin{center}
	\includegraphics[scale=0.7]{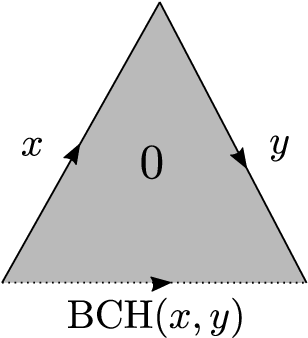}
\end{center}
\end{Prop}

\subsubsection{Deformation Complex}\ 
\label{Sec_DeformationComplex}

\medskip

In this section we define a groupoid of deformations of algebraic structures obtained by considering algebraic structures on the tensor product with a local commutative algebra such that the quotient by the maximal ideal recovers the initial structure we want to deform. For example, if $\_O_X$ is some algebra of functions, a classical question is to understand its deformations as an associative algebra along the algebra of formal power series $k[\![t]\!]$. This is given by an associative product $\star$ on $\_O_X\otimes_k k[\![t]\!] = \_O_X[\![t]\!]$ such that for all $a, b \in \_O_X$, $a.b = a \star b \ \tx{mod} \ (t)$.\\

First let us describe the space of $\_P$-algebra structures on a given cochain complex $A$. Such a structure naturally defines a map $\mu : \_P_\infty \rightarrow \mathbf{End}_A$. We know from Proposition \ref{Prop_RosettaStone} that this is equivalent to the datum of a twisting morphism in the convolution pre-Lie algebra.

\begin{Def}\label{Def_ConvolutionPreLieAlgebra}
	Given a Koszul operad $\_P$, we define the \defi{convolution Lie algebra} of the cochain complex $A$ (Definition \ref{Def_TotatlizationOfAnOperad}) as:
	\[ \G_g_{\_P_\infty,A} := \tx{Tot}(\tx{Conv}(\_P^\antishriek, \mathbf{End}_A)) := \prod_{n\geq0} \iHom_{\Sigma_n} \left( \_P^\antishriek(n), \tx{End}_A(n) \right) \]
	
	together with the Lie structure induced by the natural pre-Lie product of Definition \ref{Def_TotatlizationOfAnOperad}. As such, we have an equivalence between $\_P_\infty$-structures and Maurer--Cartan elements in $\G_g_{\_P_\infty, A}$: 
	\[ \lbrace \_P_\infty-\tx{structures \ on }\ A \rbrace \simeq \tx{Tw}(\_P^\antishriek, \mathbf{End}_A) := \tx{MC}(\G_g_{\_P_\infty,A})\]
	Notice that the convolution Lie algebra is naturally complete filtered with filtration given by the arity.

\end{Def}

\begin{Def} \label{Def_DeformationComplex} 
		Let $A$ be a $\_P_\infty$-algebra, represented by $\phi \in \tx{MC}(\G_g_{\_P_\infty,A})$.
		We define the \defi{deformation complex of $A$} (or deformation complex at $\phi$) to be the twisted Lie algebra $\G_g_{\_P_\infty, A}^\phi$  (using the notation of Proposition \ref{Prop_DeformationLieAlgebraMaurerCartan}).
			
\end{Def}

{\begin{RQ}		
		The construction of the twist of a Lie algebra $\G_g^\phi$ is functorial in the pair $(\G_g, \phi)$. More precisely, for a morphism: 
		\[ f : (\G_g, \phi) \to (\G_h, \phi') \]
		given by a morphism of Lie algebras $f:  \G_g \to \G_h$ such that $f(\phi) = \phi'$, then we get a morphism of Lie algebras: 
		\[ f : \G_g^\phi \to \G_h^{\phi '}\] 
		
		To see this, we only need to show that $f$ exchanges $d_\phi$ and $d_{\phi '}$: 
		\[\begin{split}
			f \circ d_\phi  & = f\circ (d + [\phi, -]_\G_g  )\\
			& = f\circ d + f \left([\phi, -]_\G_g \right) \\
			& = d \circ f + [f(\phi), f(-)]_\G_h \\
			& = d_{\phi '} \circ f.
		\end{split}\] 
	
	It therefore follows from Remark 
	\ref{RQ_NaturalityAlgebraicBarCobar}, that deformation complexes are functorial with respect to morphisms of operads.
\end{RQ}}

We will now explain how this deformation complex does indeed encode deformations of the $\_P$-algebra structure determined by $\phi$. To do so we fix a classical (non-dg) local  commutative algebra $\G_R$ with maximal ideal $\G_M$ that will parameterize the type of deformation we are interested in. For example we can consider $\G_R$ a local Artinian algebra (Definition \ref{Def_SmallComAlgebra}), the algebra of dual numbers $\G_R = k[\epsilon]$ (we take the geometers' convention that $\epsilon^2 = 0$) for ``first-order'' deformations (more generally any $\G_R$ with $\G_M^2 =0$ will encode ``first-order'' deformations) or the pro-Artinian algebra (a cofiltered limit of Artinian algebras) $\G_R = k[\![t]\!]$ to encode formal deformations.

\begin{Def}\label{Def_DeformationGroupoidAtAMaurerCartanElement}
	
	We define the set of deformations of a $\_P$-algebraic structure on $A$, $\phi : \_P \rightarrow \mathbf{End}_A$, as the set of $\G_R$-linear $\_P_\infty\otimes \G_R$-algebra structure on $A \otimes \G_R$ that coincide with $\phi$ modulo $\G_M$. This can be described as: 
	\[ \tx{Def}_\phi ( \G_R)  := \left\lbrace \Phi \in \tx{MC}\left( \G_g_{\_P_\infty, A}^\phi \otimes \G_R \right), \ \Phi \cong \phi \ \tx{mod} \ \G_M \right\rbrace, \]
	
see \cite[Lemma 12.2.5]{LV} for details. The functor sending $\G_R$ to $\tx{Def}_\phi ( \G_R)$ is what we would like to think as a prototypical example of a \emph{deformation functor}. A more thorough study of deformation functors is done in Section \ref{Sec_Formal Moduli Problem}. We also refer to \cite{Ni06} for a more classical approach to deformation functors.
	
	Following \cite{DSV16}, an equivalence between two such deformations is given by an $\G_R$-linear $\infty$-isomorphism $f: (A \otimes \G_R, \Phi) \rightarrow (A \otimes \G_R, \Psi)$ such that $f_1 \cong \tx{id} \ \tx{mod} \ \G_M$. \\
	
	Together with this notion of equivalence, $\tx{Def}_\phi (\G_R)$ forms a groupoid. We will denote this groupoid by $\underline{\tx{Def}}_\phi(\G_R)$ and its quotient by: 
	\[\textit{Def}_\phi(\G_R) := \faktor{\tx{Def}_\phi (\G_R)}{\sim} = \pi_0 \left( \underline{\tx{Def}}_\phi (\G_R) \right)\]

\end{Def}

\begin{Th}[{\cite[Theorem 12.2.10]{LV}}]\label{Th_DeformationAndDeligneGroupoids}
	For $\G_R$ an Artinian local algebra (Definition \ref{Def_SmallComAlgebra}), there is a equivalence between the deformation groupoid and the Deligne groupoid (Definition \ref{Def_DeligneGroupoid}): 
	\[ \underline{\tx{Def}}_\phi (\G_R) \simeq \underline{\tx{Del}}\left( \G_g_{\_P_\infty, A}^\phi \otimes \G_M \right)\]
\end{Th}

	\begin{RQ}
		For the pro-Artinian algebra $\G_R = k[\![t]\!] = \lim  \left(\faktor{k[t]}{(t^{n})}\right)$, we could \emph{define} the value of  $ \underline{\tx{Def}}$ and $\underline{\tx{Del}}$ as the limits: 
		\[  \underline{\tx{Def}}_\phi (k[\![t]\!]) := \lim \ \underline{\tx{Def}}_\phi \left(\faktor{k[t]}{(t^{n})}\right)  \qquad \underline{\tx{Del}} (k[\![t]\!]) := \lim \ \underline{\tx{Del}} \left(\faktor{k[t]}{(t^{n})}\right) \]
		so that the equivalence also extends to pro-Artinian algebras.
\end{RQ}

\begin{RQ}\label{RQ_DeformationAndCohomology}
	For $\G_R = k[\epsilon],$ the algebra of dual numbers, we have $\G_M = k$ and we find that $\tx{Def}_\phi (k[\epsilon]) = Z^1 (\G_g_{\_P_\infty, A }^\phi)$ and $\textit{Def}_\phi (k[\epsilon]) = H^1 (\G_g_{\_P_\infty, A }^\phi)$. In that sense, the complex underlying the Lie algebra $\G_g_{\_P_\infty, A }^\phi$ encodes the first order infinitesimal deformations of $\phi$.
	
	If one has a first order infinitesimal deformation of $\phi$, of the form $\phi + t \alpha_1$ living over $\G_R = k[t]/(t^2)$ and wishes to extend it to a second order deformation of the form $\phi + t \alpha_1 + t^2 \alpha_2$ living over $k[t]/(t^3)$, simplifying the Maurer--Cartan equation one sees that the condition is that $2[\phi,\alpha_2] + [\alpha_1,\alpha_1] =0$. In other words, the condition is that the class of $[\alpha_1,\alpha_1]$ vanishes in $H^2(\G_g_{\_P_\infty, A }^\phi)$.
	
	This line of reasoning concludes that the obstruction to extending an infinitesimal deformation to a formal deformation is controlled by $H^2(\G_g_{\_P_\infty, A }^\phi)$ (\cite[Theorem 12.2.14]{LV}). 
\end{RQ}
	
	As we will see in Proposition \ref{Prop_TangentDGLAAndDeformationComplex}, the deformation complex is closely related to the operadic tangent complex which classifies higher deformations including higher homology as explained in Remarks \ref{RQ_HigherCotangentComplex} and \ref{RQ_HigherDeformationFormalSpectrumCommutative}. 
	

\begin{Ex}\
	
	\begin{itemize}
		\item Let us compute the deformation complex of an associative\footnote{Recall from Example \ref{Ex_InfityAlgebras} that $\mathbf{Ass}^\antishriek (n)$ is isomorphic to $k[\Sigma_n]$ in arity $n$ and degree $1-n$.} algebra $A$.
		
		\[  \begin{split}
			\G_g_{A_\infty, A}  & = \prod_{n\geq 0} \iHom_{\Sigma_n} \left( \mathbf{Ass}^\antishriek (n), \mathbf{End}_A(n) \right) \\
			& = \prod_{n\geq 1} \iHom_{\Sigma_n} \left( k[\Sigma_{n}], \mathbf{End}_A(n) \right)[n-1] \\
			& = \prod_{n\geq 1} \iHom_{k}(A^{\otimes n}, A)[n-1]
		\end{split} \]
		
		When twisting $\G_g_{A_\infty, A}$ with a Maurer--Cartan element (i.e. an associative multiplication), the differential on $\G_g_{A_\infty, A}^m$ is given by $[m, -]$. 
		If we explicitly write out the differential of an element $f$ living in the $n=2$ component, we obtain $$df(a_1,a_2,a_3) = a_1f(a_2,a_3) \pm f(a_1a_2,a_3) \pm f(a_1,a_2a_3) \pm f(a_1,a_2)a_3.$$
	
		This coincides with the Hochschild cochain complex (shifted by $1$) of the algebra $(A,m)$ with the Gerstenhaber bracket, which was mentioned in Example \ref{Ex:intro}.
		
		 From Remark \ref{RQ_DeformationAndCohomology}, if we take $m$ an associative multiplication, the infinitesimal deformations are classified by the first Hochschild cohomology group, and the obstruction to extending to higher formal deformations in controlled by the second Hochschild cohomology group.   
		
		\item The deformation complex for a Lie\footnote{Recall from Example \ref{Ex_InfityAlgebras} that $\mathbf{Lie}^\antishriek (n)$ is $1$-dimensional concentrated in degree $1-n$.} algebra $A$ is computed as follows:
		\[ \begin{split}
			\G_g_{L_\infty, A}  & = \prod_{n\geq 0} \iHom_{\Sigma_n} \left( \mathbf{Lie}^\antishriek(n), \mathbf{End}_{A}(n) \right) \\
			& = \prod_{n\geq 1} \iHom_{\Sigma_n} \left( k, \mathbf{End}_{A}(n) \right)[n-1] \\
			& = \prod_{n\geq 1} \iHom_{k}(\left(\Sym^n A[-1]\right)[n], A)[n-1] \\
			& = \prod_{n\geq 1} \iHom_{k}(\Sym^n (A[-1]), A[-1])
		\end{split} \]
		
		This coincides with the Chevalley--Eilenberg algebra valued in $A$ with trivial Lie algebra structure on $A$. Given a Lie structure on $A$ (corresponding to a Maurer--Cartan element $\phi$ in $ \G_g_{L_\infty, A}$), we obtain that $ \G_g_{L_\infty, A}^\phi$ is the Chevalley--Eilenberg algebra associated to $A$ (with the structure associated to $\phi$) valued in $A$.

	\end{itemize}
\end{Ex}

\begin{Ex}[Deformation quantization of Poisson manifolds]\

Given a Poisson manifold $(M,\{-,-\})$, a (formal deformation) \emph{quantization} of $M$ is a so-called \emph{star-product}, i.e., an $\mathbb R[\![\hbar]\!]$-linear associative product $- \star -$ on formal power series of functions $C^\infty(M)[\![\hbar]\!]$, such that: 

\begin{enumerate}[label=(\roman*)]
	\item The star-product is associative.
	\item The star-product is a deformation of the original product: $f\star g = f \cdot g + O(\hbar)$, for $f,g\in C^\infty(M)$.
	\item The star-product quantizes the Poisson bracket: $f\star g- g \star f = \{f,g\}\hbar  + O(\hbar^2)$. 
\end{enumerate}

Given a star-product on $M$, one can check that 

\[
\{f,g\}_\star = \lim_{\hbar \to 0}\frac{f\star g- g \star f}{\hbar}
\] 
defines a Poisson bracket on $M$. If our original Poisson bracket is obtained from a star-product in this form, we say that  $\star$ \emph{quantizes} $(M,\{-,-\})$. The natural questions being raised are:

\begin{enumerate}
	\item 	Can every Poisson manifold be quantized?
	
	\item 	If so, can one provide explicit formulas?
	
	\item 	Are quantizations unique?
\end{enumerate}


In his seminal paper \cite{kontsevichdefquant}, Kontsevich used the machinery described above to give a complete solution to the previous questions: (1) yes, (2) yes up to computing some integrals and (3) essentially yes.

The first observation is that the data of a star-product is encoded by the (Hochschild) deformation complex of $A=\_C^\infty(M)$ seen as an associative algebra. Ignoring condition (iii), a star-product is a Maurer--Cartan element of $\G_g_{A_\infty, \_C^\infty(M)} \hat\otimes \mathbb R[\![\hbar]\!]$. More precisely, requiring the components of the star-product to be multidifferential operators leads us to consider the Lie subalgebra $D_\mathrm{poly}$, where
$$D_\mathrm{poly}^{n-1}(M)  =  \left\lbrace D\colon \_C^\infty(M)^{\otimes n} \to \_C^\infty(M) \Big|  D \stackrel{\tx{loc.}}{=} \sum f\frac{\partial}{\partial x_{I_1}}\otimes \dots\otimes \frac{\partial}{\partial x_{I_n}}\right\rbrace.$$

On the other hand, the datum of a Poisson bracket is equivalently given by a bivector field $\pi \in \Gamma(M,T_X \wedge T_X)$, such that $[\pi,\pi]=0$ with respect to the Schouten--Nijenhuis bracket. One can therefore interpret the Poisson structure as a Maurer--Cartan element in the Lie algebra (with $d=0$) of polyvector fields $T_\mathrm{poly}(M)$, where $T_\mathrm{poly}^{-1} (M)= \_C^\infty(M)$ and $T_\mathrm{poly}^{n-1}(M) = \Gamma(M,T_X^{\wedge n})$.
\begin{Th}[Kontsevich Formality, \cite{kontsevichdefquant}]
	For every manifold $M$, there exists an $\infty$-quasi-isomorphism
	\[
	\mathcal U \colon T_\mathrm{poly}(M) \rightsquigarrow D_\mathrm{poly}(M).
	\]
	The first component of the $\infty$-morphism is the Hochschild--Kostant--Rosenberg (HKR) map, which is obtained by de-symmetrization. 
	The higher components are parametrized by certain kinds of graphs, with coefficients given by integrals over configuration spaces of points on the upper half-plane.
\end{Th}

It now follows from the Goldman--Milson Theorem \ref{Th_GoldmannMilson} that for every Poisson bivector $\pi$, its image under $\mathcal U$ is a star-product quantizing it. Furthermore, any such quantization is unique up to gauge equivalence.

\end{Ex}

\paragraph{\emph{Functoriality of the deformation complex:}}\

\medskip

The construction of the deformation complex is functorial in the following sense:
Suppose $f\colon \_Q \to \_P$ is a map of Koszul operads, induced by a map from the generators of $\_Q$ to the generators of $\_P$. Then, there is an associated map of cooperads $f^\antishriek\colon \_Q^\antishriek \to \_P^\antishriek$, as well as the Koszul dual map $f^!\colon \_P^!\to \_Q^!$. Then, $(f^\antishriek,f)$ form a map between the canonical twisting morphisms as in Remark \ref{RQ_NaturalityAlgebraicBarCobar}. The map $f^\antishriek$ induces a Lie algebra morphism at the level of the respective deformation complexes:
\[
\mathfrak g_{\_P_\infty,A} \to \mathfrak g_{\_Q_\infty,A}
\]

\begin{RQ}\label{Ex:ComAssLie def complex}

	The morphisms of operads: 
	\[ \mathbf{Lie} \rightarrow \mathbf{Ass} \rightarrow \mathbf{Com}\] 
	which are the same as: 
	\[ \mathbf{Com}^! \rightarrow \mathbf{Ass}^! \rightarrow \mathbf{Lie}^!\] 
	
	are dual (up to some shifts) to:
	\[ \mathbf{Lie}^\antishriek \rightarrow \mathbf{Ass}^\antishriek \rightarrow \mathbf{Com}^\antishriek\]
	
	This induces morphisms between the deformation complexes:
	\[ \G_g_{\mathbf{Com}_\infty, A} \rightarrow \G_g_{\mathbf{Ass}_\infty, A} \rightarrow \G_g_{\mathbf{Lie}_\infty, A}\] 
	
	If $A$ is a commutative algebra, applying the Maurer--Cartan functor to the first map recovers the usual inclusion from the Harrison complex into the Hochschild complex. Similarly, if $A$ is an associative algebra, applying the Maurer--Cartan functor to the second map yields the classical projection of the Hochschild complex into the Chevalley--Eilenberg complex of the Lie algebra structure on $A$.

\end{RQ}

\subsection{A Deformation Approach to $\_P_\infty$-Algebras}\
\label{Sec_CommutativeAndAssociativeStructures}

\medskip

In this section we will show how the operadic and deformation tools introduced previously apply to address a problem that is not (or at least does not seem) deformation theoretical in nature.

A bit more concretely, assume $\_P$ is a Koszul operad and let $A$ and $B$ be  two $\_P$-algebras. We consider the following questions:

	\begin{Que}\label{Que:w.e.}\ 
			\begin{itemize}
		\item Weak equivalence: Do we have a zig-zag of quasi-isomorphisms of $\_P$-algebras from $A$ to $B$? We will denote this equivalence relation by $A \sim_{\_P} B$.

		\item Formality: Is $A \sim_\_P H(A)$?
		
	\end{itemize}
	\end{Que}
	
	Given a map of operads $f\colon \_Q \to \_P$ and a $\_P$ algebra $A$, to avoid excessive notation we will refer to $f^*(A)$ as ``$A$ seen as a $\_Q$-algebra'' (see Proposition \ref{Prop_NaturalityAlgebraAdjunction}). We will be particularly interested in seeing a commutative algebra as just an associative algebra. This abuse of notation justifies the $\_P$ in  $\sim_\_P$, which would otherwise be redundant.

\begin{RQ}
Let us consider the unit operad $I=I(1)=k$. Every $I$-algebra is given by some $A \in \chk$ and is formal since $k$ a field. In particular, when forgetting the $\_P$-structures (via the unit $I \rightarrow \_P$), every $\_P$-algebra is equivalent to its homology as a cochain complex, but not as a $\_P$-algebra a priori.
\end{RQ}

Recall that Theorem \ref{Th_MainPropertiesOfInftyMaps} implies in particular that the weak-equivalence problem can be fully reformulated in terms of $\infty$-morphisms.

\begin{Prop}\label{prop:w.e. equiv to infty map}
	
	$A \sim_{\_P} B$ if and only if there is a $\_P_\infty$-quasi-isomorphism $A \rightsquigarrow B$. 
\end{Prop}
%

This allows us to simplify the question of whether $A$ and $B$ are weakly equivalent $\_P$-algebras:
First, supposing $H(A) \cong H(B)$ (otherwise the result is trivial), the homotopy transfer theorem allows us to intepret the $\_P$-algebra structure of both $A$ and $B$ as a $\_P_\infty$ structure in the same graded vector space $H\coloneqq H(A) \cong H(B)$, which we denote $\alpha$ and $\beta$ respectively.
Secondly, Proposition \ref{prop:w.e. equiv to infty map} implies that the existence of such a weak equivalence is equivalent to an $\infty$-quasi-isomorphism:
	\[ f: (H, \alpha) \rightsquigarrow (H, \beta) \]
 notice that since $H$ has null differential, an $\infty$-quasi-isomorphism $H\rightsquigarrow H$ is just an $\infty$-isomorphism. We can in fact suppose that $f_1\colon H \to H$ is $\id_H$ and therefore get an $\infty$-isotopy (Definition \ref{Def_InftyIsotopy}).
	
We can therefore reformulate Questions \ref{Que:w.e.} as follows:

\begin{Refo}\
			\begin{itemize}
	\item Weak equivalence: Given two $\_P_\infty$-algebra structures $\alpha$ and $\beta$ on $H$, are they $\infty$-isotopic?
	
	\item Formality: Is the transferred $\_P_\infty$ structure on $H(A)$ $\infty$-isotopic to a strict one?	
\end{itemize}
\end{Refo}


 The following result gives as a deformation theoretical interpretation of the reformulation above.

\begin{Th}[{\cite[Theorem 3]{DSV16}}] \label{Th_GaugeActionAndIsotopies}
	
	The action of the gauge group of $\G_g_{\_P_\infty, A}$ on $\tx{MC}(\G_g_{\_P_\infty, A})$ corresponds to the action of $\infty$-isotopies $f : A \rightsquigarrow A$ on the space of $\_P_\infty$ structures on $A$. 
In other words we have:
\[
\underline{\mathrm{Del}}(\G_g_{\_P_\infty, A}) \cong (\_P_\infty\text{- }\mathrm{algebra\ structures\ on\ }A, \infty\text{- }\mathrm{isotopies})
\]

\end{Th}

To finish the section let us  mention briefly what kind of tools go into the proof of this result. 
The proof of this theorem makes use of the fact that the Lie bracket on the deformation complex $\G_g_{\_P_\infty, A}$ arises as the anti-symmetrisation of a pre-Lie product $\star$, as defined in \ref{Def_TotatlizationOfAnOperad}. 
	
In \cite{DSV16}, Dotsenko--Shadrin--Vallette develop a method of pre-Lie exponentials, making use of the fact that the expression $\sum\limits_k \frac{X^{\star k}}{k!} \coloneqq \sum\limits_k \frac{(((X\star X)\dots )\star X}{k!}$ makes sense for complete pre-Lie algebras, in order to produce explicit formulas for the gauge group product and action. 
The explicit formula for the exponential shows that the image consists indeed of the $\infty$-isotopies. Together with the explicit  computation of \cite[Lemma 2]{DSV16}, one shows that the gauge group structure (the BCH formula) agrees with the composition of $\infty$-isotopies, noted $\circledcirc$. Using an argument of uniqueness of solutions of formal differential equations, one can show that the gauge group acts indeed by $\infty$-isotopies, via the formula $\lambda \cdot \alpha = (e^\lambda\star \alpha)\circledcirc e^{-\lambda}$, see \cite[Proposition 5]{DSV16}.

%
%

\subsubsection{Commutative and associative algebras}\

\medskip

Let us start this section with an example arising from rational homotopy theory \cite{Qu69,Su77}, which is at first unrelated to any $\infty$-structures.
\begin{Ex}
	In rational homotopy theory à la Sullivan, we consider the functor $\Omega: \mathbf{Top} \rightarrow \cdga^{\Qq/}$ of rational differential forms. One might naively wonder whether we could consider instead the simpler functor of rational singular cochains $C^*(-,\Qq)$ landing in $ \dga^{\Qq/}$. Despite being quasi-isomorphic, it is unclear whether this functor remembers the homotopy type of the space.  This is because Corollary 10.10 in \cite{FHT01} only gives a zig-zag of natural quasi-isomorphisms between the singular rational cochains $C^*(-,\Qq)$ functor and $F \circ \Omega$ where $F$ is the forgetful functor $ \cdga^{\Qq/} \rightarrow  \dga^{\Qq/}$. It raises the following question: given $X, Y \in \mathbf{Top}$ such that $C^*(X, \Qq) \sim_{ \dga^{\Qq/}} C^*(Y,\Qq)$, are they rationally equivalent, i.e. do we have $\Omega(X) \sim_{ \dga^{\Qq/}} \Omega (Y)$? Thanks to Corollary 10.10 in \cite{FHT01}, this reduces to proving the following implication:   
	\[ \Omega (X) \sim_{ \dga^{\Qq/}} \Omega (Y) \Rightarrow \Omega (X) \sim_{ \cdga^{\Qq/}} \Omega (Y) \]

\end{Ex}

This problem is actually algebraic. Motivated by this example, we consider the following question: 

\begin{Que}\label{Que:CPRW}
	If $A$, $B$ are  commutative algebras over a field of characteristic zero, do we have:
\[ A \sim_{\mathbf{Ass}} B \Rightarrow A \sim_{\mathbf{Com}}B \]
\end{Que}

In other words, if there is a zig-zag of quasi-isomorphisms passing by associative algebras, must there be one passing by commutative algebras?
A first attempt to show this directly could be to abelianize the associative zig-zag. Unfortunately, abelianization does not commute with taking homology, as we saw in Warning \ref{war:abelianization}.
 Nevertheless, the answer to the question is yes.

\begin{Th}[{\cite[Theorem (A) and (B)]{CPRW19}}] \label{Th_CommutativeAndAssociativeWeakEquivalences}
	Two (non-)unital commutative algebras $A$ and $B$ are quasi-isomorphic if and only if they are quasi-isomorphic as (non-)unital associative algebras.
\end{Th}

 At first this question might look too simple to have a difficult answer. The answer being yes, one might wonder whether this is a formal consequence of some simple property of the (fully faithful) forgetful functor $F$ from commutative to associative algebras, which is also induced from the map of operads $\mathbf{Ass} \to \mathbf{Com}$.
This does not seem to be the case: While it is indeed true that $F$ induces a faithful functor at the level of the homotopy categories (this not trivial, see \cite{CPRW22}), $h(F)$ is not full as can be seen in the following example.

\begin{Ex}
	Consider the algebra $A= k[x,y]$ and a commutative algebra $S$. Then, in the homotopy category, we have: 
	\[ \Hom_{h\mathbf{Alg}_{\mathbf{Com}}}(k[x,y], S) \sim H^0(S) \oplus H^0(S)\] 
	\[ \Hom_{h\mathbf{Alg}_{\mathbf{Ass}}}(k[x,y], S) \sim H^0(S) \oplus H^0(S) \oplus H^{-1}(S)\]
	
	This comes from the fact that $k[x,y]$ is not cofibrant as an associative algebra. It needs to be replaced by the quasi-free associative algebra generated by $x, y$ and $\epsilon$ of degree $-1$ such that $d\epsilon = xy - yx$. 
\end{Ex}

Making use of Reformulation \ref{Th_GaugeActionAndIsotopies} and denoting $H\coloneqq H(A) \cong H(B)$, we can also reformulate this question as follows.

\begin{Refo}
 If $\alpha$ and $\beta$ are two $C_\infty$-structure on a graded vector space $H$ that are $A_\infty$ isotopic, are they $C_\infty$-isotopic?
\end{Refo}

To answer that question, using Theorem \ref{Th_GaugeActionAndIsotopies}, we need to ask whether an associative gauge transformation $\lambda$ between $\alpha$ and $\beta$ in $\underline{\tx{Del}}\left( \G_g_{A_\infty, H}\right)$ (corresponding to the two $A_\infty$ structures on $H$) can be lifted to a commutative gauge $\tilde{\lambda}$ in $\underline{\tx{Del}}\left(\G_g_{C_\infty, H} \right)$ such that the image of $\tilde{\lambda}$ through the map between the Deligne groupoids induced by the map $\G_g_{C_\infty, H} \rightarrow \G_g_{A_\infty, H}$ from Example \ref{Ex:ComAssLie def complex} is $\lambda$. If $\G_g_{C_\infty, H} \rightarrow \G_g_{A_\infty, H}$ were a quasi-isomorphism, then our conclusion would follow from the Goldmann--Milson Theorem \ref{Th_GoldmannMilson}. In our case, the map is injective and it even remains injective in homology upon twisting by a $C_\infty$ structure, but this does not generally suffice to obtain an injective map at the level of gauge equivalence classes of Maurer--Cartan elements. 
The following theorem is the appropriate ``injectivity'' refinement of the Goldmann--Milson Theorem.

	\begin{Th}[Theorem 1.7 in \cite{CPRW19}]\label{Th:CPRW gauge}
		Let $i: \G_h \hookrightarrow \G_g $ a map of filtered complete differential graded Lie algebra. Suppose that $i$ has a retract $r$ as $\G_h$-modules (i.e. $r [i(h), x] = [h, r(x)]$).  Then the following map is injective: 
		\[ \faktor{\tx{MC}(\G_h)}{\sim} \hookrightarrow  \faktor{\tx{MC}(\G_g)}{\sim}\] 
	\end{Th}
		
	\begin{proof}[Idea of proof of Theorem \ref{Th_CommutativeAndAssociativeWeakEquivalences}]
	Admitting Theorem \ref{Th:CPRW gauge}, we need to construct such a retract of $i\colon \G_g_{C_\infty, H} \to \G_g_{A_\infty, H}$. If there existed a retract of $f\colon \mathbf{Lie}=\mathbf{Com}^! \rightarrow \mathbf{Ass}=\mathbf{Ass}^!$ as operads, then it would induce a retract of $i$ as Lie algebras. The much weaker existence of a retract of $f$ as symmetric sequences, essentially only tells us about the injectivity of $i$.
	
	It turns out that to obtain a retract as $\G_g_{C_\infty, H}$-modules, it suffices to construct retract of $\mathbf{Lie} \rightarrow \mathbf{Ass}$ as infinitesimal $\mathbf{Lie}$-bimodules, which is a type of structure whose strength sits between the two. 
	
	To obtain such a retract, we can use a suitable variant of the Poincaré--Birkhoff--Witt theorem \ref{Th:PBW} to express $\mathbf{Ass} \cong \Sym(\mathbf{Lie})$ and then projecting onto the summand $\Sym^1(\mathbf{Lie})=\mathbf{Lie}$. The reader can find the details in \cite[Section 2]{CPRW19}.
	\end{proof}
	
We point out that in Theorem \ref{Th:CPRW gauge}, the gauge equivalence $\tilde \lambda\in \exp(\mathfrak h)$ one constructs is not just $r(\lambda)$. It is rather obtained by an iterative procedure which is not fully explicit. Therefore, while the answer to Question \ref{Que:CPRW} is yes, even given explicit deformation retracts of $A$ and $B$ into $H$, there is no known constructive way of exhibiting such a zig-zag.

\subsubsection{Lie algebras and their enveloping algebras}\

	\medskip
	
Let us consider a seemingly unrelated question to the previous section:
\begin{Que}
	Can we recover a Lie algebra from its universal enveloping algebra?
\end{Que}

Recalling the digression in Section \ref{sec:Lie digression}, the immediate answer should be yes: $\mathfrak g$ is isomorphic (as a Lie algebra) to the primitive elements of $\mathfrak U \mathfrak g$.
However, this makes use of the Hopf algebra structure. 
Interpreting $\mathfrak U$ as a functor into associative algebras, we can rephrase the previous question more similarly to Question \ref{Que:CPRW}

\begin{Que}\label{Que:CPRW2}
If $\mathfrak g$ and $\mathfrak h$ are Lie algebras, do we have:

$$\mathfrak U \mathfrak g \sim_\mathbf{Ass} \mathfrak U \mathfrak h \Rightarrow \mathfrak g \sim_\mathbf{Lie} \mathfrak h.$$
\end{Que}

Let us start by pointing out two differences between Questions \ref{Que:CPRW} and \ref{Que:CPRW2}. The first one is that the converse of Question \ref{Que:CPRW} is obviously true: the same zig-zag exhibiting the equivalence between two commutative algebras exhibits their equivalence as associative algebras. 
The same argument holds for Question \ref{Que:CPRW2}, as soon as we know the much less trivial fact that universal enveloping algebras preserve quasi-isomorphisms. Indeed, by the PBW Theorem \ref{Th:PBW}, as a cochain complex $\mathfrak U \mathfrak g\cong \Sym \mathfrak g$ and the symmetric algebra functor preserves quasi-isomorphisms. Notice that both of these facts require $\mathrm{char}\ k = 0$.

The second difference is that in the classical (i.e. non-dg) setting, Question \ref{Que:CPRW2} is non-trivial. In fact, in characteristic $0$ the following is still an open question:

\begin{Que}[The isomorphism problem for enveloping algebras]\label{Que:iso problem}
	Let $\mathfrak g$ and $\mathfrak h$ be Lie algebras in vector spaces, such that the associative algebras $\mathfrak U \mathfrak g$ and $\mathfrak U \mathfrak h$ are isomorphic. Are $\mathfrak g$ and $\mathfrak h$ necessarily isomorphic?
\end{Que}

Question \ref{Que:CPRW} concerns the restriction functor associated to the operad morphism $f\colon \mathbf{Ass} \to \mathbf{Com}$, while Question \ref{Que:CPRW2} concerns the induction functor associated to the Koszul dual morphism $f^!\colon \mathbf{Lie} \to \mathbf{Ass}$. In a sense that we will make precise, Koszul duality intertwines these two functors in a way that makes the two questions almost equivalent. The word almost is present because Koszul duality, as presented in Section \ref{Sec_KoszulDuality} is only an equivalence between $\_P$-algebras and $\_P^\antishriek$-coalgebras and some nilpotence assumptions are required in order to have an equivalence between $\_P$-algebras and $\_P^!$-algebras (see Warning \ref{war:dual is not dual}).

\begin{Th}\label{Th:CPRW2}
	Let $\mathfrak g$ and $\mathfrak h$ be Lie algebras in non-negative degrees that are either nilpotent or concentrated in strictly positive degrees.
	Then we have:
	\[\mathfrak U \mathfrak g \sim_\mathbf{Ass} \mathfrak U \mathfrak h \Rightarrow \mathfrak g \sim_\mathbf{Lie} \mathfrak h\]
\end{Th}

\begin{proof}
	We consider the bar-cobar adjunction associated to the twisting morphism $\alpha : \_P^\antishriek \dashrightarrow \_P$ (Definition \ref{Def_BarCobarAlgebra}) for $\_P = \mathbf{Lie}$ and $\_P=\mathbf{Ass}$:
	\[ \begin{tikzcd}
		\Omega_\alpha:  \mathbf{coAlg}_{\_P^\antishriek}^{\tx{conil}}  \arrow[r, shift left] & \arrow[l, shift left] \mathbf{Alg}_{\_P}  : \mathbf{B}_{\alpha}
	\end{tikzcd}\]
	
	Recall that we have a morphism of twisting morphisms (Definition \ref{Def_TwistingMorphisms})
	\[\begin{tikzcd}
		\mathbf{Com}^\vee\lbrace -1 \rbrace \simeq \mathbf{Lie}^\antishriek \arrow[r, dashed, "\alpha"] \arrow[d] & \mathbf{Lie} \arrow[d]\\
		\mathbf{Ass}^\vee\lbrace -1 \rbrace \simeq \mathbf{Ass}^\antishriek \arrow[r, dashed, "\beta"]& \mathbf{Ass}
	\end{tikzcd}\]

the induces the  following diagram of adjunctions thanks to the naturality of the bar-cobar construction (see Remark \ref{RQ_NaturalityAlgebraicBarCobar}):
\[ \begin{tikzcd}
	\mathbf{coAlg}_{\mathbf{Lie}^\antishriek} \arrow[r, "\Omega_\alpha", shift left] \arrow[d, shift left, "F"] & \mathbf{Alg}_{\mathbf{Lie}} \arrow[l, shift left, "\mathbf{B}_\alpha"] \arrow[d, shift left,  " \G_U"] \\\mathbf{coAlg}_{\mathbf{Ass}^\antishriek} \arrow[u, shift left] \arrow[r, shift left, "\Omega_\beta"] & \mathbf{Alg}_{\mathbf{Ass}} \arrow[l, shift left, "\mathbf{B}_\beta"] \arrow[u, shift left] 
\end{tikzcd}\]

The diagram of left (respectively right) adjoint commute and therefore $\G_U \circ \Omega_\alpha = \Omega_\beta \circ  F$, where $F$ the is the forgetful functor. 

Let us now suppose $\mathfrak U\mathfrak g \sim_\mathbf{Ass} \mathfrak U\mathfrak h$. The rest of the proof decomposes into the following cochain of implications explained below:

\begin{enumerate}	
	\item \label{1} $\mathfrak U\Omega_\alpha \mathbf{B}_\alpha \mathfrak g \sim_{\mathbf{Ass}} \mathfrak U\Omega_\alpha \mathbf{B}_\alpha \mathfrak h$.
	
	\item $\Omega_\beta \mathbf{B}_\alpha \mathfrak g\sim_{\mathbf{Ass}} \Omega_\beta \mathbf{B}_\alpha \mathfrak h.$
	
	\item\label{3}  $\mathbf{B}_\alpha \mathfrak g\sim_{\mathbf{Ass}^\antishriek} \mathbf{B}_\alpha  \mathfrak h$.
	
	\item\label{4} $\mathbf{B}_\alpha  \mathfrak g\sim_{\mathbf{Lie}^\antishriek} \mathbf{B}_\alpha\mathfrak h$.

		\item \label{5} $\mathfrak g\sim_\mathbf{Lie}\mathfrak h$.

%
\end{enumerate}

The first point follows from the PBW theorem \ref{Th:PBW},  which implies that $$\mathfrak U \Omega_\alpha \mathbf{B}_\alpha \mathfrak g  \stackrel{\sim}{\to} \mathfrak U\mathfrak g$$ is a resolution. The second point follows from the equality $\G_U \circ \Omega_\alpha = \Omega_\beta \circ  F$ established before. 

The third point follows from applying $\mathbf{B}_\beta$ and the fact that $\id \to \mathbf{B}_\beta \Omega_\beta $ is a natural quasi-isomorphism since $\mathbf{Lie}$ is Koszul and using Proposition \ref{RQ_UnitCouniQIAlgebraBarcoBar}. Notice that we are also using that bar constructions preserve quasi-isomorphisms (see Proposition \ref{RQ_UnitCouniQIAlgebraBarcoBar}), since we are applying $\mathbf{B}_\beta$ to the whole zig-zag.

We find ourselves in a situation where we have two (shifted) cocommutative algebras that are quasi-isomorphic only as (shifted) coassociative algebras. The same methods used to prove Theorem \ref{Th_CommutativeAndAssociativeWeakEquivalences} can be applied in the coalgebra setting to show that such coalgebras must also be quasi-isomorphic as cocommutative algebras \cite[Theorem 4.27]{CPRW19}, showing point \ref{4}.

Finally, to deduce point \ref{5} we would like to proceed as in point \ref{3} and apply $\Omega_\alpha$, but unfortunately cobar constructions do not preserve quasi-isomorphisms in general.

What one can instead do is to apply the \emph{complete} cobar construction $\Omega_\alpha^\wedge$, which does preserve quasi-isomorphisms. The resulting Lie algebras are no longer quasi-isomorphic to $\mathfrak g$ and $\mathfrak h$, but rather to their \emph{homotopy completions}. The hypotheses of the theorem are such that the homotopy completions of $\mathfrak g$ and $\mathfrak h$ are equivalent to $\mathfrak g$ and $\mathfrak h$, thus concluding the proof.
\end{proof}

\needspace{6\baselineskip}

	\section{Formal Moduli Problems and Koszul Duality}\label{Sec_Operadic Deformation Theory}

This section aims at explaining recent results on ``operadic deformation theory'' and extend the classical philosophy that the space of deformations is classified by a Lie algebra.  \\

We start with the Section \ref{Sec_Classical Deformation Theory} where we define deformations from a geometric point of view and recall the notion of infinitesimal deformations, and their relationship with the tangent complex (Section \ref{Sec_CotangentComplexHigherDeformation}). In Section \ref{Sec_Deformation from the Algebraic Point of View} we will look at deformations of algebras and compare with what was studied in Section \ref{Sec_DeformationofAlgebraicStructures}.\\

Then Section \ref{Sec_Formal Moduli Problem} sets up a general context to describe deformation theories (following \cite{Lu11}). We will see a generalization of the notion of infinitesimal objects using the formalism of small objects and morphisms. This leads to the definition of \emph{formal moduli problem} (Definition \ref{Def_FMP}) which gives an axiomatic framework for what a deformation functor\footnote{Essentially, a deformation functor is a functor that sends a ``small'' deformation parameter (e.g. $k[\epsilon]$) to the ``space of all deformations'' along this parameter space.} is in a given deformation context. Formal moduli problems, assemble into an $\infty$-category noted $\mathbf{FMP}$. 

With the right definitions in place, this allows us to formalize the classical slogan that deformation problems correspond to Lie algebras into a rigorous equivalence of $\infty$-categories. 


\begin{restatable*}[\cite{Lu11,Pr10}]{Th}{LuriePridhamEquivalence}
	\label{Th_LuriePridham}
	
	There is an equivalence of $\infty$-categories: 
	\[ \begin{tikzcd}[row sep = tiny, ampersand replacement=\&]
 \mathbf{MC} \colon	\mathbf{Alg}_{\mathbf{Lie}} \arrow[r,shift left, "\sim"] \& \arrow[l,shift left] \mathbf{FMP} \colon \Tt[-1]  
	\end{tikzcd}\]
	

\end{restatable*}

This equivalence  is roughly given by the Deligne $\infty$-groupoid in one direction and by the shifted tangent space in the other direction.

	As we will see, this is in fact an instance of the Koszul duality between the operads $\mathbf{Com}$ (deformations are parametrized by commutative algebras) and $\mathbf{Lie}$.  In fact, we can consider deformations parameterized by more general algebraic structures. \\

Indeed, in Section \ref{Sec_Operadic Formal Moduli Problems and Koszul Duality} we describe the categories of formal moduli problems over a class of deformation contexts for $\_P$-algebras. We finish by providing the tools that go into understanding and showing the following generalization of Theorem \ref{Th_LuriePridham}:
	
	\begin{restatable*}[\cite{CCN20}]{Th}{KoszulOperadicFMPEquivalence}\label{Th_KoszulOperadicFMPEquivalence}
		For $\_P$ a Koszul binary quadratic operad concentrated in non-positive degrees, there is an equivalence of $\infty$-categories: 	
		\[  \mathbf{FMP}_\_P \simeq \mathbf{Alg}_{\_P^!}. \]	
	\end{restatable*}

%
%

\subsection{Classical Deformation Theory}
\label{Sec_Geometric Deformation Theory}
\subsubsection{Deformation theory of varieties}\
\label{Sec_Classical Deformation Theory}

\medskip

	Let us start by reviewing classical deformation theory. We refer the readers to \cite{Ha09}, as well as the classical books \cite{Il71, Il72}
	 First of all, we recall that classically, the space along which we deform  is often chosen to be the spectrum of a local Artinian algebra. Such algebras are infinitesimal extensions of a point and therefore will encode infinitesimal deformations.
	\begin{Def}\label{Def_SmallComAlgebra}
		
		The category $\art$ of \defi{classical local Artinian $k$-algebras} with residue field $k$ is the category of local non-graded commutative $k$-algebras $A$, with residue field $k$, such that the maximal ideal satisfies $ \G_M_A^n = 0$ for $n \gg 0$.

	\end{Def}

The adjective ``classical'' in the definition above is non-standard. The reason for this is to distinguish it from the differential graded version of Artinian algebras which will be central starting from Section \ref{Sec_Formal Moduli Problem}.
	
	\begin{RQ}
		Geometrically, Artinian algebras contain only one closed point (given by the augmentation $A\to k$), and they correspond to infinitesimal  neighborhoods of this point. For example, the algebras $k[x]/(x^n)$ are classical local Artinian algebras. 
\end{RQ}

The following definition is given for $k$-schemes but the reader can keep in mind the example of proper smooth $\Cc$-varieties.

\begin{Def}[Deformation of Algebraic Schemes] \label{Def_DeformationOfCVariety}

	Let $X$ be an algebraic scheme (locally finitely generated $k$-scheme). Let $(S, s)$ be a  pointed scheme. Then, an $(S,s)$ deformation is given by a scheme $\widetilde{X}$ together with a flat surjective morphism $\pi: \widetilde{X} \rightarrow S$ such that the following square is a pullback: 
	\[ \begin{tikzcd}
		X \arrow[r] \arrow[d] & \widetilde{X} \arrow[d, "\pi"] \\
		\star \arrow[r, "s"] & S
	\end{tikzcd}\] 
Deformations along $\tx{Spec}\left( k[\epsilon] \right)$, where $\epsilon^2=0$, are called \defi{first-order} deformations.
More generally, we say that a deformation is \defi{infinitesimal} if $S = \tx{Spec}(A)$ with $A$ a local Artinian $k$-algebra with residue field $k$. 
	
	Two deformations $\widetilde{X}_1$ and $\widetilde{X}_2$ are said to be isomorphic if there is an isomorphism of schemes $\phi : \widetilde{X}_1 \rightarrow \widetilde{X}_2$ making the following diagram commute: 
	\[ \begin{tikzcd}
		& X \arrow[dl] \arrow[dr] & \\
		\widetilde{X}_1 \arrow[dr]  \arrow[rr, "\phi"]& & \arrow[dl] \widetilde{X}_2 \\
		& S & 
	\end{tikzcd}\]
	
	Observe that this implies that $\phi$ induces the identity on the pullback, $X$. 
\end{Def}

\begin{RQ}\label{RQ_DeformationAlgebraicGeometry}\
	
	\begin{itemize}
%
		
		\item If $X = \tx{Spec}(B)$ and $S = \tx{Spec}(A)$ are affine and $A$ is a classical local Artinian algebra, then it is shown in \cite[Section 1.2.2]{Se07} that any deformation $\widetilde{X}$ is also affine. We can rephrase the notion of deformation of an affine space by saying that a deformation of $B$ along $A$ is the data of a commutative algebra $B'$ together with a flat map $A \rightarrow B'$ and an isomorphism $ B' \otimes_A k \overset{\sim}{\to} B$.  
		
		As before, an equivalence of two such deformations is given by an isomorphism $\phi : B_1' \rightarrow B_2'$ of $A$-algebras making the following diagram commute: 
		\[ \begin{tikzcd}
			& A\arrow[dl] \arrow[dr] & \\
			B_2' \otimes_{A} k  \arrow[dr, "\sim"']  \arrow[rr, "\phi"]& & \arrow[dl,"\sim"] B_1' \otimes_{A} k \\
			& B & 
		\end{tikzcd}\]
		
		\item Smooth algebras are rigid. If $B$ is a smooth (non-graded) algebra, then Corollary 4.8 in \cite{Ha09} 
		shows that if $A$ is a classical local Artinian algebra, then any deformation of $B$ along $A$ is trivial, that is, any deformation $B'$ is equivalent to $B\otimes_k A$.   
	\end{itemize}
\end{RQ}

Let us now mention two classical results on deformations of complex manifolds that will motivate the relationship between first-order deformations, the tangent space and a set of Maurer--Cartan elements. These relations will then be explored in detail and generalized later on.
\begin{Th}[Kodaira--Spencer, \cite{KS58}] \ 
	First order deformations (for $A = \Cc[\epsilon]/\epsilon^2$) of a complex analytic manifold are controlled by {the cohomology of $X$ with values on the tangent bundle $H^*(X,T_X)$, namely: }
	\[ H^1(X, T_X)  \overset{\tx{1:1}}{\longleftrightarrow} \tx{deformations \ along } \ \tx{Spec} \left(\Cc[\epsilon]\right) \] 
	\[ H^2(X, T_X)  \overset{\tx{1:1}}{\longleftrightarrow} \tx{Obstructions \ to \ extending \ the \ deformation \ to \ higher \ order.} \] 

%

\end{Th}
	This statement is a geometric version of \cite[Theorem 12.2.14]{LV} and Remark \ref{RQ_DeformationAndCohomology}. We will see the relationship between algebraic and geometric deformations in Section \ref{Sec_Deformation from the Algebraic Point of View}.
In this setting, the Lie algebra controlling infinitesimal deformations of the complex manifold $X$ is precisely the Dolbeault complex $\Omega^{0, *}(T_X)$, with degree and differential coming from the differential forms, and inheriting a bracket from the Lie bracket of vector fields.

	\begin{Th}[{\cite[Theorem V.55]{Ma05}}]\label{prop:from manetti}
	Given a classical local Artinian algebra $A$ with maximal ideal $\G_M_A$ we have an isomorphism of groupoids:
		\begin{equation*}
			\begin{tikzcd}[column sep = small]
				\arrow[loop left, "\tx{Aut}"] \left\lbrace \tx{Deformations \ of} \ X \  \tx{ along }\ A \right\rbrace \arrow[r, leftrightarrow] & \tx{MC}\left( \Omega^{0, *} (T_X) \otimes \G_M_A  \right) \arrow[loop right, "\tx{Gauge}"]
			\end{tikzcd} 
		\end{equation*}
\end{Th}

\subsubsection{Cotangent complex and higher deformations}\
\label{Sec_CotangentComplexHigherDeformation}

\medskip

As we have seen in Proposition \ref{Prop_Pi0DeligneGroupoid}, the right hand side of Theorem \ref{prop:from manetti} is part of the full Deligne $\infty$-groupoid. The construction of $\mathbf{Del}$ involves tensoring the Lie algebra $\Omega^{0, *} (T_X) \otimes \G_M_A$, living in non-negative degrees with differential forms, also living in non-negative degrees. Since the Maurer--Cartan equation lives in degree $2$, even for the Deligne $\infty$-groupoid we only see the degrees $0$, $1$ and $2$ of $\Omega^{0, *} (T_X)$ playing a part. It is natural to wonder what is the role of the higher degrees of $\Omega^{0, *} (T_X)$, or more generally, of any Lie algebra controlling a deformation problem. In short, these higher degrees control \emph{derived deformations}.\\

This section will explore the relationship between infinitesimal deformation, square zero extensions, (higher) infinitesimal deformations, higher $\mathbf{Ext}$ groups and the cotangent complex. 

\begin{Def}\label{Def_SquareZeroExtension}
	A \defi{square zero extension} of $B \in \mathbf{Alg}_{\mathbf{Com}}$ is a commutative algebra $B' \in \mathbf{Alg}_{\mathbf{Com}}$ together with a surjection $\pi : B' \rightarrow B$ that fits in an exact sequence in $\mathbf{Alg}_{\mathbf{Com}}$:
	\[ \begin{tikzcd}
		0 \arrow[r] & I \arrow[r] & B' \arrow[r, "\pi"] & B \arrow[r] & 0 
	\end{tikzcd}\]
	
	such that $I^2 = 0$. This is an algebraic version of the extension of a scheme described in \cite[Section 1.1.3]{Se07}. Moreover, if $B \in  \mathbf{Alg}_{\mathbf{Com}}^{A/}$, i.e. $B$ comes with a map $A\to B$ for some fixed $A \in \mathbf{Alg}_{\mathbf{Com}}$, then the same definition makes sense with all objects under $A$.

	Given any $B$-module $M$, we can construct the \defi{trivial square zero extension} ${B \oplus M}$ given by the direct sum as a $B$-module, and with the product defined by $(b, m)\times (b', m') = (bb', bm' + b'm)$. This square zero extension corresponds to the split exact sequence:  
	\[ \begin{tikzcd}
		0 \arrow[r] & M \arrow[r] & B\oplus M \arrow[r, "\pi"] & B \arrow[r] & 0 
	\end{tikzcd}\]
	
	A \defi{morphism between two square zero extensions} $B_1'$ and $B_2'$  is a morphism $\phi : B_1' \rightarrow B_2'$ that commutes with the projections on $B$. 
\end{Def} 

\begin{Ex}\label{Ex:sq zero}
Any iterated square zero extension of $k$ by a finite dimensional vector space is classical local Artinian.

Conversely, any classical local Artinian algebra can be obtained as an iterated square zero extension of the trivial algebra $k$. Indeed, if $A$ is such that  $\G_M_A^n = 0$ and $\G_M_A^{n-1}\ne 0$, then: 

\[
0\to \G_M_A^{n-1} \to A \to \faktor{A}{\G_M_A^{n-1}}\to 0
\]

exhibits $A$ as a square zero extension of an algebra of lower nilpotence.

As a particular example, the quotient $\faktor{k[x]}{(x^n)} \to \faktor{k[x]}{(x^{n-1})}$ is a square zero extension.

\end{Ex}

\begin{RQ}\label{RQ_TrivialSquareZeroExtension}
	Let us explain how trivial square zero extensions are functorially related to derivations.
	
	\begin{itemize}
		\item Given $B \in \mathbf{Alg}_{\mathbf{Com}}^{A/}$, the trivial square zero extension construction defines a functor: 
		\[ B \oplus - : \begin{tikzcd}[row sep = tiny, column sep = small]
			\Mod_B \arrow[r] & \mathbf{Alg}_{\mathbf{Com}}^{A//B} \\
			M  \arrow[r, mapsto] & B \oplus M
		\end{tikzcd} \]
		
		\item $A$-linear sections of the projection $B \oplus M \rightarrow B$ are precisely $A$-linear $M$-valued derivations\footnote{$\underline{\tx{Der}}_A(B,M)$ denotes the $B$-module of $A$-linear $M$-valued derivations concentrated in degree $0$ and $\Rr \tx{Der}_A(B,M)$ its derived functor.} on $B$:
		\[ \Hom_{\mathbf{Alg}_{\mathbf{Com}}^{A//B}} \left( B, B\oplus M \right) \cong \tx{Der}_A (B,M)\]
		
		\item  The square zero extension functor $B \oplus -$ has a left adjoint: 
		\[ L : \mathbf{Alg}_{\mathbf{Com}}^{A//B} \rightarrow \Mod_B\]
		
		Moreover this adjunction is Quillen for the standard model structures on $\mathbf{Alg}_{\mathbf{Com}}^{A//B}$ and $\Mod_B$. 
	\end{itemize}
\end{RQ}

\begin{Def}\label{Def_CotangentComplex}
	Given $B \in \mathbf{Alg}_{\mathbf{Com}}^{A/}$, the \defi{module of Kähler differentials} is defined as the $B$-modules $\Omega_{B/A}^1$ representing $A$-linear derivations: 
	\[ \Hom_{B} \left(\Omega_{B/A}^1, M \right) \cong \tx{Der}_A (B, M) \]
	
	Remark \ref{RQ_TrivialSquareZeroExtension} ensures that: 
	\[\Omega_{B/A}^1 = L (\tx{id} : B \rightarrow B)\]

{A direct construction of $\Omega_{B/A}^1$ is obtained by constructing the free $B$-module on formal generators $d_\mathrm{dR}b$ for all $b\in B$, subject to the relations $d_\mathrm{dR} a=0$ for $a\in A$, $d_{\Omega^1}d_\mathrm{dR}x = d_\mathrm{dR}d_B x$ and $d_\mathrm{dR}(xy) =( d_\mathrm{dR}x) y +x d_\mathrm{dR}y$.}

	The derived version is called the \defi{cotangent complex} $\Ll_{B/A}$ representing the derived module of derivations: 
	\[ \iHom_{B} \left(\Ll_{B/A}, M \right) \simeq \mathbb R\mathrm{Der}_A (B, M) := \Rr  \iHom_{\mathbf{Alg}_{\mathbf{Com}}^{A//B}} \left( B, B\oplus M \right). \]
	
	The $B$-linear dual of $\Ll_{B/A}$ is called the \defi{tangent complex} and we have:
	\[ \Tt_{B/A} := \iHom_{B} \left( \Ll_{B/A}, B \right) \simeq \Rr \tx{Der}_A(B,B)\] 
\end{Def}

\begin{RQ}\label{RQ_HigherCotangentComplex} Here are a few basic properties of the cotangent complex:
	
	\begin{itemize}
		\item 	Using the derived functor of the adjunction in Remark \ref{RQ_TrivialSquareZeroExtension}, we obtain: 
		\[ \Ll_{B/A} :=  \Ll (\tx{Id}: B \rightarrow B) := \Omega_{Q(B)/A}^1 \otimes_{Q(B)} B,  \]
		
		where $Q$ is a cofibrant replacement functor in $\cdga^{A//B}$ and $\Ll$ denotes the left derived functor of $L$. 
		
		\item Given $f: A \rightarrow B$ and $g: B \rightarrow C$, we have a homotopy fiber sequence: 
		\[ \begin{tikzcd}
			\Ll_{B/A} \otimes_B C \arrow[r] & \Ll_{C/A} \arrow[r] &  \Ll_{C/B}
		\end{tikzcd}\]
		
		In particular for $A = k$ we write $ \Ll_{B/k} =  \Ll_{B} $ and we get: 
		\[ \begin{tikzcd}
			\Ll_{B} \otimes_B C \arrow[r] & \Ll_{C} \arrow[r] &  \Ll_{C/B}
		\end{tikzcd}\]

		\item The isomorphism classes of square zero extensions {of $B$ along $I$} are in one to one correspondence with $\mathbf{Ext}_A^1 \left( \Ll_B, I \right)$.	In particular, for $I =  B$ and $A =B$, we get $\mathbf{Ext}_B^1 \left( \Ll_B, B \right) = \mathbf{Ext}_B^1 (B, \Tt_B)$. Moreover, the trivial square zero extension of $B$ by $M$ described in Remark \ref{RQ_TrivialSquareZeroExtension} represents the class of $0 \in \mathbf{Ext}^1(\Ll_B, M)$. 
		
		\item  {Following \cite{Ve10}, t}he cotangent complex enables us to consider deformations in higher derived directions. Such higher deformations are encoded by higher $\mathbf{Ext}$ groups, and those are controlled by derived mappings from higher infinitesimal disks into the $k$-scheme $X$ we want to deform.
		
		More precisely, for any $R$, a commutative $k$-algebra, a $R$-point in $X$, $x: \tx{Spec}(R) \rightarrow X$, and any $M \in \Mod_R$. Then we define the derived $i$-th order infinitesimal disk over $R$ by:
		\[ \mathbf{D}_{R,i}(M) := \tx{Spec} \left( R \oplus M[i] \right) \]
		More generally, for $\_M$ a quasi-coherent sheaf on $X$, we define:
		\[ \mathbf{D}_{X,i}(\_M) := \tx{Spec} _X\left( \_O_X \oplus \_M[i] \right) \]
		
		Then Proposition 2.1 in \cite{Ve10} tells us that this is related to the $\mathbf{Ext}$\footnote{$\mathbf{Ext}_A(M,N)$ is the derived functor of $\underline{\Hom}_{A}(-,-)$ evaluated at $M$, $N$.} functor via the equivalences: 
		\[ \mathbf{Ext}_{R}^i \left( \Ll_{X,x}, M \right) \simeq \Rr \iHom_\star ( \mathbf{D}_{R,i}(M), (X,x))\]
		\[ \mathbf{Ext}_{\_O_X}^i \left( \Ll_X, \_M \right) \simeq \Rr \iHom_X ( \mathbf{D}_{X,i}(\_M), X)\]
		
		Above, the first Hom is required to be pointed and the second Hom is required to restrict to the identity on $X$. 
		Extensions of the Kodaira--Spencer morphism and obstruction are also discussed in \cite[Section 3]{Ve10}.
	\end{itemize}
\end{RQ}

We can now show how \emph{first-order} derivations of smooth schemes are related to the tangent cohomology.

\begin{Prop}
	Let $X$ be a smooth algebraic variety, then first-order deformations up to equivalence are in bijection with $H^1(X,T_X)$. 
\end{Prop}

\begin{proof}
	
	First, notice that a first-order deformation of a smooth algebra $\tx{Spec}(B)$ is trivial as explained in Remark \ref{RQ_DeformationAlgebraicGeometry}, and given by $B' := B \otimes {k[\epsilon]}$, with $\epsilon^2 = 0$. It is not difficult to see that $B'$ coincides with the trivial square zero extension of $B$ by $B$ itself. In other words $B' = B \oplus B$. To simplify the notations, we will write $B' = B[\epsilon]$ and the product is given by $(b + b'\epsilon)(a + a' \epsilon) = ba + (ba'+ b'a)\epsilon$ as expected when having $\epsilon^2 = 0$. \\
	
	Now a first order deformation $\widetilde{X}$ can be covered by smooth affine opens. Take $\lbrace\tx{Spec}(B_i')\rbrace$ such a smooth open chart. Then pulling back this covering along $X \rightarrow \widetilde{X}$ gives a covering of $X$ . Then, we have the following pullback:  
	\[ \begin{tikzcd}
		\tx{Spec}(B_i) \arrow[r] \arrow[d] & \tx{Spec}(B_i') \arrow[d] \\
		X \arrow[r] & \widetilde{X}
	\end{tikzcd}\] 
	
	Since $\tx{Spec}(B_i)$ is an open of $X$ which is smooth, then $\tx{Spec}(B)$ is also smooth, and by rigidity (see Remark \ref{RQ_DeformationAlgebraicGeometry}) we have $B_i' = B_i[\epsilon]$. $\widetilde{X}$ is therefore only characterized by the gluing data of the $B_i[\epsilon]$, lifting the gluing data of $X$. In other other, if $B_i$ and $B_j$ intersect, we need a lifting of the isomorphism $\phi_{ij} : B_{i, j} \rightarrow B_{j, i}$\footnote{$B_{ij}$ denotes the restriction of $B_i$ to the intersection of $B_i$ with $B_j$. In other words, $B_{i,j}$ is a localisation of $B_i$ by some ideal realising this intersection.} characterizing the gluing on $X$: 
	\[ \begin{tikzcd}
		B_{i, j} \arrow[r, "\phi_{ij}"] \arrow[d] \arrow[dd, bend right = 40, "\tx{id}"'] & B_{j, i} \arrow[d]\arrow[dd, bend left = 40, "\tx{id}"]\\
		B_{i, j} [\epsilon] \arrow[r, dashed, "\phi_{ij}^\epsilon"] \arrow[d] & B_{j, i}[\epsilon] \arrow[d] \\
		B_{i, j} \arrow[r, "\phi_{ij}"]  & B_{j, i} 
	\end{tikzcd}\]     	
	
	Since $\phi$ is an isomorphism, giving a filling is equivalent to taking a section of $B_{j, i}[\epsilon] \rightarrow B_{j,i}$. We can also show that such a lift is necessarily an isomorphism.\\
	
	Since such sections are equivalent to derivations of $ B_{j,i}$, a deformation is given by a family of derivations on pairwise intersections of $X$. \\
	
	Moreover the maps, $\phi_{ij}$, satisfy a cocycle condition on triple intersections that lift to a cocycle condition on the family of derivations. Such a family is then exactly an element in the first cohomology class of $T_X$.
\end{proof}

\subsubsection{Deformation theory of algebras}\label{Sec_Deformation from the Algebraic Point of View} \
\medskip

We want to describe deformations from the algebraic point of view. To do so, we start by describing deformations of commutative algebras similarly to the affine geometric situation of Remark \ref{RQ_DeformationAlgebraicGeometry}. This definition extends naturally from commutative algebras to general $\_P$-algebras and we will compare this notions of deformation that mimics geometric deformations, with deformations of $\_P$-algebraic structures as defined in Section \ref{Sec_MaurerCartanSpace} and \ref{Sec_DeformationComplex}.
 From now on we will assume that $\_P$ is Koszul, without differential. \\

 We consider the canonical model structure on groupoids in which weak equivalences are equivalences of categories and fibrations are isofibrations\footnote{An isofibration is a functor $p:\_C \to \_D$ such that for any object $e$ of $\_C$ and any isomorphism $f:p(e)\to b$, there exists an isomorphism $\tilde{f}: e \to e'$ such that $p(\tilde{f})=f$.}. This model structure is the natural model structure restricting the model structure on the category of small categories, or equivalently the transferred model structure from $\mathbf{sSet}$ via the nerve adjunction (see \cite[Theorem 2.1 and Corollary 2.3]{Ho08}).

\begin{Def} \label{Def_GroupoidDeformationPAlgebra}
	The \defi{groupoid of geometric deformations} of a $\_P$-algebra $B$ along a local Artinian  algebra $A$ (Proposition \ref{Prop_dgArtinian}) is defined as the homotopy pullback of groupoids\footnote{The superscript ${}^\mathrm{iso}$ denotes the subcategory given by the maximal subgroupoid.}: 
	\[\begin{tikzcd}
		\underline{\t_{\tx{Def}}}_B^\_P(A) \arrow[r] \arrow[d] & \star \arrow[d, "B"]\\
		\mathbf{Alg}_{\_P\otimes_k A}^{\tx{iso}}(\Mod_A) \arrow[r] & \mathbf{Alg}_{\_P}^{\tx{iso}}
	\end{tikzcd}\]
	
	where the lower functor sends $B'$ to $B' \otimes_A k$. For groupoids all objects are fibrant so we only need to have one fibration. To do that, we need to replace the map $B$ by an equivalent fibration given by the following factorization:
	\[ \begin{tikzcd}
		\star \arrow[r, "\tx{triv. cof}"]& \left(\mathbf{Alg}_{\_P}^{\tx{iso}} \right)^{/B} \arrow[r, "\tx{fib.}"] & \mathbf{Alg}_{\_P}^{\tx{iso}}
	\end{tikzcd} \]
	Computing the strict pullback with this replacement, we obtain that $\underline{\t_{\tx{Def}}}_B^\_P(A)$ is equivalent to the groupoid whose objects are elements $B' \in \mathbf{Alg}_{\_P}(\Mod_A)$ together with an isomorphism $B' \otimes_A k \overset{\sim}{\to} B$. This corresponds to the affine version of the set of geometric $A$-deformations as described in Remark \ref{RQ_DeformationAlgebraicGeometry}. 
	
	The morphisms in $\underline{\t_{\tx{Def}}}_B^\_P(A)$ correspond to the notion of equivalence described in Remark \ref{RQ_DeformationAlgebraicGeometry}. In other words, two deformations $B_1'$ and $B_2'$ are \emph{equivalent} if there exists an isomorphism $\phi : B_1' \rightarrow B_2'$ in $	\mathbf{Alg}_{\_P\otimes_k A}(\Mod_A) $ such that the following diagram commutes: 
	\[ \begin{tikzcd}
		B_1' \otimes_A k  \arrow[dr]  \arrow[rr, "\phi\otimes_A k"]& & \arrow[dl] B_2' \otimes_A k \\
		& B & 
	\end{tikzcd}\]
\end{Def}


 In that situation, the $\_P$-algebra structure on $B$ is given by a map $\_P \to \mathbf{End}_B$. Recall that the set of deformations of a $\_P$-algebra structure $\phi : \_P \to \mathbf{End}_B$ along $A$, given by Definition \ref{Def_DeformationGroupoidAtAMaurerCartanElement}, is the set: 
\[\tx{Def}_\phi(A) := \left\lbrace \Phi \in \tx{MC}\left( \G_g_{\_P_\infty, B}^\phi \otimes A \right), \ \Phi \cong \phi \ \tx{mod} \ \G_M \right\rbrace \]

It turns out that these \emph{algebraic} deformations coincide with the more \emph{geometric} deformations (of affine):

\begin{Prop} \label{Prop_EquivalenceAlgebraicAndGeometricDeformationOperadic} 
	
	Given $B$ a $\_P$-algebra with $\phi : \_P \rightarrow \mathbf{End}_B$ and $A$ a classical local Artinian $k$-algebra, there is a map of groupoids: 	
	\[\underline{\tx{Def}}_\phi(A) \hookrightarrow \underline{\t_{\tx{Def}}}_B^{\_P_\infty}(A) \] 
	
\end{Prop}

\begin{proof}
An element in $\underline{\tx{Def}}_\phi(A)$ is a $\_P_\infty$-structure on $B \otimes_k A$ such that modulo $\G_M_A$, it recovers the $\_P$-structure on $B$. In other words, the natural map: 
\[ B \otimes_k A \otimes_A k \to B\]
is an isomorphism on $\_P$-algebras. Therefore there is an inclusion of the set of objects. 

Furthermore, the definition of morphisms in $\underline{\tx{Def}}_\phi(A)$ (Definition \ref{Def_DeformationGroupoidAtAMaurerCartanElement}) induces a morphism in $\underline{\t_{\tx{Def}}}_B^{\_P_\infty}(A)$ defining the faithfull functor we seek.
%
\end{proof}

We will now extend the previous construction to the full $\infty$-groupoids of deformations. This $\infty$-groupoid will both extend the previous ones to higher homotopies of deformations and will take into account the homotopy theory of $\_P$-algebras.  

\begin{Def}[{\cite[Definition 2.1]{Hi04}}] \label{Def_HigerDeformationGroupoidGeometric}
	
	Let $A$ be an Artinian commutative algebra concentrated in non-positive degree (Proposition \ref{Prop_dgArtinian}) and $\_P$ a Koszul operad. We define $\_W^c (\_P, A)$ to be the simplicial category whose objects are cofibrant $\_P \otimes_k A$-algebras in $\Mod_A$ and whose $n$-morphisms from $B$ to $C$ are given by quasi-isomorphisms $B \rightarrow \Omega (\Delta^n) \otimes C$. Then the \defi{higher deformation groupoid} $\mathbf{Def}_A (B)$ of $B \in \mathbf{Alg}_\_P$ is defined as the homotopy pullback: 
	
	\[ \begin{tikzcd}
		\mathbf{Def}_A (B) \arrow[r] \arrow[d] & \star \arrow[d, "\tilde{B}"] \\
		N(\_W^c (\_P,A)) \arrow[r] & N(\_W^c (\_P,k))
	\end{tikzcd}\]
	
	Where $\tilde{B}\rightarrow B$ is a cofibrant resolution of $B$ and $N$ denotes the simplicial nerve functor. 
\end{Def}

\begin{RQ}
	In \cite{Hi04}, it is shown thanks to Theorem 2.1.2, that for $\_P$ and $B$ non-positively graded, the deformation $\infty$-groupoid $\mathbf{Def}_A (B)$ is naturally equivalent to the higher Deligne groupoid $\mathbf{Del}\left(\Tt_B^\_P \otimes_k A\right)$ of the $\_P$-operadic tangent Lie algebra (see Definition \ref{Def_Operadic(Co)TangentComplex}). 
\end{RQ}

	The operadic tangent and cotangent complexes are defined in a very similar way to what was explained in Section \ref{Sec_CotangentComplexHigherDeformation} for $\_P = \mathbf{Com}$. 
	
\begin{Def}[Operadic Tangent Complex] \label{Def_Operadic(Co)TangentComplex}\
Let $A$ be a $\_P$-algebra. A \defi{derivation }of $A$ valued in $A$  is a map $D\colon A \to A$ such that for all $p\in \_P(n)$, 
\[
	D(p(b_1,\dots,b_n))= \sum_{i=1}^n p(b_1,\dots,Db_i,\dots,b_n).
\]
The $k$-module of all derivations $\mathrm{Der}(A,A)$ is a Lie algebra with bracket given by the commutator of derivations.

The derived derivations of $A$, which we denote $\Tt_A^\_P$, is called the \defi{tangent complex} of $A$. 	
\end{Def}

As the name indicates, the derived derivations can be obtained as a derived functor. Indeed, there is an operadic notion of module of Kähler differentials, $\Omega_A^\_P$ and the \defi{operadic cotangent complex} $\Ll_A^\_P$, see \cite[Sections 12.3.8--12.3.10]{LV}. One could equivalently define $\Tt_A^\_P= \iHom_{A}\left(\Ll_A^\_P, A \right)$.

To complete this story, we would like to have a relationship between $\Tt_A^{\_P}$ and the deformation complex, to have  a relationship between the higher deformation groupoid of Definition \ref{Def_HigerDeformationGroupoidGeometric} and the higher Deligne groupoid (Definition \ref{Def_HigherDeligneGroupoid}).

To obtain a model of $\Tt_A^{\_P}$ (which is only defined up to quasi-isomorphism) we choose the cofibrant resolution of $A$ given by the bar-cobar resolution. We have then

\[
\Tt_A^{\_P}
\simeq \tx{Der}_k(\Omega_\kappa \mathbf B_\kappa A,\Omega_\kappa \mathbf B_\kappa A) = \Hom_k(\mathbf B_\kappa A,\Omega_\kappa \mathbf B_\kappa A) \stackrel{p}\to \Hom_{k} \left( \mathbf B_\kappa A, A \right).\]
The map $p$ is induced by the bar-cobar resolution and is therefore a quasi-isomorphism of complexes. In fact, $\Hom_{k} \left( \mathbf B_\kappa A, A \right)$ can be identified with coderivations on $\mathbf B_\kappa A$ and therefore has a Lie algebra structure making $p$ a quasi-isomorphism of Lie algebras. 
Notice that ignoring differentials $\mathbf B_\kappa A = \_P^\antishriek \circ A$ and therefore as graded vector spaces we have

\[\iHom_{k} \left( \mathbf B_\kappa  A, A \right) \cong \prod_{n\geq 1} \iHom_{k} \left( \_P^\antishriek(n), \End_A(n) \right) \]

We recognize here the underlying graded vector space of the deformation complex, up to a caveat: In the deformation complex one should remove the counit by taking $\overline{\_P^\antishriek}$ instead, see Remark \ref{rem:A remark on augmentations}.

 Keeping track of the differential and the bracket, we see that they correspond precisely to the differential and bracket of the deformation complex. Taking this model for $ \Tt_A^{\_P}$, the projection $\_P^\antishriek \to \overline{\_P^\antishriek}$ yields a well defined inclusion of Lie algebras:

$$\G_g_{\_P_\infty, A}^\phi\to \Tt_A^\_P$$

		This map is the Lie algebra version of the map of groupoids of Proposition \ref{Prop_EquivalenceAlgebraicAndGeometricDeformationOperadic}.
		In short, we have shown the following Proposition:
		
		\begin{Prop}\label{Prop_TangentDGLAAndDeformationComplex} 
	Let $\phi : \_P \rightarrow \mathbf{End}_A$ be a $\_P$-algebra structure on $A\in \chk$. There is a homotopy fiber sequence of Lie algebras:
			
			\[ \G_g_{\_P_\infty, A}^\phi\to \Tt_A^\_P \to \Hom(A,A)\]
		\end{Prop}
	
\begin{RQ}\label{RQ:heuristic associative FMP}
Heuristically,  Proposition \ref{Prop_TangentDGLAAndDeformationComplex} expresses that the ``difference'' between the two deformation problems is the deformation of the underlying $k$-module, which is encoded by $\Hom(A,A)$.
		 	 
Notice that the Lie algebra structure on $\Hom(A,A)$ in fact arises from an associative algebra structure. From the deformation perspective, this means that the corresponding deformation problem is also richer. This will correspond to an \emph{associative} (as opposed to commutative, not Lie) formal moduli problem and its formalisation will be the topic of the next sections.
	\end{RQ}

\begin{RQ}[Global Deformations]
	
	So far, we only talked about deformations of ``affine objects'', that is $\_P$-algebras. However, we could be interested in deforming not only $\_P$-algebras, but a sheaf of $\_P$-algebras in the spirit of the first part of Remark \ref{RQ_DeformationAlgebraicGeometry}. 
	
	In \cite{Hi05}, it is shown that the deformations of a sheaf of $\_P$-algebras $\_A$ are controlled, under some conditions on $\_P$ and $\_A$ given in \cite[Assumption 3.5.1]{Hi05}, by the higher Deligne groupoid of the derived global sections of the presheaf of local tangent Lie algebras, which by definition gives the tangent cohomology. 
\end{RQ}

\subsection{Formal Moduli Problems}\
\label{Sec_Formal Moduli Problem}

\medskip

This section is about axiomatizing the notion of deformation theory. After seeing how deformations from an algebraic and from a geometric point of view are related, we expect them to have very similar properties which we will encode in the notion of formal moduli problems.

The idea is that deformations can be understood as a functor sending an ``Artinian'' object (which we think think of as a parameter space) to a space of deformations along this object. For example such a deformation functor assigns for each choice of Artinian local algebra $A$ a set, groupoid or $\infty$-groupoid of deformations of $X$ along $A$. As the category of local Artinian algebras and the notion of ``infinitesimal'' object will play a more important role from here, we will start, following \cite{Lu11}, by defining a general framework in which we can speak of deformations and Artinian objects.

In this section, $\_P$ will be an augmented Koszul operad. 

\subsubsection{Artinian and small algebras in a deformation context}\
\label{Sec_Artin and Small Algebras in a Deformation Context}

\medskip

The general idea behind Artinian algebras is to have a class of algebras considered ``small'' so that deforming along those algebras amounts to consider \emph{infinitesimal deformations}, looking at the \emph{formal neighborhood} of what we want to deform. We will always assume that our ambient category $\_A$, in which we will define the notion of Artinian object, has a terminal object. This leads to the general framework of deformation context: 

\begin{Def}[{\cite[Definition 1.1.3]{Lu11}}] \label{Def_DeformationContext}
	A \defi{deformation context} is a pair $\left( \_A, \lbrace E_\alpha \rbrace_\alpha \right)$  where $\_A$ is a presentable\footnote{{$\_A$ is presentable if it is generated by a small set of ``small objects'' under homotopy colimits.}} $\infty$-category and $\lbrace E_\alpha \rbrace_\alpha$ is a set of objects in $\mathbf{Stab}(\_A)$, the stabilization\footnote{Recall that objects in $\mathbf{Stab}(\_A)$ correspond to spectrum objects in $\_A_*$, i.e., a sequence  $E=(\cdots, a_{2}, a_{1},a_0, a_{-1}, \cdots)$  of pointed objects of $\_A$, such that $a_i$ is equivalent to the loop space $\Omega a_{i+1}$. For $n\geq 0$, we denote $\Omega^{\infty-n}E = a_{n}$.} of $\_A$.  
	
\end{Def}

	  Intuitively, $\lbrace E_\alpha \rbrace_\alpha$ corresponds to the first-order objects we want to deform along. 	
To compare with the classical picture, we take $\_A$ to be the category of commutative $k$-algebras. Since previously we took $k[\epsilon]$ as the algebra representing first-order deformations, we now take it as the underlying object of the spectrum object $E= (\cdots,k[\epsilon_2],k[\epsilon_1],k[\epsilon_0],\cdots)$, where $k[\epsilon_i]$ is the $2$-dimensional augmented commutative algebra with $\deg \epsilon_i=-i$ and $\epsilon_i^2 = 0$. {We will see in Section \ref{Sec_tangent-complex-of-a-formal-moduli-problem} that deformations along $k[\epsilon_i]$ are both related to the higher tangent complex and control obstructions to lifting deformations.} 

In all cases we consider, the family $\lbrace E_\alpha \rbrace_\alpha$ contains only a single object. Here are some other examples of deformation contexts. 

\begin{RQ}\label{RQ: Stabilisation for P algebras}\
	
	\begin{itemize}
		\item Given a commutative algebra $A$, the stabilization of $A$-augmented commutative $A$-algebra, $\mathbf{Stab}\left(\mathbf{Alg}_{\mathbf{Com}}^{A//A}\right)$ can be identified with $\Mod_A$ using the square zero extension functor, see \cite[Theorem 3.7]{Ba04}
		\[ \begin{tikzcd}[row sep = tiny]
			\Mod_A \arrow[r, "\sim"] & \mathbf{Stab}\left(\mathbf{Alg}_{\mathbf{Com}}^{A//A}\right)\\
			M \arrow[r, mapsto] & \left( A \oplus M[n] \right)_{n \in \Zz}
		\end{tikzcd}\]  
		\item To understand $\mathbf{Stab}\left(\mathbf{Alg}_{\mathbf{Com}}\right)$ first observe that there is an equivalence: 
		\begin{equation}\label{eq: equivalence unital non-unital commutative algebras}
		\begin{tikzcd}
		\mathbf{Alg}_{\mathbf{Com}} \arrow[r, shift left] & \arrow[l, shift left] \mathbf{Alg}_{\mathbf{Com}}^{k//k}  
		\end{tikzcd}
		\end{equation}
		
		sending a non-unital commutative algebra $B$ to its unitalization ${k \oplus B}$ and an augmented $k$-algebra $A$ to the fiber of the augmentation $A \to k$. This equivalence induces an equivalence: 
		
		\[ \begin{tikzcd}
		\mathbf{Stab}\left(\mathbf{Alg}_{\mathbf{Com}} \right) \arrow[r, shift left] & \arrow[l, shift left] \mathbf{Stab}\left( \mathbf{Alg}_{\mathbf{Com}}^{k//k}  \right) 
		\end{tikzcd} \]
		
		As such $\mathbf{Stab}\left(\mathbf{Alg}_{\mathbf{Com}} \right)$ is also equivalent to $\Mod_k$ and under this equivalence the spectrum object corresponding to $M \in \Mod_k$ is $(M[n])_{n \in \Zz}$ with the trivial commutative product on each $M[n]$.
		\item Similarly $\mathbf{Stab}\left(\mathbf{Alg}_{\_P} \right)$ will be given by the category of $\_P$-modules. A $\_P$-module $M$ corresponds the spectrum object $(M[n])_{n \in \Zz}$ where $M[n]$ is viewed with the trivial $\_P$-algebra structure\footnote{Since we assume that $\_P$ is augmented, Remark \ref{RQ_TrivalPAlgebra} gives us a functor $\Mod_k \to \mathbf{Alg}_\_P$ giving the trivial $\_P$-algebra structure to a module.}.
		
	\end{itemize}

\end{RQ}

\begin{Ex}\label{Ex_DeformationContexts}\
	
	\begin{itemize}
		
		\item For $\_A = \Mod_A$ (with $A \in \mathbf{Alg}_{\mathbf{Com}}$), we can define the deformation context $\left(  \Mod_A, \lbrace E \rbrace \right)$ with only one spectrum object $E = \left( A[n] \right)_{n \in \Zz}$ (\cite[Example 2.2]{CG18}). 
		
		\item  Given $A$ a commutative algebra, we consider $\_A = \mathbf{Alg}_{\mathbf{Com}}^{A//A}$  and we can define the following deformation context: 
		\[\left( \mathbf{Alg}_{\mathbf{Com}}^{A//A}, \lbrace E = \left( A \oplus A[n] \right)_{n \in \Zz} \rbrace \right)\]
		
		where $A \oplus A[n]$ is the square zero extension of $A$ by $A[n]$ (\cite[Remark 2.5]{CG18}). When $A=k$, we will call this deformation context $\dccomaug$ for short.

		\item For $\_A = \mathbf{Alg}_{\mathbf{Com}}$ we consider the following deformation context: 
		\[ \left(\mathbf{Alg}_{\mathbf{Com}}, \lbrace \left( k[n] \right)_{n \in \Zz} \rbrace \right) \]
		
		where $k[n]$ is the trivial commutative algebra on $k[n]$. This deformation context can also be seen as obtained from a transfer of deformation contexts along the adjunction of Equation \eqref{eq: equivalence unital non-unital commutative algebras} (see \cite[Lemma 2.6]{CG18}). 
		
		\item For $\_A = \mathbf{Alg}_\_P$ we consider the deformation context:
		\[\left( \mathbf{Alg}_\_P, (k[n])_{n\in \Zz} \right)\]
		We will denote it by $\dcP$ for short.
	
	\end{itemize}
\end{Ex}

From a deformation context, we can define the notion of ``small'' objects and morphisms using $\Omega^{\infty-n}E_\alpha$ as building blocks.

\begin{Def}[{\cite[Definition 1.1.14]{Lu11}}] \label{Def_SmallObjectsAndMorphisms}
	Given a deformation context $\left( \_A, \lbrace E_\alpha \rbrace_{\alpha \in T} \right)$ we say that:
	\begin{itemize}
		\item A morphism $f : A \rightarrow A'$ is \defi{elementary} if it is given by a pullback of the form: 
		\[ \begin{tikzcd}
			A \arrow[r] \arrow[d, "f"] & \star \arrow[d] \\
			A' \arrow[r] & \Omega^{\infty -n} ( E_\alpha)
		\end{tikzcd}\]
		
		for some $\alpha \in T$ and $n \geq 1$. 
		\item A morphism $f: A \rightarrow A'$ is \defi{small} if it is a finite composition of elementary morphisms.
		\item An object $A \in \_A$ is \defi{Artinian} if the morphism $\_A \rightarrow \star$ is small\footnote{Artinian objects can be though as ``small'' objects. In fact they are called ``small'' in \cite{Lu11}.}. We denote by $\mathbf{Art}_\_A$ the full sub-category of $\_A$ given by small objects. 
		
	\end{itemize}
	
\end{Def}

\begin{Ex}\label{ex:what are artinian algebras}
	Going back to the example of (augmented) commutative algebras, we  have $\Omega^{\infty -n} E = k[\epsilon_n]$ so that in order to compute the homotopy pullback of $\star \to \Omega^{\infty -n} E$, one can replace the point $\star = k$ by the algebra $k\oplus k\epsilon_n \oplus k\epsilon_{n-1}$ (notice that $n\geq 1$), with the only non-zero product being the unit and with differential sending $\epsilon_{n}$ to $\epsilon_{n-1}$. 
	
	The strict pullback exhibits therefore $A$ as a square zero extension of $A'$ along $k\epsilon_{n-1}$. In particular, an Artinian commutative algebra concentrated in degree zero is a classical Artinian local algebra in the sense of Definition \ref{Def_SmallComAlgebra}, thanks to Example \ref{Ex:sq zero}.
	
Conversely, suppose $A$ is Artinian. Then for any square zero extension by the algebra $k[\epsilon_n]$:
	\[ \begin{tikzcd}
		k[\epsilon_{n}]  \arrow[r] & B \arrow[r] & A ,
	\end{tikzcd}\]

	$B$ is also Artinian.
	Indeed up to {pulling back to} a quasi-free resolution of $A$, the surjection $B \rightarrow A$ splits as graded algebras and $B \simeq A \oplus k\epsilon_{n}$ with differential $d(a,v) = (da,\chi(a))$. We can then check that $B$ can be realised as the homotopy pullback square:
	\[ \begin{tikzcd}
		B \arrow[r] \arrow[d] & k \arrow[d] \\
		A \arrow[r, "\chi"] & k[\epsilon_{n+1}] 
	\end{tikzcd}\] 
	
	where $\chi$ is obtained as a (shift of) the part of the differential on $A \oplus k[n]$ going from $A$ to $k[n]$. In particular, we have $B \simeq  A \times_{k[n+1]}^h 0$ and therefore $B$ is Artinian. 

\end{Ex}

We will now specialize these definitions to the deformation contexts we are interested in, namely $\dccomaug$ and $\dcP$. 

\begin{Prop}[Artinian Commutative Augmented $k$-Algebras]

	The category $\small$ is the smallest full sub-category of $\mathbf{Alg}_{\mathbf{Com}}^{k//k}$  such that:  
	\begin{itemize}
		\item $ k \oplus k[n] \in \small$ for all $n\geq 0$
		\item For any $A \in \small$ and any map $A \rightarrow k \oplus k[n]$ with $n\geq 1$, the homotopy pullback $A \times_{k \oplus k[n]} k$ is also Artinian:
		
		\[\begin{tikzcd}
			A \times_{k\oplus k[n]} k \arrow[r] \arrow[d] & k \arrow[d] \\
			A \arrow[r] &  k \oplus k[n]
		\end{tikzcd}\] 
	\end{itemize}   

This is a variation of Proposition \ref{Prop_SmallCategoryDescription}. 	
\end{Prop}

We can also describe Artinian algebras in more classical terms:

\begin{Prop}[{\cite[Proposition 1.1.11]{Lu11}}]\label{Prop_dgArtinian}
	$A \in \mathbf{Alg}_{\mathbf{Com}}^{k//k}$ is Artinian if and only if:
	\begin{itemize}
		\item $H^i(A) = 0$ for all $i >0$.
		\item $H^*(A)$ is of total finite dimension.
		\item Artinian condition: 		
		\begin{center}
			$\ker \left( H^0(A) \rightarrow k \right)$ is nilpotent. 
		\end{center} 
	\end{itemize}
	
In particular $H^0(A) \in \art$ and there is an embedding of strict categories $\art \rightarrow \small$. Indeed, as we have seen in Example \ref{Ex:sq zero}, classical local Artinian algebras are created by successive square-zero extensions. 
\end{Prop}

 With Definition \ref{Def_SmallObjectsAndMorphisms} and the deformation context on $\mathbf{Alg}_\_P$ of Example \ref{Ex_DeformationContexts} we have now a suitable notion of Artinian algebra over an operad.

\begin{Prop}  \label{Prop_SmallCategoryDescription}
	The $\infty$-category $\smallp$ of Artinian $\_P$-algebras is the smallest full sub-category of the $\infty$-category $\mathbf{Alg}_\_P$ such that:  
	\begin{itemize}
		\item $k[n] \in \smallp$ for all $n\geq 0$
		\item For any $A \in \smallp$ and any map in $\mathbf{Alg}_\_P$, $A \rightarrow k[n]$ with $n\geq 1$, the following homotopy pullback $A \times_{k[n]} 0$ is also Artinian:
		
		\[\begin{tikzcd}
			A \times_{k[n]} 0 \arrow[r] \arrow[d] & 0 \arrow[d] \\
			A \arrow[r] &  k[n]
		\end{tikzcd}\] 
	\end{itemize}  	 
	This is corresponds to \cite[Definition 2.2]{CCN20}.   
\end{Prop}

\begin{proof}
	By definition, we have that $k[n]$ is Artinian since the map $k[n] \rightarrow 0$ is elementary, given by the homotopy pullback:
	\[ \begin{tikzcd}
		k[n] \arrow[r] \arrow[d] & 0 \arrow[d] \\
		0 \arrow[r] & k[n+1]
	\end{tikzcd}\] 
	
	Moreover, the map $A \times_{k[n]} 0 \rightarrow A$  is elementary, therefore if $A$ is Artinian, then the composition $A \times_{k[n]} 0 \rightarrow A \rightarrow 0$ is Artinian as well. Therefore $\smallp$ satisfies the properties given in the proposition. It is the smallest such category because any Artinian object can be obtained from some $k[n]$ in finitely many pullbacks along some $0 \rightarrow k[n]$ for $n\geq 1$. 
\end{proof}

Similarly we can rephrase the Artinian $\_P$-algebra property as a simpler condition.
\begin{Lem}[{\cite[Lemma 2.8]{CCN20}}]\label{Lem_ArtinianConditionPAlgebra}
	Suppose the operad $\_P$ is non-positively graded. A $\_P$-algebra $A \in \mathbf{Alg}_\_P$ is Artinian if: 
	\begin{itemize}
		\item $H^i(A) = 0$ for $i>0$. 
		\item $H^*(A)$ has total finite dimension. 
		\item $H^i(A)$ is nilpotent with respect to the $H^0(\_P)$ algebra $H^0(A)$ i.e. for all $a_1, \cdots, a_n \in H^0(A)$, $P \in H^0(\_P)(n+1)$ the map: 
		\[ \_P(a_1, \cdots, a_n, -) : H^i (A) \rightarrow H^i(A) \]
		
		is such that the $k$-th iteration of this map is $0$ for some $k \in \Nn$.  
	\end{itemize} 
\end{Lem}

\begin{RQ}
	\label{RQ_RemarkOnSmallP}
	If $A \in \smallp$, then for any square zero extension by the trivial $\_P$-algebra $k[n]$:
		\[
			k[n] \to B \to 
			A 
	\]
		
		$B$ is also in $\smallp$. The argument is precisely the same as in Remark \ref{ex:what are artinian algebras}, see also \cite[Example 2.3]{CCN20}.

\end{RQ}

\subsubsection{Formal moduli problems}\label{Sec_FMP}\

\medskip

In Sections \ref{Sec_Artin and Small Algebras in a Deformation Context} and \ref{Sec_DeformationofAlgebraicStructures}, we defined, given a local Artinian algebra $A$, different notions of sets, groupoids or even $\infty$-groupoids of deformations of a object along $A$.  \\

These assignments that send $A$ to such spaces of deformations define functors which will be the prototypical examples of what we will call ``formal moduli problems''. In other words, given an object $X$ (for example a $\_P$-algebra, a complex manifold, etc...), we will be interested in studying the functor that sends an Artinian object $A$ to the \emph{space} of all deformations of $X$ along $A$.
To be precise about what we mean by ``space'', we take $\_S$ to be the $\infty$-category of $\infty$-groupoids.

\begin{Def}[Formal Moduli Problem, {\cite[Definition 1.1.14]{Lu11}}] \label{Def_FMP}\
	
	Given a deformation context $\left( \_A, \lbrace E_\alpha \rbrace_{\alpha \in T} \right)$, a functor of $\infty$-categories $F : \mathbf{Art}_{\_A} \rightarrow \_S$  is called a \defi{formal moduli problem} if it satisfies the following conditions: 
	
	\begin{itemize}
		\item  Deformations along the point (terminal object) are trivial:
		\[ F(\star) \simeq \star \]
		
		\item Any homotopy pullback in $\mathbf{Art}_{\_A}$ along a small morphism $\phi : A \rightarrow B$ is sent to a homotopy pullback in $\_S$: 
		
		\[ F \left( \begin{tikzcd}
			A' \arrow[r] \arrow[d] & A \arrow[d, "\phi"] \\
			B' \arrow[r] & B
		\end{tikzcd} \right) = \begin{tikzcd}
			F(A') \arrow[r] \arrow[d] & F(A) \arrow[d, "F(\phi)"] \\
			F(B') \arrow[r] & F(B)
		\end{tikzcd} \]
	\end{itemize}
	
	The $\infty$-category of such functors will be denoted $\mathbf{FMP}\left( \_A, \lbrace E_\alpha \rbrace_{\alpha \in T} \right)$ or $\mathbf{FMP}_\_A$ for short if the choice of the collection of spectrum objects $E_\alpha$ is clear. 
\end{Def}
\begin{Ex}\label{Ex_CategoryOperadicFMP}
	A formal moduli problem on the deformation context $\dcP$ is a functor $F : \smallp \rightarrow \_S$ such that $F(\star) \simeq \star$ and sends the following homotopy pullbacks of Artinian objects to homotopy pullbacks (for all $n \geq 0$): 
	
	\[ F \left( \begin{tikzcd}
		A \arrow[r] \arrow[d] & 0 \arrow[d] \\
		B \arrow[r] & k[n+1]
	\end{tikzcd} \right) = \begin{tikzcd}
		F(A) \arrow[r] \arrow[d] & \star \arrow[d] \\
		F(B) \arrow[r] & F(k[n+1])
	\end{tikzcd} \] 
	This is the class of examples we will be interested in. We will denote it by $\mathbf{FMP}_\_P$. The category $\mathbf{FMP}_\mathbf{Com}$ will be denoted $\mathbf{FMP}$.
\end{Ex}

\begin{RQ}\label{RQ_ObstructionFMP}
	Given a square zero extension of an Artinian object along $k[n]$, we have seen in Remark \ref{RQ_RemarkOnSmallP} that it is equivalent to a homotopy pullback of the small morphism $0 \rightarrow k[n]$. In particular, we have a homotopy pullback diagram: 
	\[ \begin{tikzcd}
		F(B) \arrow[r] \arrow[d] & \star \arrow[d,"F(0)"] \\
		F(A) \arrow[r] & F(k[n+1]),
	\end{tikzcd}\]
	where (up to taking a cofibrant resolution) the graded vector space underlying $B$ can be taken to be $A \oplus k[n]$.

	Therefore an $A$-deformation, given by a point $\star \to F(A)$, lifts to a $B$-deformation $\star \to F(B)$ if and only if the composition $\star \to F(A) \to F(k[n+1])$ is homotopic to $0$. In other words, the image of the point in $F(k[n+1])$ is the obstruction to a lift and if the obstruction vanishes, the choice of a lift corresponds to the choice of a $1$-simplex of $F(k[n+1])$. 
	
	We will see in Section \ref{Sec_tangent-complex-of-a-formal-moduli-problem} that the collection of $F(k[n+1])$ defines the notion of tangent complex which controls the obstructions to lifting deformations and indeed a deformation lifts if it induces an exact element in the tangent complex.
\end{RQ}

\begin{Ex}\label{Ex_FMP_AffineMappingSpace}
	For any $B \in \mathbf{Alg}_{\mathbf{Com}}^{k//k}$ the following functor is a formal moduli problem:
	\[  \begin{tikzcd}[row sep = tiny, column sep = tiny]
		\tx{Spf} (B) : & \small \arrow[r] & \_S \\
		& A \arrow[r, mapsto] & \Map \left( B,A \right) 
	\end{tikzcd} \]
	
	Moreover we get a functor $\mathbf{Alg}_{\mathbf{Com}} \rightarrow \mathbf{FMP}$ sending $B$ to $\tx{Spf} (B)$.
	
	 $\tx{Spf} (B)$ is called the \defi{formal spectrum} of $B$. More generally, we can define the formal spectrum in a similar way for $B \in \mathbf{Alg}_\_P$ as:  
	\[  \begin{tikzcd}[row sep = tiny, column sep = tiny]
		\tx{Spf} (B) : & \smallp \arrow[r] & \_S \\
		& A \arrow[r, mapsto] & \Map \left( B,A \right) 
	\end{tikzcd} \]
	
\end{Ex}

\begin{Prop}	
	Following \cite[Lemma 1.1.20]{Lu11}, a morphism $f: A \rightarrow B$ between Artinian commutative algebras, $A,B \in \small$ is small if and only if it induces a surjection of commutative rings $H^0(A) \rightarrow H^0(B)$. Therefore the second condition of Definition \ref{Def_FMP} can be rephrased as follows: 
	
	Any homotopy pullback in $\small$ such that $H^0 (A) \rightarrow H^0(B)$ or $H^0 (B') \rightarrow H^0(B)$ is surjective is sent to a pullback in $\_S$: 
	
	\[ F \left( \begin{tikzcd}
		A' \arrow[r] \arrow[d] & A \arrow[d, "\phi"] \\
		B' \arrow[r] & B
	\end{tikzcd} \right) = \begin{tikzcd}
		F(A') \arrow[r] \arrow[d] & F(A) \arrow[d, "F(\phi)"] \\
		F(B') \arrow[r] & F(B)
	\end{tikzcd} \]
	It implies in particular the classical Schlessinger condition (see \cite{Sch68}).
\end{Prop}

As we have mentioned multiple times, it is an old philosophy that deformation problems are classified by Lie algebras. This classical idea is formalized by the following theorem:  

%
%

\LuriePridhamEquivalence

\begin{War}\label{War_FMP equivalence is NOT always given by Maurer-Cartan} 
	For $\G_g$ a Lie algebra, the one should think heuristically of $\mathbf{MC}$ as the functor assigning to an Artinian algebra $A$ the space $\mathbf{Del}(\G_g\otimes A)$, that is, the space of Maurer--Cartan elements. Unfortunately, $\mathbf{Del}(\G_g\otimes -)$ does not generally define a formal moduli problem, essentially for the same reasons the Goldmann--Milson theorem \ref{Th_GoldmannMilson} does not preserve all weak-equivalences, but only the filtered ones. On finite dimensional $A$, the formula $\mathbf{Del}(\G_g\otimes A)$ gives the correct result on objects. The functor $\mathbf{MC}$ is in some sense more general in order to circumvent this problem. 
\end{War}
 
In the next sections we will explore the techniques used to prove this result, as well as explain its generalisation to operadic algebras.

\subsection{Operadic Formal Moduli Problems and Koszul Duality }\
\label{Sec_Operadic Formal Moduli Problems and Koszul Duality}

\medskip

The goal of this section is to make sense of a generalization of Theorem \ref{Th_LuriePridham}. The main insight is that the equivalence between (commutative) formal moduli problems and Lie algebras arises from the Koszul duality between the operads $\mathbf{Com}$ and $\mathbf{Lie}$. With this in mind, this Theorem generalizes to Theorem \ref{Th_KoszulOperadicFMPEquivalence}, stating the existence of an equivalence of $\infty$-categories: 
	\[ \mathbf{Alg}_{\_P} \simeq \mathbf{FMP}_{\_P^!} \]
	for $\_P$ a Koszul binary quadratic operad concentrated in non-positive degrees.
	
	The key observation is that if the generators $E=\_P(2)$ are finite dimensional, there is a map of operads $\mathbf{Lie}\to \_P \otimes \_P^!$, defined by sending the generator $\ell_2$ of $\mathbf{Lie}$ to $\sum_{e\in E} e\otimes e^\vee$, where the sum runs over a basis of $E$, and $e^\vee$ represents a dual basis. 
	In this case, if $\mathfrak g$ is now  a $\_P$-algebra and $A$ is an Artinian $\_P^!$-algebra, $\mathfrak g\otimes A$ is a Lie algebra and, with the same caveats of Warning \ref{War_FMP equivalence is NOT always given by Maurer-Cartan}, the formula $\mathbf{Del}(\G_g\otimes A)$ is the correct one.\\

We will start in Section \ref{Sec_tangent-complex-of-a-formal-moduli-problem} by explaining how the inverse of $\mathbf{MC}$ is related to the tangent complex. Then Sections \ref{Sec_koszul-duality-context} and \ref{Sec_KoszulContextDualityAndKoszulDuality} are devoted to the study of Koszul duality in the context of deformation theory. Then the main result and its functoriality are discussed in Sections \ref{Sec_main-result} and  \ref{sec:naturality-of-the-main-result}.

\subsubsection{Tangent complex of a formal moduli problem}\ 
\label{Sec_tangent-complex-of-a-formal-moduli-problem}

\medskip

The goal of this section is to explain that the algebra controlling a given formal moduli problem is the ``tangent'' of this formal moduli problem. In particular this explains why first-order deformations and obstructions to lifting those deformations are controlled by the tangent complex.

To motivate the following definition, we will start by explaining the example of the formal spectrum (Example \ref{Ex_FMP_AffineMappingSpace}) of an augmented commutative algebra $B$. We have that $\tx{Spf} (B)(A) = \Map \left( B,A \right)$. In the case of first order deformations, i.e. $A =  k[\epsilon]$, with $\epsilon^2 = 0$, this mapping space corresponds to the tangent complex of $\tx{Spec} (B)$ at the point\footnote{{Given a map $B \to A$ with $A$ local Artinian, the composition $B \to A \to k$ defines a unique point which is deformed within $\tx{Spec}(B)$. The tangent complex at a point $x$ is defined as the $k$-linear derivations of $B$ at $x$, with $\tx{Der}_k^x(B,M) := \Hom_{\mathbf{Alg}_{\mathbf{Com}}^{k//k}}\left(B, k \oplus M\right)$ with $B$  viewed over $k$ via $x$} (see \cite[Definition 1.4.1.2]{TV08}).} corresponding to the augmentation of $B$,
 $x: \star \rightarrow \tx{Spec} (B)$. 
 This is a consequence of the following equivalences: 
 \[ \begin{split}
 \Map_{\mathbf{Alg}_\mathbf{Com}^{k//k}}(B, k[\epsilon]) \simeq & \Map_{\mathbf{Alg}_\mathbf{Com}^{k//B}}(B, B \oplus x_*k[n]) \\
  \simeq & \Map_{\Mod_B}\left( \Ll_B, x_* k[n] \right) \\
   \simeq & \Map_{\Mod_k}\left( x^*\Ll_B, k[n] \right) \\
    \simeq  & \left|  \Tt_{B,x}[n] \right|
 \end{split}  \]

The goal will be to generalize this by applying the functor not only to $k[\epsilon]$ but to all ``first-order elements'' of a given deformation context. Such ``first-order'' objects correspond to the elements in $\lbrace \Omega^{\infty -n} (E_\alpha) \rbrace_{\alpha \in T, n \in \Nn}$ (for example the square zero extensions $k\oplus k[n]$ for $n \in \Nn$ in the case of augmented commutative algebras). Therefore, we will define the tangent functor of a general formal moduli problem by evaluating it at first-order elements.

\begin{Def}[Tangent Functor, {\cite[Section 1.2]{Lu11}}] \ \label{Def_TangentFunctor}
	Given a deformation context $\left( \_A, \lbrace E_\alpha \rbrace_{\alpha \in T} \right)$, (see Definition \ref{Def_DeformationContext}), we define the \defi{tangent space} of a formal moduli problem $F : \mathbf{Art}_{\_A} \rightarrow \_S$ on this deformation context at $\alpha \in T$ to be $T_\alpha F := F(\Omega^\infty E_\alpha) \in \_S$.
	
	To remember higher deformations, it is not enough to evaluate $F$ at $\Omega^\infty E_\alpha$ but we should also evaluate it at all $\Omega^{\infty -n}E_\alpha$ for $n \geq 0$. We define \defi{the tangent complex} of $F$ at $\alpha$ to be the spectrum object $F(E_\alpha) \in \mathbf{Sp}\coloneqq \mathbf{Stab}(\mathbf{sSet})$, which we denote\footnote{This notation is by analogy to the tangent complex but is a priori only a (collection of) spectrum objects. However, we will see that $\Tt_F$ can be ``represented'' by objects with more structure (see the discussion following Proposition \ref{Prop_TangentSpaceAndTangentComplex} and Proposition \ref{Prop_FMPEquivalenceAndTangent}).} $\Tt_F$ in case there is only a single $\alpha\in T$.

\end{Def}   

 This spectrum object verifies $\Omega^{\infty -n}F(E_\alpha) \simeq F(\Omega^{\infty -n}E_\alpha)$ for all $n\geq 0$ (see \cite[Remark 1.2.7]{Lu11}). Notice that the tangent complex $\Tt_F$ is not actually a chain complex but rather only a spectrum object a priori. These two categories are related by the composition:

\begin{equation}\label{Eq:DoldKan}
	 \Mod_k \stackrel{\tx{Forget}}{\longrightarrow} \Mod_{\mathbb Z} \stackrel{DK}{\simeq} \mathbf{Stab}(\mathbf{sAb}) \stackrel{\tx{Forget}}{\longrightarrow}  \mathbf{Stab}(\mathbf{sSet})=\mathbf{Sp}
\end{equation}
In the setting of operadic formal moduli problems, the following proposition justifies the terminology.

\begin{Prop}[{\cite[Lemma 2.15]{CCN20}}]\label{Prop_TangentSpaceAndTangentComplex}
	
A formal moduli problem $F\in \mathbf{FMP}_{\_P}$ has a unique pre-image in $\Mod_k$ under the map \eqref{Eq:DoldKan}, which we also denote by $\Tt_F$.	
 Moreover for any $n \geq 0$, we have an equivalence: 
	\[ \Map_{\Mod_k} \left( k[-n], \Tt_F \right) \simeq F(k[n])\]
		
\end{Prop}

As a motivation for part of the $k$-module structure, notice that the ground field $k$ acts naturally on $k[n]$ by multiplication, which implies that it acts on $\Tt_F$ by functoriality.
	
In the case of the formal spectrum of unital commutative rings, we saw that $\tx{Spf}(B)(k[\epsilon]) \simeq \vert \Tt_{B,x} \vert$ (and similarly when shifting $\epsilon$). Now the associated spectrum object to $\Tt_{B,x}$ (given by the composition \eqref{Eq:DoldKan}) coincides with the spectrum object given by $\tx{Spf}(B)(k[\epsilon_n])$ so that $\Tt_{B,x}$ is the representative of $\Tt_{\tx{Spf}(B)}$ given by Proposition \ref{Prop_TangentSpaceAndTangentComplex}. \\

It turns out that this also extends to the formal spectrum of a $\_P$-algebra $A$. Indeed, the tangent functor is given by the collection for each $n$ of the spaces:
 \[ \tx{Spf}(A)(k[n]) \simeq \Map_{\mathbf{Alg}_\_P}(A, k[n]) \]
 
 where $k[n]$ is endowed with the trivial $\_P$-algebra structure (induced by the augmentation $\_P \to I$). Taking the trivial algebra is a right adjoint functor (see Remark \ref{RQ_TrivalPAlgebra}) and we get: 
 \[ \tx{Spf}(A)(k[n]) \simeq \Map_{\mathbf{Alg}_\_P}(A, k[n]) \simeq \Map_{\Mod_k}(I \circ_\_P \tilde{A}, k[n]) \]
 where $\tilde{A}$ is a cofibrant resolution of $A$. In particular, we can take the cobar-bar resolution, $\tilde{A} := \Omega_\kappa \mathbf{B}_\kappa A$. Note that $I \circ_\_P \tilde{A}$ is the \emph{space of derived indecomposables}, which is the derived quotient by all products in $A$. We have the following observations:
 \begin{itemize}
 	\item There is an equivalence: \[ I \circ_\_P \Omega_\kappa \mathbf{B}_\kappa A \simeq \left( I \circ_\_P \_P \circ_I   \mathbf{B}_\kappa A, d_\Omega + d_{\mathbf{B}_\kappa A} \right)  \]
 	\item We have: 
 	\[ I \circ_\_P \_P \circ_I   \mathbf{B}_\kappa A \simeq  \mathbf{B}_\kappa A \]
 	\item Since $d_\Omega$ has positive weight on the free algebra $\_P \circ_I   \mathbf{B}_\kappa A$ and the functor $I \circ_\_P$ kills all the positive weight parts, the differential induced on $I \circ_\_P \_P \circ_I   \mathbf{B}_\kappa A$ is exactly $ d_{\mathbf{B}_\kappa A}$ and we have an equivalence: 
 	 \[I \circ_\_P \Omega_\kappa \mathbf{B}_\kappa A \simeq \mathbf{B}_\kappa A \]
 \end{itemize}
 
 Therefore we get:
 \[  \tx{Spf}(A)(k[n]) \simeq \Map_{\Mod_k}(\mathbf{B}_\kappa A, k[n])  \simeq \left| \mathbf{Der}^\_P\left(A, k[n] \right)  \right| \simeq \left|  \Tt_{A,x}^{\_P}[n]  \right|\]

%
%

\begin{Prop}
	For the deformation context $\dccomaug$, the tangent complex of a formal moduli problem $F$ is given by the $\Omega$-spectrum determined\footnote{An $\Omega$-spectrum object is completely determined by a collection $(x_i)_{i\in \Nn}$, \emph{indexed by $\Nn$}, such that there are weak equivalences $x_i \to \Omega x_{i+1}$. We recover the full spectrum object by applying $\Omega$ successively to the $x_i$.} by $\left( F(k \oplus k[n]) \right)_{n \geq 0}$.
	
	For the deformation context $\dcP$, the tangent functor of a formal moduli problem $F$ is given by the $\Omega$-spectrum defined by $\left( F( k[n]) \right)_{n \geq 0}$.
\end{Prop}
\begin{proof}
 Essentially, the properties of $F$ ensure that $$F(k \oplus k[n]) \rightarrow \Omega F(k \oplus k[n+1])$$ for $n \geq 0$ are equivalences making $F(k \oplus k[n])$ an $\Omega$-spectrum. See the discussion preceding \cite[Definition 2.14]{CCN20}.
\end{proof}

\begin{RQ}
	
Following Remark \ref{RQ_ObstructionFMP}, this implies that the obstruction to lifting deformations along square zero extensions is controlled by the cohomology of $\Tt_F$.
\end{RQ}

\begin{RQ}\label{RQ_HigherDeformationFormalSpectrumCommutative}
{Let $B\in \mathbf{Alg}_{\mathbf{Com}}^{k//k}$ and let  $x : \star \rightarrow \tx{Spec} (B)$ be the point corresponding to the augmentation.}
	Then we have: 
	
	\[ \Tt_x \tx{Spf} (B) := \left( \Map_{\mathbf{dSch}_{/\star}} \left(\mathbf{D}_{k,n}(k), \tx{Spec} (B)\right) \right)_{n \geq 0} \]
	
	and we recover the higher deformations of $\tx{Spec} (B)$ as discussed in \ref{RQ_HigherCotangentComplex} (see also \cite{Ve10}). In $\chk$ this corresponds to the tangent complex at the point $x$, $\Tt_{B,x}$.
\end{RQ}

\begin{RQ}
There is a classical version of a deformation functor of a variety $X$ taking values in sets, $F : \art \rightarrow \mathbf{Set}$,  given by deformations of $X$ over the artinian algebra modulo isomorphism and  similar to our more general version the natural notion of the tangent of this functor is given by $F(k[\epsilon])$. Schlessinger proves \cite{Sc68} (see also \cite[Theorem 2.4.1]{Se07}) that given $X$ a smooth proper variety over $k$, the tangent of the deformation functor $\tx{Def}_X$ is given by the $k$-vector space $H^1(X, T_X)$. 
\end{RQ}

\subsubsection{Koszul duality context}\
\label{Sec_koszul-duality-context}

\medskip

Going back again to the setting of a general deformation context $\_A$, we are looking to identify an $\infty$-category of algebraic objects $\_B$ which is equivalent to $\mathbf{FMP}_{\_A}$.

The formal spectrum $\tx{Spf} $ from  Example \ref{Ex_FMP_AffineMappingSpace} is a functorial way to construct formal moduli problems out of objects of $\_A$. 
Assuming the existence of the desired equivalence, we could construct a functor:
\[\mathfrak D \colon \_A^\mathrm{op} \stackrel{\tx{Spf} }{\longrightarrow}\mathbf{FMP}_{\_A} \stackrel{\sim}\to \_B\]

This functor should be interpreted as a ``weak duality'' functor, which is not an equivalence, but at least its restriction to Artinian objects should behave as an equivalence into a subcategory of $\_B$ of ``good'' objects.
Of course, we are interested in the case where $\_B$ are $\_P$-algebras and $\_A$ are $\_P^!$-algebras. In that case the reader can think of $\mathfrak D$ as the dual of the bar construction, see Warning \ref{war:dual is not dual}.

In this section, we introduce the notion of Koszul duality context as the appropriate axiomatic framework that enables us to obtain the desired equivalence, as shown in Theorem \ref{Th_FMPEquivalenceInKoszulDualityContext}.

\begin{Def}[Dual Deformation Context, {\cite[Definition 2.11]{CG18}}] \label{Def_DualDeformationContext}
	A pair $\left( \_B, \lbrace F_\alpha \rbrace_{\alpha \in T} \right)$ is called a \defi{dual deformation context } if $\_B$ is a presentable $\infty$-category and $F_\alpha \in \mathbf{Stab}(\_B^{\tx{op}})$.\\
	
	We say that an object (resp. morphism) of $\_B$ is \defi{good} if it is Artinian when considered in $\left( \_B^{\tx{op}}, \lbrace F_\alpha \rbrace_{\alpha \in T} \right)$\footnote{Here $\left( \_B^{\tx{op}}, \lbrace F_\alpha \rbrace_{\alpha \in T} \right)$ is generally not a deformation context since $ \_B^{\tx{op}}$ will in general not be presentable. However the definition of Artinian object still makes sense in $\_B^{\tx{op}}$.}. We denote by $\_B^{\tx{gd}}$ the full sub-category of good objects of $\_B$. 
\end{Def}  

\begin{Ex}\ \label{Ex_DualDeformationContexts}
	\begin{itemize}
		\item If $A \in \mathbf{Alg}_{\mathbf{Com}}$, then $\left( \Mod_A, \left( A[n] \right)_{n \in \Zz} \right)$ is a deformation context and since taking the opposite category exchanges the suspension and desuspension functor, we get the dual deformation context $\left(\Mod_A, \left( A[-n] \right)_{n \in \Zz} \right)$.
		
		If $A$ is bounded and concentrated in non-positive degrees, \cite[Lemma 2.16]{CG18} tells us that an element $M \in \Mod_A$ in this dual deformation context is good if it is perfect and cohomologically concentrated in positive degrees (see also \cite[Remark 2.17]{CG18}). 
		
		\item Dual deformation contexts can be transferred along left adjoints  (see \cite[Section 2.2]{CG18}). Using this on the free-forget adjunction: 
		\[ \begin{tikzcd}
			\Mod_k \arrow[r, shift left]& \arrow[l, shift left] \mathbf{Alg}_\_P
		\end{tikzcd}\]
		
		we get a dual deformation context: 
		\[ \left( \mathbf{Alg}_\_P, \lbrace \left( \_P \circ k[-n] \right)_{n \in \Zz}\rbrace \right)\]
	\end{itemize}
\end{Ex}

\begin{Def}[{\cite[Definition 2.18]{CG18}}] \label{Def_KoszulDualityContext}
	
	A \defi{weak Koszul duality context}  is the data of: 
	\begin{itemize}
		\item A deformation context $\left( \_A, \lbrace E_\alpha \rbrace_{\alpha \in T} \right)$ 
		\item A dual deformation context $\left( \_B, \lbrace F_\alpha \rbrace_{\alpha \in T} \right)$ 
		\item An adjunction: 
		
		\[ \begin{tikzcd}
			\G_D : \_A \arrow[r, shift left] & \arrow[l, shift left] \_B^{\tx{op}} : \G_D'
		\end{tikzcd}\]
	\end{itemize}
	
	such that for all $n \geq 0$, there is an equivalence $\Omega^{\infty - n} E_\alpha \simeq \G_D' \left( \Omega^{\infty -n} F_\alpha \right)$. \\
	
	It is called a \defi{Koszul Duality context} if the following holds:
	
	\begin{enumerate}
		\item For every object $B \in \_B^{\tx{gd}}$, the counit morphism $\G_D \G_D' B \rightarrow B$ is an equivalence. 
		\item For each $\alpha$, the functor:
		\[ \varTheta_\alpha :  \_B \rightarrow \mathbf{Sp} \]
		sending $B \in \_B$ to the spectrum object: \[\left( \mathbf{Hom}_{\_B}(\Omega^{\infty -n} E_\alpha, B) \right)_{n \in \Zz} \in \mathbf{Sp}\] under the functor is conservative and preserves sifted colimits. 
	\end{enumerate} 
\end{Def}

\begin{RQ}
	A weak Koszul duality context together with condition (1) gives us a weak deformation theory according to the terminology of Lurie, \cite[Definition 1.3.1]{Lu11}. A Koszul duality context is an example of a \defi{deformation theory} according to \cite[Definition 1.3.9]{Lu11}. 
\end{RQ}

\begin{RQ}
	The name ``Koszul duality context'' comes from the fact that Koszul duality for operads $(\_C, \_P)$ induces a weak Koszul duality context between $\_P$-algebras and $\_C ^\vee$-algebras (see Proposition \ref{Prop_WeakKoszulDualityFromOperadicKoszulDuality}). Under the assumptions of Theorem \ref{Th_CriteriaForOperadicFMP}, this is even a Koszul duality context.
\end{RQ}

\begin{Ex}
	Let $A \in \mathbf{Alg}_{\mathbf{Com}}$. We have a Koszul duality context given by the dualization:
	\[ \begin{tikzcd}
		(-)^\vee : \left(\Mod_A, \left( A[n] \right)_{n \in \Zz} \right) \arrow[r, shift left] & \arrow[l, shift left] \left( \Mod_A^\mathrm{op}, \left( A[-n] \right)_{n \in \Zz} \right) : (-)^\vee
	\end{tikzcd} \]
\end{Ex}

\begin{Prop}[{\cite[Proposition 2.22]{CG18}}]  \label{Prop_PropertiesKoszulDualityContext}
	Given a Koszul duality context (using the same notation as in Definition \ref{Def_KoszulDualityContext}), we have the following:
	
	\begin{itemize}
		\item $\G_D \left( \Omega^{\infty -n} E_\alpha \right) \simeq \Omega^{\infty -n} F_\alpha$ for all $n \geq 0$. 
		\item For every Artinian object $A \in \mathbf{Art}_\_A$, the unit map $A \rightarrow \G_D' \G_D A$ is an equivalence. 
		\item The adjunction $\G_D \dashv \G_D'$ induces an equivalence:
		\[ \begin{tikzcd}
			\G_D : \mathbf{Art}_\_A \arrow[r, shift left] & \arrow[l, shift left]\left( \_B^{\tx{gd}}\right)^{\tx{op}} : \G_D' 
		\end{tikzcd}\]
		
		\item If $M \in \mathbf{Art}_\_A$ and $f\colon A \rightarrow B$ is a small morphism in $\_A$ then $\G_D$ sends the pullback diagram: 
		\[ \begin{tikzcd}
			P \arrow[r] \arrow[d] & A \arrow[d, "f"] \\
			M \arrow[r] & B 
		\end{tikzcd}\]
		
		to a pullback diagram. 
		
	\end{itemize}
\end{Prop}

\begin{Cor}\label{Cor_KoszulDualityContextAnd FMP}
	Given a Koszul duality context, we can construct the following functor: 
	\[ \begin{tikzcd}
		\psi : \_B \arrow[r] &  \mathbf{Fun}\left(\_B^\tx{op}, \_S \right) \arrow[r, "\circ \G_D"] & \mathbf{Fun}\left( \_A, \_S \right)
	\end{tikzcd}  \]
		
	Then $\psi$ factors through $\mathbf{FMP}(\_A, \lbrace E_\alpha \rbrace)$ and we get a map:
	\[\Psi : \_B \rightarrow \mathbf{FMP}(\_A, \lbrace E_\alpha \rbrace)\]
\end{Cor}

It turns out that the definition of Koszul duality context ensures that the map $\Psi : \_B \rightarrow \mathbf{FMP}(\_A, \lbrace E_\alpha \rbrace)$ is in fact an equivalence.

\begin{Th}[{\cite[Theorem 2.33]{CG18}} or {\cite[Theorem 1.3.12]{Lu11}}] \label{Th_FMPEquivalenceInKoszulDualityContext}
	
	Given a Koszul duality context as above, the functor $\Psi$ is an equivalence:
	\[\Psi : \_B \rightarrow \mathbf{FMP}(\_A, \lbrace E_\alpha \rbrace)\] 
\end{Th}

\begin{Prop}[{\cite[Proposition 2.36]{CG18}}] \label{Prop_FMPEquivalenceAndTangent}
	We have the following commutative diagram: 
	\[ \begin{tikzcd}
		\_B \arrow[rr, "\Psi"] \arrow[rd, "\varTheta"'] & & \mathbf{FMP}\left( \_A, \lbrace E_\alpha \rbrace \right) \arrow[dl, "T"] \\
		& \mathbf{Sp} & 
	\end{tikzcd}\]
	
	where $\varTheta$ is the functor described in Definition \ref{Def_KoszulDualityContext} and $T$ is the tangent complex of $F$ (Definition \ref{Def_TangentFunctor}). 
\end{Prop}

\begin{RQ}\label{RQ_StructureOnTangentFunctor}
This diagram shows that $TF$ is equivalent to $\varTheta (\Tt_F)$ with $\Tt_F \in \_B$ so this tangent functor has more structure that being just a spectrum object. When we will work with the Koszul Duality context for operads (Definition \ref{Def_KoszulDualityContextForOperad}), with $\tx{Spf}(B)$, we get that $\Tt_F$ is exactly $\Tt_{B,x}^\_P$ the $\_P$-tangent complex together with a $\_P^!$ structure.
\end{RQ} 

\begin{Def}[{\cite[Definition 2.23]{CG18}}]  \label{Def_MorphismKoszulDualityContext}
	
	Given two (weak) Koszul duality contexts: 
	\[\begin{tikzcd}
		\G_D_1 : \left( \_A_1, \lbrace E_\alpha^1 \rbrace \right) \arrow[r, shift left] & \arrow[l, shift left] \left( \_B_1^{\tx{op}}, \lbrace F_\alpha^1 \rbrace \right): \G_D_1'  
	\end{tikzcd} \]
	\[\begin{tikzcd}
		\G_D_2 : \left( \_A_2, \lbrace E_\alpha^2 \rbrace \right) \arrow[r, shift left] & \arrow[l, shift left] \left( \_B_2^{\tx{op}}, \lbrace F_\alpha^2 \rbrace \right): \G_D_2'  
	\end{tikzcd} \]
	
	A \defi{morphism of (weak) Koszul duality context} is given by a diagram of adjunctions: 
	\[  \begin{tikzcd}
		\_A_1 \arrow[r, shift left,"\G_D_1"] \arrow[d, shift left, "L_1"]  & \arrow[l, shift left, "\G_D_1'"] \arrow[d, shift left, "L_2"]  \_B_1^{\tx{op}}  \\
		\_A_2 \arrow[u, shift left, "R_1"] \arrow[r, shift left,"\G_D_2"] & \arrow[l, shift left, "\G_D_1'"] \arrow[u, shift left, "R_2"] \_B_2^{\tx{op}}
	\end{tikzcd}\]
	
	such that the diagram of right adjoint commute: 
	\[  \begin{tikzcd}
		\_A_1  & \arrow[l, shift left, "\G_D_1'"]  \_B_1^{\tx{op}}\\
		\_A_2 \arrow[u, shift left, "R_1"]  & \arrow[l, shift left, "\G_D_1'"] \arrow[u, shift left, "R_2"]  \_B_2^{\tx{op}}
	\end{tikzcd}\]
	
	and we have equivalences $R_2 \left( \Omega^{\infty -n} F_\alpha^2 \right) \simeq F_\alpha^1$ for all $\alpha \in T$ and $n \geq 0$. 
\end{Def}

\begin{Prop}[{\cite[Proposition 2.39]{CG18}}] \label{Prop_NaturalityKoszulDualityContext}
	
	Given a morphism of Koszul duality contexts (following the notation of Definition \ref{Def_MorphismKoszulDualityContext}), we get a commutative diagram: 
	\[\begin{tikzcd}
		\_B_1 \arrow[dd, "L_2"] \arrow[r, "\Psi_1"] & \mathbf{FMP}\left( \_A_1, \lbrace E_\alpha^1 \rbrace \right)  \arrow[dd, "R_1"] \arrow[dr, "T"] & \\ 
		& & \mathbf{Sp} \\
		\_B_2 \arrow[r, "\Psi_2"] & \mathbf{FMP}\left( \_A_2, \lbrace E_\alpha^2 \rbrace \right)  \arrow[ur, "T"'] 
	\end{tikzcd}\]
	
\end{Prop}

\subsubsection{Koszul duality context from Koszul duality}\ \label{Sec_KoszulContextDualityAndKoszulDuality}

\medskip

In this section we will study examples of (weak) Koszul duality contexts that arise from operadic Koszul duality. The typical instance of such Koszul duality context can be exemplified by the following proposition which can be extracted from \cite[Theorem 2.3.1]{Lu11}, see \cite[Proposition 2.30]{CG18}.

\begin{Prop}
	
	We have a Koszul duality context: 
	\[ \begin{tikzcd}
		\G_D : \left( \mathbf{Alg}_\mathbf{Com}^{\tx{aug}}, (k \oplus k[n]) \right) \arrow[r, shift left] & \arrow[l, shift left] \left( \mathbf{Alg}_{\mathbf{Lie}}^{\tx{op}}, \left( \mathbf{Lie}(k[-n-1]) \right)_{n \in \Zz} \right) : \tx{CE}^\bullet
	\end{tikzcd} \]
	
	where the right adjoint is the ($\infty$-categorical) Chevalley--Eilenberg cochain complex,  see Example \ref{Ex:CE and HH cochain cx}.
\end{Prop}

While \cite{CG18} and \cite{Lu11} do not construct left adjoint $\mathfrak D$ explicitly (they deduce it from the Adjoint Functor Theorem), one can check that the functor $\mathfrak D$ coincides with the so-called Harrison complex, constructed as the functor (also denoted $\mathfrak D$) preceding Example \ref{Ex:CE and HH cochain cx}. These duality functors preserve quasi-isomorphisms and therefore define functors of $\infty$-categories, but notice that they do not arise from a Quillen adjunction.
The proof of this proposition makes use of the fact that we are in the conditions of Warning \ref{war:dual is not dual}.

It is clear that this result has more to do with the Koszul duality between the operads $\mathbf{Com}$ and $\mathbf{Lie}$ and with the operads themselves. We will show how construct a (weak) Koszul duality context from operadic Koszul duality, which as we will see is the main tool to prove Theorem \ref{Th_KoszulOperadicFMPEquivalence}.

Throughout this section, we will work with the deformation context $\dcP$ described in Example \ref{Ex_DeformationContexts}. {From now on, we will consider $\_P$ a $k$-augmented $k$-operad so that $\_P = \bar{\_P}\oplus I$.} Given a twisting morphism $\phi: \_C \dashrightarrow \_P$, we use the bar--cobar adjunction associated to $\phi$ to define the following functor:
\[ \begin{tikzcd}
	\G_D_\phi: \mathbf{Alg}_\_P \arrow[r, "\mathbf{B}_\phi"] & \mathbf{coAlg}_\_C \arrow[r, "(-)^\vee"] & \left( \mathbf{Alg}_{\_C^\vee} \right)^{\tx{op}}
\end{tikzcd}\] 

\begin{Prop}\label{Prop_WeakKoszulDualityFromOperadicKoszulDuality}
	If $\phi$ is a weakly Koszul twisting morphism (see Definition \ref{Def_KoszulTwistingMorphisms}), then we have the following:
	\begin{itemize}
		\item For any $A \in \mathbf{Alg}_\_P$, we have: 
		\[ \G_D_\phi (A) \simeq \Rr\tx{Der}_\_P \left( A, k \right) \]
		\item $\G_D_\phi$ preserves all colimits and therefore has a right adjoint: 
		
		\[ 	 \begin{tikzcd}
			\G_D_\phi : \mathbf{Alg}_\_P \arrow[r, shift left] & \arrow[l, shift left] \mathbf{Alg}_{\_C^\vee}^{\tx{op}} : \G_D_\phi '
		\end{tikzcd} \]
		
		\item The aforementioned adjunction defines a weak Koszul duality context with $\mathbf{Alg}_\_P$ seen with the deformation context $\dcP$ and with the dual deformation context $\left( \mathbf{Alg}_{\_C^\vee}, \left( \_C^\vee ( k[-n]) \right)_{n \in \Zz} \right) $ described in Example \ref{Ex_DualDeformationContexts}.
	\end{itemize}
\end{Prop}
\begin{proof}
	The first two parts can be extracted from \cite[Lemma 4.1]{CCN20}, while the last claim is a consequence of \cite[Corollary 4.7]{CCN20} where we apply the result to $V= k[n] = \Omega^{\infty -n}F$. Since $\_P(k[n]) = k[n]$ (the free $\_P$-algebra structure on $k[n]$ is the trivial one), we obtain that $\G_D_\phi (k[n]) = \G_D_\phi (\_P(k[n])) = \tx{triv}(k[n]) = k[n]$ which is the criterion for the adjunction to be a weak Koszul duality context.  
\end{proof}

\begin{Def}\label{Def_KoszulDualityContextForOperad} \label{Def_OperadicKoszulDualityContext}
	Given $\_P$ an augmented operad, we denote $\G_D(\_P) := \left(\mathbf{B}\_P\right)^\vee$ and $\pi : \mathbf{B}\_P \dashrightarrow \_P$ the universal twisting morphism (Definition \ref{Def_KoszulTwistingMorphisms}). Then the associated weak Koszul duality context is denoted by:
	\[ 	 \begin{tikzcd}
		\G_D : \mathbf{Alg}_\_P \arrow[r, shift left] & \arrow[l, shift left] \mathbf{Alg}_{\G_D(\_P)}^{\tx{op}} : \G_D '
	\end{tikzcd} \]
\end{Def}

We would like to know when this weak Koszul duality context is a Koszul duality context so that under those conditions, formal moduli problems over $\_P$ will be equivalent to $\mathbf{Alg}_{\G_D(\_P)}$. The necessary conditions are given by the following theorem:

\begin{Th} \label{Th_CriteriaForOperadicFMP}
	Let $\_P$ be an augmented operad concentrated in non-positive degrees which is \defi{splendid} (see Remark \ref{RQ_Splendide}). 	Then the following holds: 
	
	\begin{itemize}
		\item For any $A \in \mathbf{Art}_\_P$, the unit map $A \rightarrow \G_D' \G_D (A)$ is an equivalence. 
		\item $\G_D(k[n])$ is freely generated by $k[-n]$ for all $n \geq 0$ , that is: \[\G_D(k[n]) \simeq \G_D(\_P)(k[-n])\]
		\item $\G_D$ sends a pullback square in $\mathbf{Art}_\_P$: 
		\[ \begin{tikzcd}
			A \arrow[r] \arrow[d] & 0 \arrow[d] \\
			A' \arrow[r] & k[n]
		\end{tikzcd}\]
		
		to a pullback square in $\mathbf{Alg}_{\G_D(\_P)}$. 
		\item In general (without the assumptions of the theorem) the functor: 
		\[ \begin{tikzcd}
			\mathbf{Alg}_{\G_D(\_P)} \arrow[r,"\mathbf{MC}"]& \mathbf{FMP}_\_P \arrow[r, "T"] & \mathbf{Sp}
		\end{tikzcd}\]
		is conservative and preserves sifted colimits. 
	\end{itemize}
	
Then, we have a Koszul duality context:
	\[ 	 \begin{tikzcd}
		\G_D : \mathbf{Alg}_\_P \arrow[r, shift left] & \arrow[l, shift left] \mathbf{Alg}_{\G_D(\_P)}^{\tx{op}} : \G_D '
	\end{tikzcd} \]
	
	for the deformation context $\left( \mathbf{Alg}_\_P , \lbrace \left( k[n] \right)_{n\in \Zz} \rbrace \right)$ of Example \ref{Ex_DeformationContexts} and dual deformation context $\left( \mathbf{Alg}_{\G_D(\_P)} , \lbrace \left( \G_D(\_P)(k[-n]) \right)_{n\in \Zz} \rbrace \right)$ of Example \ref{Ex_DualDeformationContexts}. 
\end{Th}
\begin{proof}
	The first four items of the Theorem are proven in \cite[Theorem 5.1]{CCN20}. The last one is proven at the end of the proof of \cite[Theorem 4.18]{CCN20}. For the last part, we will show that these axioms imply the axioms of a Koszul duality context (Definition \ref{Def_KoszulDualityContext}). 
	
	Using that the unit of the adjunction is an equivalence on Artinian algebras (similarly to Warning \ref{war:dual is not dual}), and that $\G_D(k[n]) \simeq \G_D(\_P)(k[-n])$ we get an equivalence:
	\[ k[n] \simeq \G_D' \G_D (k[n]) \simeq \G_D' ( \G_D(\_P)(k[-n]))\]
	
	This shows that the adjunction is a weak Koszul duality context. 
	
More generally, we can show that there is an equivalence: 
	\[ 	 \begin{tikzcd}
		\G_D : \mathbf{Art}_\_P \arrow[r, shift left] & \arrow[l, shift left]  \left( \mathbf{Alg}_{\G_D(\_P)}^{\tx{gd}}\right)^{\tx{op}} : \G_D '
	\end{tikzcd} \]
	Therefore the counit restricted to good object is also an equivalence.
	Finally, the following functor is conservative and preserves sifted colimits: 
	\[ \begin{tikzcd}
		\mathbf{Alg}_{\G_D(\_P)} \arrow[r,"\mathbf{MC}"]& \mathbf{FMP}_\_P \arrow[r, "T"] & \mathbf{Sp}
	\end{tikzcd}\]
	
	but thanks to Proposition \ref{Prop_FMPEquivalenceAndTangent}, this functor is exactly the functor $\varTheta$ from Definition \ref{Def_KoszulDualityContext} that we want to show is conservative and preserves sifted colimits. 
\end{proof}

\begin{RQ}\label{RQ_Splendide}
	The condition being \defi{splendid} is a technical condition detailed \cite[Definition 3.37]{CCN20}, which imposes that a certain derived relative composite product
	 $$\_P(1) \circ_{\_P^{\geq 1}}^h \_P(1)$$
	 has  cohomology living only in more and more negative degrees.
	 
	 The main example to keep in mind is that any (non-necessarily binary) Koszul quadratic operad generated in non-positive degrees and in bounded arity is splendid. 

	In particular, $\mathbf{Com}$, $\mathbf{Ass}$ and $\mathbf{Lie}$ are splendid and satisfy all the conditions of Theorem \ref{Th_CriteriaForOperadicFMP}. 
\end{RQ}

\subsubsection{Main result} \label{Sec_main-result}\

\medskip

Putting together \cite[Theorem 4.18]{CCN20} and Theorem \ref{Th_CriteriaForOperadicFMP}, we obtain the following result: 

\begin{Th}\label{Th_OperadicFMPEquivalence}
	Suppose that $\_P$ is an augmented operad that is connective and splendid, then we have an equivalence: 
	
	\[ \Psi^\_P : \mathbf{Alg}_{\G_D(\_P)} \rightarrow \mathbf{FMP}_\_P\]
	
	that sends $\G_g \in \mathbf{Alg}_{\G_D(\_P)}$ to $\Map_{\G_D(\_P)} \left( \G_D (-), \G_g \right)$. Moreover the inverse functor sends a formal moduli problem $F$ to its tangent complex $\Tt_F$ endowed with some $\G_D(\_P)$-algebra structure, with $\G_D(\_P) = (\mathbf{B}\_P)^\vee \simeq \_P^! \lbrace -1 \rbrace$  (see Remark \ref{RQ_StructureOnTangentFunctor}).  
\end{Th}

\begin{RQ}
	The functor $\Psi^\_P$ of Theorem \ref{Th_OperadicFMPEquivalence} is, for $\_P=\Omega \_C$ and other mild conditions, given by some space of Maurer--Cartan elements.
	
	Let us sketch why this is the case, see \cite[Theorem 7.18]{CCN20} for more details. Under these conditions, given any $\G_g \in \mathbf{Alg}_{\_C^\vee}$ and  $A \in \mathbf{Art}_{\_P}$, we want to understand the space:	
	\[ \Map_{\mathbf{Alg}_{\_C^\vee}} \left( \G_D (A), \G_g \right) \]  
	
With some finiteness assumptions on $A$, we have a natural isomorphism $\Omega_{\alpha^\dagger}(A^\vee) \rightarrow \G_D_{\alpha}(A)$, obtained from the twisting morphisms $\alpha : \_C \dashrightarrow \_P$ and $\alpha^\dagger : \mathbf{B}(\_C^\vee) \dashrightarrow \_C^\vee$.
	
	There is a Rosetta Stone for algebras similar to the case of operads \ref{Prop_RosettaStone}, where we can identify maps from a cobar construction with Maurer--Cartan elements of a certain $L_\infty$-algebra \cite[Theorem 7.1]{WierstraAlgHop}: 
		
	\[ \Map_{\mathbf{Alg}_{\_C^\vee}} \left( \Omega_{\alpha^\dagger}(A^\vee), \G_g \right)  \simeq \tx{MC} \left( \iHom_{k\text{-}\Mod} \left( A^\vee, \G_g \otimes \Omega(\Delta^\bullet )\right)\right) \]
	
	Since $A$ is Artinian, as complexes, the $L_\infty$-algebra above is isomorphic to $A\otimes_k \G_g \otimes \Omega(\Delta^\bullet)$. Expanding on the discussion at the beginning of Section	\ref{Sec_Operadic Formal Moduli Problems and Koszul Duality}, there is a (shifted) $L_\infty$-algebra structure on $\_P\otimes \_C^\vee$-algebras given by the twisting morphism (see \cite[Lemma 7.15]{CCN20}) and these two $L_\infty$-structures on $A\otimes_k \G_g \otimes \Omega(\Delta^\bullet)$ coincide. We conclude that
	\[ \Map_{\mathbf{Alg}_{\_C^\vee}} \left( \Omega_{\alpha^\dagger}(A^\vee), \G_g \right)  \simeq \tx{MC} \left(  A \otimes_k \G_g \otimes \Omega(\Delta^\bullet )\right) \simeq \mathbf{Del}\left(  A \otimes_k \G_g \right) \] 
\end{RQ}
 
 When $\_P$ is a Koszul operad there is an equivalence between $\G_D (\_P)$ and $\_P^! \lbrace -1 \rbrace$. This implies the result:

\KoszulOperadicFMPEquivalence

\begin{proof}	
	We have weak equivalences of operads $\Omega \_P^\antishriek \rightarrow 	\_P$ and $\G_D (\_P) \rightarrow \_P^! \lbrace -1 \rbrace$. 
	
	$\_P$ satisfies the condition of Theorem \ref{Th_OperadicFMPEquivalence}, and we get a zig-zag of equivalences: 
	
	\[ \begin{tikzcd}
		\mathbf{FMP}_\_P \arrow[r, "\Tt"] & \mathbf{Alg}_{\G_D(\_P)} & \arrow[l, "\simeq"'] \mathbf{Alg}_{\_P^! \lbrace -1 \rbrace } \arrow[r, "V\mapsto V {[ -1 ]}"] & \mathbf{Alg}_{\_P^!}
	\end{tikzcd}\] 
\end{proof}

\begin{Ex}
	\begin{itemize}
		\item For $\_P= \mathbf{Com}$ we get back Theorem \ref{Th_LuriePridham}: 
		
		\[ \mathbf{FMP} \overset{\sim}{\to} \mathbf{Alg}_{\mathbf{Lie}} \]
		
		\item For $\_P = \mathbf{Ass}$, we get the result of Lurie, in \cite{Lu11}: 
		
		\[ 	 \mathbf{FMP}_{\mathbf{Ass}} \overset{\sim}{\to} \mathbf{Alg}_{\mathbf{Ass}}\]
		
		\item For $\_P = \mathbf{Lie}$ we obtain a new equivalence: 
		
		\[ \mathbf{FMP}_{\mathbf{Lie}} \overset{\sim}{\to} \mathbf{Alg}_{\mathbf{Com}}\]
		
		\item For the permutative operad we get: 
		\[ \mathbf{FMP}_{\mathbf{Perm}} \overset{\sim}{\to} \mathbf{Alg}_{\mathbf{preLie}}\]
	\end{itemize}
\end{Ex}

The reader can find many examples treated in detail in Section 3 of \cite{CCN20}.

\subsubsection{Naturality of the main result}\label{sec:naturality-of-the-main-result}\

\medskip

A final important point to make concerns the naturality Theorem \ref{Th_KoszulOperadicFMPEquivalence} with respect to the operads involved. The main example of deformation problems studied in Section \ref{Sec:2} were encoded by the deformation complex $\mathfrak g_{P_\infty,A}^\phi$. If we consider deformations of the trivial algebra structure on $A$, i.e. when $\phi=0$, the deformation complex is in fact a pre-Lie algebra. 
The associated permutative formal moduli problem yields by restriction a commutative moduli problem and the naturality of Theorem \ref{Th_KoszulOperadicFMPEquivalence} asserts that the two commutative formal moduli problems are the same. 

Indeed, the following diagram of restrictions commutes:
	\[ \begin{tikzcd}
	\mathbf{FMP}_{\mathbf{Com}} & \arrow[l, "\sim"] \mathbf{Alg}_{\mathbf{Lie}} \\	
	\mathbf{FMP}_{\mathbf{Perm}} \arrow[u] & \arrow[l, "\sim"] \mathbf{Alg}_{\mathbf{preLie}} \arrow[u]\\
	\mathbf{FMP}_{\mathbf{Ass}} \arrow[u] & \arrow[l, "\sim"] \mathbf{Alg}_{\mathbf{Ass}} \arrow[u]\\
	\mathbf{FMP}_{\mathbf{Lie}} \arrow[u] & \arrow[l, "\sim"] \mathbf{Alg}_{\mathbf{Com}}\arrow[u]\\
\end{tikzcd}\]

More generally, following Theorem \ref{Prop_NaturalityKoszulDualityContext}, we can formulate a notion of naturality in the Koszul duality contexts obtained from universal twisting morphisms. We recall that Koszul twisting morphism form a category (Definition \ref{Def_KoszulTwistingMorphisms}).

\begin{Prop}[{\cite[Proposition 6.5]{CCN20}}]
	The construction in Section \ref{Sec_KoszulContextDualityAndKoszulDuality} defines a natural transformation\footnote{$\tx{Cat}_\infty^L$ denotes the $\infty$-category of $\infty$-categories with left adjoint $\infty$-functor as morphisms.}:
	
	\[ \begin{tikzcd}[column sep=huge]
		\mathbf{Kos}
		\arrow[bend left=50]{r}[name=U,label=above:$\mathbf{Alg}$]{}
		\arrow[bend right=50]{r}[name=D,label=below:$\mathbf{Alg}^{\tx{dual}}$]{} &
		\tx{Cat}_\infty^L
		\arrow[shorten <=10pt,shorten >=10pt,Rightarrow,to path={(U) -- node[label=right:$\G_D$] {} (D)}]{}
	\end{tikzcd}  \]

	such that $\G_D (\alpha) =  \G_D_\alpha : \mathbf{Alg}_\_P \rightarrow \mathbf{Alg}_{\_C}^{\tx{op}}$
\end{Prop}

If we restrict to using the universal twisting morphism as in Definition \ref{Def_OperadicKoszulDualityContext}, we get the natural transformation (\cite[Corollary 6.7]{CCN20}): 

\[ \begin{tikzcd}[column sep=huge]
	\Op_k^{\tx{aug}}
	\arrow[bend left=50]{r}[name=U,label=above:$\mathbf{Alg}$]{}
	\arrow[bend right=50]{r}[name=D,label=below:$\mathbf{Alg}^{\tx{dual}}$]{} &
	\tx{Cat}_\infty^L
	\arrow[shorten <=10pt,shorten >=10pt,Rightarrow,to path={(U) -- node[label=right:$\G_D$] {} (D)}]{}
\end{tikzcd}  \]

Note that the functor $\G_D$ gives us weak Koszul duality context (and even Koszul duality context for nice enough operads). Moreover, Proposition \ref{Prop_NaturalityKoszulDualityContext} gives us some naturality with respect to morphisms of Koszul duality context. In the end, we obtain the following result: 

\begin{Prop}[{\cite[Proposition 6.10]{CCN20}}] \label{Prop_NaturalityFMPEquivalences}
	Let $\Op_k^+$ denote the category of splendid connective $k$-operads. Then we have a natural equivalence\footnote{$\tx{Pr}^R$ denote the $\infty$-category of presentable $\infty$-categories with morphisms given by right adjoint functors.}: 
	\[ \begin{tikzcd}[column sep=huge]
		\Op_k^{+}
		\arrow[bend left=50]{r}[name=U,label=above:$\mathbf{Alg}_{\G_D}$]{}
		\arrow[bend right=50]{r}[name=D,label=below:$\mathbf{FMP}$]{} &
		\tx{Pr}^R
		\arrow[shorten <=10pt,shorten >=10pt,Rightarrow,to path={(U) -- node[label=right:$\Psi$] {} (D)}]{}
	\end{tikzcd}  \]
	
	that sends $f : \_P \rightarrow \_Q$ to the commutative square of right adjoint: 
	\[ \begin{tikzcd}
		\mathbf{Alg}_{\G_D (\_P)} \arrow[r] \arrow[d, "\G_D(f)^*"] & \mathbf{FMP}_\_P \arrow[d, "(f^*)^*"] \\
		\mathbf{Alg}_{\G_D (\_Q)} \arrow[r]  & \mathbf{FMP}_\_Q
	\end{tikzcd}\]
\end{Prop}

\subsubsection{Some vague comments on positive characteristic}\label{sec:some-vague-comments-on-positive-characteristic}\
	
	\medskip
	
The whole text we worked over a field of characteristic zero with good reason. Many of the homological results do not generally hold in positive characteristic, for instance the bar-cobar ``resolution'' can be constructed but is not a quasi-isomorphism and the Maurer--Cartan equation and gauge groups cannot even be defined. 

On the other hand, whenever the concerned operads are non-symmetric or at least the symmetric groups act freely ($\Sigma$-free) there is some hope to retain some homotopy invariance. For instance, the Maurer--Cartan equation for $A_\infty$-algebras (which are indeed equivalent to associative algebras) reads
$$0=d\mu + m\star m + m\star m \star m+\dots$$
which makes sense in every characteristic and can therefore be studied with operadic methods, see \cite{Wierstra2018lie}. 

One approach to operadic deformation theory in positive characteristic involves therefore considering $\Sigma$-free resolutions of operads. 
A particularly important example is the case of the commutative operad, whose $\Sigma$-free resolutions are called $E_\infty$-operads. Notice that since the commutative operad is the unit for the tensor product of operads (if we are careful with assumptions on arities $\leq 1$), given such an $E_\infty$-operad $\_E$, $\_P \otimes \_E$ is a $\Sigma$-free resolution of $\_P$. It is therefore important to have access to small models $\_E$, but there is no similar Koszul duality for algebraic operads able to produce minimal models in positive characteristic. The best known such model is the Barratt--Eccles operads \cite{fresseberger}, but others exist \cite{fresseberger,Anibal, brantner2021pd}.

One of the goals of Dehling--Vallette \cite{DehlingVallette} is to study the homotopy theory of operads in positive characteristic by including $\Sigma$-action on the definition of the operad of operads, as in Section \ref{sec:a-comment-on-operads-as-algebras}.

An analogue of the Lurie--Pridham Theorem \ref{Th_LuriePridham} is also available in positive characteristic: Depending on whether the formal moduli problems are parameterized by $E_\infty$-rings or simplicial commutative rings, Brantner--Mathew \cite{brantner2019deformation} have shown that these are equivalent to the $\infty$-categories of spectral or derived partition Lie algebras, respectively.

\bibliographystyle{alpha}
\bibliography{biblio}

\makeatletter
\providecommand\@dotsep{5}
\makeatother

\section*{Index of Notations}

\underline{Operads:}

\begin{itemize}[label = --]
	 \item $\symseq$: Category of symmetric sequences (Definition \ref{Def_SymmetricSequences}).

	 \item $\otimes_{\tx{Day}}$: Day convolution (Definition \ref{Def_DayConvolution}).
	 \item $\circ$: Monoidal product on $\symseq$ defining the composition of symmetric sequences (Definition \ref{def:circ of SS}).
	 
	 \item $\hat{\circ}$: Complete composition of symmetric sequences (Proposition \ref{Prop_FormulaDualMonoidalStructure}).
	 \item $\bar{\circ}$: Composition of symmetric sequences for cooperads (Equation \ref{Eq_cocompostionsymmetricsequences}).
	 
	 \item $\mathbf{Op}, \ \mathbf{coOp}, \ \mathbf{coOp}^{\tx{conil}}$: Category of operads (Definition \ref{Def_Operad}), cooperads (Definition \ref{Def_Cooperads}) and conilpotent cooperads (Definition \ref{Def_ConilpotentCooperad}). 	 
	 
	 \item $\mathbf{Op}_k^{+}$: Full sub-category of $\mathbf{Op}$ given by splendid connective $k$-operads
	  \item $\mathbf{cOp}$: Category of colored operads (Section  \ref{sec:a-comment-on-operads-as-algebras}).
	 
	 \item $\circ_i$: Partial composition of an operad (Definition \ref{def:partial comp}).
	 \item $\Delta_i$: Partial cocomposition of a cooperad.
	 
	 \item $\T, \ \Tc$: Free and cofree operad functor (see Constructions \ref{Cons_Operad} and \ref{Cons_Cooperads} respectively).
	 
	 \item $\_P(E,R)$: Operad generated by $E \in \symseq$ with relations generated by $R \subset \T E$ (Definition \ref{def:partial comp}).
	 \item $\_C(E,R)$: Cooperad cogenerated by $E \in \symseq$ with corelations generated by $R \subset \Tc E$ (Construction \ref{Cons_Cooperads}).
	 
	 \item $I$: Unit (co)operad concentrated in arity $1$ and with $I(1)=k$ (Example \ref{Ex_Operad}). 
	 \item $\mathbf{Ass}, \ \mathbf{Com}, \ \mathbf{Lie}$: Associative, commutative and Lie operads (Example \ref{Ex_Operad}).
	 
	 \item $\mathbf{Perm}$ and $\mathbf{preLie}$: Permutative and pre-Lie operads (Definition \ref{Def_PermAndPreLie}).
	 \item $\mathbf{End}_V, \mathbf{coEnd}_V $: Endomorphism and coendomorphism operads (Example \ref{Ex_Operad}).
	  
	 \item $\mathbf{coEnd}_N^M$: Symmetric sequence of map from $M$ to $N^{\otimes \bullet}$ (Definition \ref{Def_InternalHomSymSeq}).
	\item $\_O$: Colored operad encoding symmetric operads (Section  \ref{sec:a-comment-on-operads-as-algebras}).

	 \item $\_P^\vee$, $\_C^\vee$: $\_P^\vee$ is the arity-wise dual of the operad $\_P$ and $\_C^\vee$ is the operad given by the arity-wise dual of the cooperad $\_C$ (Construction \ref{Cons_Cooperads}).
	 
	 \item $\mathbf{End}_N^M$: Symmetric sequence of map from $M^{\otimes \bullet}$ to $N$ (Definition \ref{Def_InternalHomSymSeq}).

	 \item $\tx{Conv}(\_C, \_P)$: Convolution Operad (Definition \ref{Def_ConvolutionOperad}).

	 \item $\mathbf{B} \dashv \Omega$: Bar-cobar adjunction (Definition \ref{Def_BarCobarConstruction}).
		
	 \item $\_Q \lbrace -n \rbrace$: $\_Q \otimes \mathbf{End}_{k[n]}$.

	 \item $\_P^!$: Koszul Dual Operad (Definition \ref{Def_KoszulDualOperad}).
	 \item $\_P^\antishriek$: Koszul dual cooperad (Definition \ref{Def_KoszulDualOperad}).
	 \item $\_P_\infty$: Minimal cofibrant resolution of $\_P$, $\_P_\infty := \Omega \_P^\antishriek$ (Convention \ref{Not_PinftyOperad}).
	 	
	 \item $\mathbf{Tw}$: Category of all twisting morphisms (Definition \ref{Def_TwistingMorphisms}).
		\item $\mathbf{Koszul}$: Full sub-category of $\mathbf{Tw}$ given by Koszul morphisms (Definition \ref{Def_KoszulTwistingMorphisms}).
	 
	 \item $A_\infty, \ L_\infty, \ C_\infty$: Notation for $\mathbf{Ass}_\infty$, $\mathbf{Lie}_\infty$ and $\mathbf{Com}_\infty$ (Convention \ref{Not_PinftyOperad}).

\end{itemize}

\underline{Other:}

\begin{itemize}[label = --]
	\item $k$: Field of characteristic $0$.
	\item $\Sigma_n$: Group of permutations of $n$-elements.
	\item $\tx{Der}_A(B,M)$, $\underline{\tx{Der}}_A(B,M)$ and $\Rr\tx{Der}_A(B,M)$: $B$-Module, differential graded module and derived module of $M$-valued, $A$-linear derivations on $B$ respectively.  
\item $k[\![t]\!] $: Algebra of formal power series in $t$ over $k$.
\item $ \sim $: Denote an equivalence relation. In general in a model category, $A \sim B$ if there is a zig-zag of weak-equivalences between them. 
\item $ \G_M_A $: Maximal ideal of a local ring $A$.
\item $ B \oplus - $: Square zero extension functor sending a $B$-module to the trivial square zero extension $B \oplus M$ (Definition \ref{Def_SquareZeroExtension}).
\item $\Omega_{A/B}^1 $: $A$-module of Kähler differential on $A$ relative to $B$ (Definition \ref{Def_CotangentComplex}).
\item $\Ll_{A/B} $: Cotangent complex of $A$ relative to $B$ (Definition \ref{Def_CotangentComplex}). 
\end{itemize}

\underline{Categories:}\

\begin{itemize}[label = --]
	
\item $\chk$: Category of cochain complexes over $k$.
\item $\Mod_R$: Category of cochain complexes over a commutative ring $R$. 
\item $\iHom$: Internal $\Hom$ for categories enriched over themselves.
\item $\_C^{A/}$, $\_C^{/A}$ and $\_C^{A//B}$: Category of objects under $A$, objects over $A$ and objects both over $B$ and under $A$ respectively.
\item $\tx{Fun}(\_C, \_D)$: Category (or $\infty$-category) of functors from $\_C$ to $\_D$. 
\item $\_C^{\tx{iso}}$: Given a category $\_C$, $\_C^{\tx{iso}}$ denotes the sub-category of $\_C$ with same objects and morphism being the isomorphism in $\_C$.
\item $\Nn^\sim$: Category with objects the sets  $\{1,\dots,n\}$ for all $n \in \Nn$ and morphisms given by the permutations of these sets.
\item $\Nn_S$: Category with objects being $k$-tuples of elements on $S$ for all $k \in \Nn$ and morphisms given by all permutations of these tuples, $(a_1, \cdots, a_k) \to (a_{\sigma_1}, \cdots , a_{\sigma_k})$. 
\item $\int_{c \in \_C}$ and $\int^{c \in \_C}$: End and coend functors sending a functor $F : \_C^{\tx{op}}\times \_C \to \_D$ to and object $\int^{c\in F}F(c,c) \in \_D$ (see \cite{Coend}). 
\item $\tx{Mon}(\_C)$: Categorie of monoid in a monoidal category $\_C$.
\item $C-\tx{Mod}$: Categorie of $C$-module for $C \in \tx{Mon}(\_C)$.
\item $N$: Denotes all the nerve functors, including simplicial nerves and nerve of a category or groupoid.

\item $\_S$: $\infty$-categories of spaces.
\item $\mathbf{sSet}$: Category of simplicial sets.
\item $\vert -   \vert$: Geometric realisation functor $\vert-\vert: \mathbf{sSet} \to \mathbf{Top}$.
\item $h\_C$: Homotopy category of a model or $\infty$-category (Section \ref{sec:model-categories-and-homotopical-algebra}).
\item $\mathbf{Stab}(\_C)$: Stable $\infty$-category associated to an $\infty$-category $\_C$.
\item $\Omega^{\infty-n}$: Functor from $\mathbf{Stab}(\_C)$ to $\_C$ sending a spectrum object $(A_k)_{k\in \Zz}$ to $A_n$.
\item $\mathbf{Sp}$: Category of spectra, $\mathbf{Sp} = \mathbf{Stab}(\mathbf{Top})$. 
\item $\underline{\tx{Def}}_\phi (\G_R)$: Groupoid of deformation of an algebraic structure $\phi: \_P \to \mathbf{End}_A$ along the local algebra $\G_R$. 
\end{itemize}

\underline{Deformation Theory:}

\begin{itemize}[label = --]
	\item $\tx{MC}$: Functor sending a Lie algebra $\G_g$ to its set of Maurer--Cartan elements, $\tx{MC(\G_g)}$ (Definition \ref{Def_MC}).
	\item $k[\epsilon]$: $k$-algebra of dual numbers. \[ k[\epsilon] := \faktor{k[x]}{(x^2)}\]
	\item $\Omega[\Delta^\bullet]$: Simplicial commutative $k$-algebra of differential forms on simplices (Definition \ref{Def_HigherDeligneGroupoid}).
	\item $\mathbf{Del}(\G_g)$: Deligne $\infty$-groupoid associated to a Lie algebra $\G_g$ (Definition \ref{Def_HigherDeligneGroupoid}).  
	\item $\underline{\tx{Del}}(\G_g)$: Deligne groupoid associated to a Lie algebra $\G_g$ (Definition \ref{Def_DeligneGroupoid}).  
	\item $	\underline{\t_{\tx{Def}}}_B^\_P(A)$: Groupoid of geometric deformation of a $\_P$-algebra structure on $B$ along $A$ (Definition \ref{Def_GroupoidDeformationPAlgebra}).
		\item $\mathbf{Def}_B^\_P(A)$: $\infty$-groupoid of geometric deformation of a $\_P$-algebra structure on $B$ along $A$ (Definition \ref{Def_HigerDeformationGroupoidGeometric}).
		\item $\dccomaug$: Deformation context for augmented commutative $k$-algebras (Example \ref{Ex_DeformationContexts}).
			\item $\dcP$: Augmentation context for $\_P$-algebras (Example \ref{Ex_DeformationContexts}).
			\item $\art$, $\small$ and $\smallp$: Category of classical commutative local Artinian algebra, differential graded commutative local Artinian algebra and Artinian algebra in $\smallp$ (Definition \ref{Def_SmallComAlgebra} and \ref{Def_SmallObjectsAndMorphisms}).
			 \item $\mathbf{FMP}$, $\mathbf{FMP}_\_A$ and $\mathbf{FMP}_\_P$: Category of formal moduli problem associated to commutative deformations, the deformation context $\_A$ and the deformation context $\dcP$ respectively (Definition \ref{Def_FMP} and Example \ref{Ex_CategoryOperadicFMP}).
			 \item $\tx{Spf}$: Formal spectrum functor (Example \ref{Ex_FMP_AffineMappingSpace}).
\end{itemize}

\underline{Algebraic Structures:}

\begin{itemize}[label = --]
	\item $\mathbf{Alg}_{\mathbf{Com}\otimes_k A}(\Mod_A)$: Category of $A$-linear commutative algebra in $\Mod_A$.
	\item $\mathbf{CE}$: Chevalley-Eilenberg functor. 
	\item $\art$: Category of local Artinian $k$-algebras (Definition \ref{Def_SmallComAlgebra}).
	\item $\mathbf{Alg}_\_P$, $\mathbf{coAlg}_\_P$ and $\mathbf{coAlg}_\_P^{\tx{conil}}$: Category of all $\_P$-algebras, $\_P$-coalgebras and conilpotent $\_P$-coalgebras respectively (Definition \ref{Def_AlgebraCoalgebraOperad}).
	\item $\mathbf{Alg}_\_C$, $\mathbf{coAlg}_\_C$ and $\mathbf{coAlg}_\_C^{\tx{conil}}$: Category of all $\_C$-algebras, $\_C$-coalgebras and conilpotent $\_C$-coalgebras respectively (Definition \ref{Def_AlgebraCoalgebraCoOperad}).
	\item $\_P(-)$ and $\_C(-)$: Free $\_P$-algebra functor and cofree conilpotent $\_C$-coalgebra functor respectively (Proposition \ref{Prop_FreeAlgebraOverAnOperad}).
	\item $\tx{Conv}(\_C, \_P)$: Convolution operad (Definition \ref{Def_ConvolutionOperad}).
	\item $\tx{Tot}(\_P)$: Totalisation pre-Lie algebra (Definition \ref{Def_TotatlizationOfAnOperad}).
	\item $(-)^\sharp$: Forgetful functor forgetting the $\_P$ (or $\_C$) (co)algebra structure.
	\item $f_{!} \dashv f^*$: Free-forget adjunction associated to a map of operad $f: \_P \to \_Q$ (Proposition \ref{Prop_NaturalityAlgebraAdjunction}). In particular $f_!$ is the free $\_Q$-algebra functor relative to $\_P$.
	\item $\G_U$: Universal enveloping algebra functor (Section \ref{sec:Lie digression}).
	\item $\Sym$: Symmetric algebra functor. 
	
	\item $\Omega_\alpha \dashv \mathbf{B}_\alpha$: Algebraic cobar--bar adjunction associated to a twisting morphism $\alpha$ (Definition \ref{Def_BarCobarAlgebra}).
	\item $\G_D$: Koszul dual algebra (Definition \ref{Def_OperadicKoszulDualityContext}).
	\item $f: A \rightsquigarrow B$: $\infty$-morphism from $A$ to $B$   (Definition \ref{Def_InftyMorphisms}).
	\item $F^\bullet \G_g$: Filtration on a Lie algebra $\G_g$ (Definition \ref{Def_Filtration}). 
	\item $\widehat{-}$: Completion functor for filtered algebras (Definition \ref{Def_Filtration}). 
	\item $\tx{BCH}$: The BCH product satisfying $ e^X e^Y = e^{\tx{BCH}(X,Y)}$ (Proposition \ref{Prop_BCH}).
	\item $ \G_g_{\_P_\infty,A}$: Convolution Lie Algebra of the cochain complex $A$ (Definition \ref{Def_ConvolutionPreLieAlgebra}). 
	\item $ \G_g^\phi$: Twisted Lie algebra by a Maurer--Cartan element $\phi$ of $\G_g$ (Proposition \ref{Prop_DeformationLieAlgebraMaurerCartan}).
	\item $ \G_g_{\_P_\infty, A}^\phi$: Deformation complex of $A$ with $\_P_\infty$ structure $\phi \in \tx{MC}(\G_g_{\_P_\infty,A})$  (Definition \ref{Def_DeformationComplex}).

\end{itemize}

\end{document}